\documentclass[10.5pt]{article}
\usepackage{mathrsfs}
\usepackage{bbm}
\usepackage[leqno]{amsmath}
\usepackage{amsfonts}
\usepackage{graphicx}
\usepackage{amsmath}
\usepackage{amssymb}
\usepackage{latexsym}
\usepackage{amsmath, amsfonts,amssymb, amsthm, euscript,makeidx,color,mathrsfs}
\usepackage{indentfirst}
\usepackage{footmisc}
\usepackage{tikz}
\usetikzlibrary{shapes, arrows, positioning, matrix} 
\usepackage[polish,english]{babel}  

\AtBeginDocument{\selectlanguage{english}}

\def\sqr#1#2{{\vcenter{\vbox{\hrule height.#2pt
				\hbox{\vrule width.#2pt height#1pt \kern#1pt \vrule width.#2pt}
				\hrule height.#2pt}}}}

\def\3n{\negthinspace \negthinspace \negthinspace }
\def\2n{\negthinspace \negthinspace }
\def\1n{\negthinspace }

\def\dbA{\mathbb{A}}
\def\dbB{\mathbb{B}}

\def\dbE{\mathbb{E}}
\def\dbF{\mathbb{F}}

\def\dbH{\mathbb{H}}

\def\dbN{\mathbb{N}}

\def\dbP{\mathbb{P}}
\def\dbQ{\mathbb{Q}}
\def\dbR{\mathbb{R}}
\def\dbS{\mathbb{S}}

\def\dbV{\mathbb{V}}
\def\dbW{\mathbb{W}}
\def\dbX{\mathbb{X}}
\def\dbY{\mathbb{Y}}
\def\dbZ{\mathbb{Z}}

\def\sA{\mathscr{A}}

\def\sD{\mathscr{D}}

\def\sF{\mathscr{F}}

\def\sS{\mathscr{S}}

\def\sU{\mathscr{U}}
\def\sV{\mathscr{V}}

\def\sY{\mathscr{Y}}
\def\sZ{\mathscr{Z}}

\def\={\buildrel \triangle \over =}

\def\ds{\displaystyle}

\def\ns{\noalign{\ss}}
\def\a{\alpha}
\def\b{\beta}
\def\g{\gamma}
\def\d{\delta}
\def\e{\varepsilon}

\def\k{\kappa}

\def\n{\nu}
\def\si{\sigma}
\def\t{\tau}
\def\f{\varphi}

\def\o{\omega}

\def\i{\infty}
\def\D{\Delta}

\def\F{\Phi}
\def\O{\Omega}

\def\cA{{\cal A}}
\def\cB{{\cal B}}
\def\cC{{\cal C}}

\def\cF{{\cal F}}

\def\cL{{\cal L}}
\def\cM{{\cal M}}
\def\cN{{\cal N}}

\def\cP{{\cal P}}

\def\cS{{\cal S}}

\def\cU{{\cal U}}
\def\cV{{\cal V}}

\def\cX{{\cal X}}
\def\cY{{\cal Y}}
\def\cZ{{\cal Z}}

\def\no{\noindent}

\def\ss{\smallskip}
\def\ms{\medskip}

\def\q{\quad}
\def\qq{\qquad}

\def\limsup{\mathop{\overline{\rm lim}}}

\def\lan{\mathop{\langle}}
\def\ran{\mathop{\rangle}}

\def\argmin{\mathop{\rm argmin}}

\def\essinf{\mathop{\rm essinf}}

\def\wt{\widetilde}

\def\cd{\cdot}
\def\cds{\cdots}

\def\tr{\hbox{\rm tr$\,$}}

\def\les{\leqslant}
\def\ges{\geqslant}

\def\({\Big (}
\def\){\Big )}
\def\[{\Big[}
\def\]{\Big]}
\def\bde{\begin{definition}}
	\def\ede{\end{definition}}
	\def\bt{\begin{theorem}}
		\def\et{\end{theorem}}
	\def\bc{\begin{corollary}}
		\def\ec{\end{corollary}}
	\def\bl{\begin{lemma}}
		\def\el{\end{lemma}}
	\def\bp{\begin{proposition}}
		\def\ep{\end{proposition}}
	\def\bas{\begin{assumption}}
		\def\eas{\end{assumption}}
	\def\br{\begin{remark}}
		\def\er{\end{remark}}
	\def\ba{\begin{array}}
		\def\ea{\end{array}}
	\def\ed{\end{document}}

\def\square#1{\vbox{\hrule\hbox{\vrule height#1%
			\kern#1\vrule}\hrule}}
\def\rectangle#1#2{\vbox{\hrule\hbox{\vrule height#1%
			\kern#2\vrule}\hrule}}

\font\tenbb=msbm10 \font\sevenbb=msbm7 \font\fivebb=msbm5

\newfam\bbfam
\scriptscriptfont\bbfam=\fivebb \textfont\bbfam=\tenbb
\scriptfont\bbfam=\sevenbb

\newtheorem{lemma}{Lemma}[section]
\newtheorem{remark}{Remark}[section]

\newtheorem{theorem}{Theorem}[section]
\newtheorem{corollary}{Corollary}[section]

\newtheorem{definition}{Definition}[section]
\newtheorem{proposition}{Proposition}[section]
\newtheorem{assumption}{Assumption}[section]

\allowdisplaybreaks

\makeatletter

\@addtoreset{equation}{section}
\makeatother

\begin{document}
	\title{\bf  Reflected stochastic recursive control problems  with jumps: dynamic programming  and stochastic verification theorems\footnote{This work is supported  by NSF of Jilin Province for Outstanding Young Talents (No. 20230101365JC),  the National Key R\&D Program of China (No. 2023YFA1009002), the NSF of
		P. R. China (No.  12371443), and the Changbai Talent Program of Jilin
		Province}}

\author{ Lu Liu \footnote{School of Mathematics and Statistics, Northeast Normal University, Changchun 130024, China; email: {\tt liulu@nenu.edu.cn} }\and   Qingmeng Wei\footnote{Corresponding author. School of Mathematics and Statistics, Northeast Normal University, Changchun 130024, China; email: {\tt weiqm100@nenu.edu.cn}}
}

\maketitle

\begin{abstract}
	This paper mainly investigates reflected stochastic recursive control problems governed by jump-diffusion dynamics. The system's state evolution is described  by a stochastic differential equation driven by both Brownian motion and Poisson random measures, while the recursive cost functional is formulated via  the solution process 
	$Y$ of a reflected backward stochastic differential equation  driven by the same dual stochastic sources.
	By establishing the  dynamic programming principle,  
	we provide the probabilistic interpretation of an obstacle problem for partial integro-differential equations of Hamilton-Jacobi-Bellman type in the viscosity solution sense  through our control problem's value function. 
	Furthermore,  the value function is proved to inherit  the semi-concavity and joint Lipschitz continuity in state and time coordinates,   which play key roles in deriving stochastic verification theorems of control problem within the  framework of viscosity solutions.
	We remark that some restrictions  in previous study are eliminated, such as  
	the  frozen of the reflected processes  in time and state, and   the independence of the driver   from diffusion variables.
\end{abstract}
\bf Keywords. \rm  RBSDE with jumps; dynamic programming principle; obstacle problem; partial integro-differential equations; viscosity solution; probabilistic interpretation;   stochastic verification theorem. 
\ms

\bf AMS Mathematics subject classification. \rm 93E20; 35D40; 49K45

\section{Introduction}\label{Sec_In}

Reflected backward stochastic differential equations (RBSDEs), as a pivotal extension of classical backward stochastic differential equations (BSDEs), were first systematically studied by El Karoui et al.  \cite{El-KPPQ}.
Distinguished from standard BSDEs, RBSDE  incorporates an additional non-decreasing process that enforces the solution path $Y$ to maintain above (or below) a prescribed barrier in a kind of minimal way. 
This  structure innovation has promoted RBSDEs as a powerful mathematical tool in various fields, such as  options pricing (\cite{KPQ-1997,H-2006}), mixed stochastic  game problems (\cite{H-2006,KH-2002}),  obstacle problems for partial differential equations (PDEs)  (\cite{KPQ-1997,BL-2009,WY-2008,BL-2011}) and so on.

In response to the inherent discontinuities observed in financial markets (such as abrupt price jumps, liquidity shocks), the theoretical framework of RBSDEs has been extended to incorporate both Brownian motion and Poisson random measures, termed as RBSDEs with jumps. Such an advancement  is particularly critical     not only in mathematical theory but also in applications. 
The work by Hamad\`{e}ne,  Ouknine \cite{HO-2003} established the existence and uniqueness of solutions for such systems    when the obstacle process has only inaccessible jumps. 
Essaky \cite{E-2008} further generalized these results to the case when the obstacle process has predictable jumps    through a penalization method.
In the aspect of application,   some researches has been done to link  these stochastic systems to obstacle problems for partial integro-differential equations (PIDEs). 
Matoussi et al. \cite{M-S-Z-2015} presented a probabilistic interpretation for   weak Sobolev solutions  for semilinear parabolic PIDEs with obstacles. 
In the viscosity solution sense,  Sylla \cite{Sylla-2019} gave a probabilistic interpretation for  PIDEs via the solution of RBSDEs with jumps. 
Zhang, Liu \cite{Zhang-Liu-2019}  employed  the solutions to reflected forward-backward stochastic differential equations (FBSDEs) to establish   a nonlinear Feynman-Kac representation for mild supersolutions of PIDEs. 
However, despite these contributions, the study of obstacle problems  specifically for  PIDEs of HJB type  rising from stochastic control problems   remains largely unexplored.

Motivated by this, we investigate  a class of stochastic recursive optimal control problems with jump-diffusion dynamics.
Precisely, the system's state evolution is governed by a stochastic differential equation driven by both Brownian motion and Poisson random measures (SDE with jumps), while the recursive cost functional is defined  through the solution process  $Y$ of a reflected backward stochastic differential equation  driven by the same dual stochastic sources (RBSDE with jumps).  
The approach of dynamic programming principle is employed here  to link  the value function with  the unique viscosity solution to an obstacle problem for PIDE   of HJB type.
Throughout the study,  a sequence of penalized control problems is introduced to approximate the  original control problem.
Subsequently, more properties of the value function are explored. We  demonstrate the value function  inherits  the semi-concavity and joint Lipschitz continuity in state and time coordinates under some additional conditions.
These properties are indispensable and key, which is revealed in   the research of   stochastic verification theorems of our control problem, especially in the framework of viscosity solutions.    
The following is about  the main contributions and novelties of this work.

%

The first is the study of the regular properties (including the semi-concavity and joint Lipschitz continuity in state and time coordinates) of the value function.
The study of the semi-concavity    is very complex and  needs the meticulous deduction.
Seeing more clearly  from  the study of semi-concavity, we are  involved in the comparison and estimates of three different time-state configurations of the value function, 
while being encountered with  the    difficulties  caused by jump diffusion.
The conventional transformation technique from \cite{BHL-2012, YZ-1999}, effective in pure Brownian motion  settings, is insufficient when we confront with   the jump diffusion. 
Our resolution lies in an ingenious adaptation of  a kind of time-stretching  transformations of the jump noise introduced by \cite{K-2001}, which we  call as  the Kulik's transformation, refer to \cite{K-2006, J-2013}, etc.
However, this powerful tool  comes with a price, which  necessitates a technical assumption $\n(E)<\infty $ in Section 4.
We point out such condition is not necessary    in other parts.

Furthermore, this work makes a significant  breakthrough in relaxing the critical constraints imposed by the existing studies for reflected BSDEs.  While the previous studies like \cite{BHL-2012} required  either the reflected process to be frozen in time and state  or the driver coefficient to be independent from diffusion variables (see Conditions (H6), (H7) therein).
By recasting the problem into a sequence of penalized control problems, we focus on the study of the semi-concavity  of the penalized value functions, which allows  the  residual terms in the inequality characterization of the semi-concavity.  
Crucially, these residual  terms   exhibit asymptotically  decay as the penalized system converges to the original control problem, thereby preserving essential regularity properties without  relying on restrictive assumptions.
The appearance and rigorous characterization of the residual terms within the semi-concavity inequalities   are the key innovations in the research.

The paper is organized as follows. Section 2 introduces the Wiener-Poisson space and presents the  preliminaries about   RBSDEs with jumps. Section 3 formulates  stochastic control problems involving reflected  FBSDEs  with jumps. Moreover, by employing a dynamic programming approach,  the value function  is connected   with the corresponding obstacle problem of  PIDE of HJB  type. Section 4 investigates more regularity properties of the value functions, focusing on the semi-concavity and joint Lipschitz continuity in $(t,x)$. The concluding Section 5 develops the stochastic verification theorems in  both classical solution and viscosity solution frameworks.

\section{Preliminaries}\label{Sec-Pre}
 
	Let the triple  $(\Omega,\mathcal{F},\dbP)$ be  the completed product of the two probability spaces $(\O_1,\cF_1,\dbP_1)$ and $(\O_2,\cF_2,\dbP_2)$, i.e., $\Omega :=\Omega_1\times\Omega_2 $, $\mathcal{F} :=\mathcal{F}_1\otimes\mathcal{F}_2 $, $\dbP :=\dbP_1\otimes \dbP_2 $ with $\mathcal{F}$ being completed with respect to $\dbP$. The detailed information of spaces  $(\O_1,\cF_1,\dbP_1)$ and $(\O_2,\cF_2,\dbP_2)$ are as follows.

 \begin{itemize}
   \item $(\Omega_1, \mathcal{F}_1, \dbP_1)$ is a classical Wiener space, i.e., 
	$\Omega_1 = C_0(\mathbb{R};\mathbb{R}^d)$, $\mathcal{F}_1$ is the completed Borel $\sigma$-field on $\Omega_1$, 
	$\dbP_1$ is the Wiener measure.  Under $\dbP_1$, the canonical processes $B_s(\omega)=\omega_1   ( s)$ and $B_{-s}(\omega)=\omega_1 ( -s)$, $s\in\mathbb{R}_{+}$, $\omega_1 \in \Omega_1$  are two independent $d$-dimensional Brownian motions.  
	Denote $\mathbb{F}^B=\{ \mathcal{F}_s^B \}_{s \geqslant 0}$  is 
the natural  filtration generated by the  Brownian motion $B(\cd)$. i.e., 
	$\mathcal{F}_s^B :=\sigma \big\{ B_r, r\les s \big\}\vee\mathcal{N}_{\dbP_1},$ 
	with $\mathcal{N}_{\dbP_1}$ being the collection of $\dbP_1$-null sets. 

   \item $(\Omega_2,\mathcal{F}_2, \dbP_2)$ is a Poisson space. Precisely, 
 $\Omega_2$ is  the set of all point functions $  p : D_{  p}\rightarrow E$, with $D_{ p}$ being a countable subset of $\mathbb{R}$ and $E :=\mathbb{R}^l \backslash \{ 0 \}$ being  equipped with its Borel $\sigma$-field $\mathcal{B}(E)$.  We identify the point function $p$ with $N( p,\cdot)$, where   $N $ is the counting measure defined on $\mathbb R \times E$, i.e.,
	$$
	N(p, (s,t] \times A) := \sharp\big\{r \in D_p \cap (s,t] \mid  p(r) \in A  \big\},\ A\in \mathcal{B}(E),\  s,t\in\mathbb{R},\  s<t.
		  $$
	Here $\sharp$ represents    the cardinal number of the set.
  $\mathcal{F}_2$ denotes the smallest $\sigma$-field on $\Omega_2$ such that the coordinate mapping $  p \rightarrow N \big(   p,(s,t] \times A \big)$, $A \in \mathcal{B}(E)$, $s,t\in\mathbb{R}$, $s<t$ is measurable.
 $\dbP_2$ is the probability  on  $(\Omega_2,\mathcal{F}_2)$ such that the coordinate measure $N( p,\mathrm  dt\mathrm de)$ is a Poisson random measure with the compensator $\hat N(\mathrm dt\mathrm de) :=\mathrm dt\nu(\mathrm de)$, where $\n$  is supposed to be a  $\si$-finite measure on $(E,\mathcal{B}(E))$ satisfying  $\displaystyle \int_E(1\wedge|e|^2) \nu(\mathrm de)<\infty$.
  Then, for any $A\in\mathcal{B}(E)$ with $\nu(A)<\infty$,  the process $\big\{ \tilde N  \big((s,t]\times A \big) \big\}_{t\ges s} :=\big\{ (N-\hat N ) \big( (s,t]\times A \big) \big\}_{t\ges s}$  is  a martingale.
By setting  
$$ \dot{\mathcal{F}}_t^N  :=\sigma \Big\{N \big( (s,r]\times A \big), -\infty <s\leqslant r\leqslant t, A \in\mathcal{B}(E) \Big\},\ t\geqslant 0, $$
we get the filtration $\dbF^N :=( \mathcal{F}_t^N  )_{t\geqslant0}$ with $\displaystyle\mathcal{F}_t^N  :=\( \bigcap_{r>t} \dot{\mathcal{F}}_r^N \)\vee \mathcal{N}_{\dbP_2}$. 

 \end{itemize}
  Based on the above, the filtration on $(\O,\cF,\dbP)$ is introduced as  $\mathbb{F}:=\{\mathcal{F}_t\}_{t\geqslant0}$ with $\mathcal{F}_t :=\mathcal{F}_t^B  \otimes \mathcal{F}_t^N $ augmented by all $\dbP$-null sets.

	For any $t \in [0,T]$, $p\ges 1$ and  Euclidean space $ \dbR^k$ ($k \geqslant 1$), we introduce the following  spaces,
	$$
	\begin{array}{ll}
		\ns\ds L^2_{\mathcal F_t}( \Omega ; \mathbb R^k) := \Big\{ \xi:  \Omega \rightarrow \mathbb R^k  \mid \xi \mbox{ is}\ \mathcal F_t\mbox{-measurable},\  \mathbb E |\xi|^2 <  \infty \Big\}; \\
		\ns\ds \mathcal{S}^2_{ \mathbb F }( t,T; \mathbb R ^k ) :=\Big\{ \phi: \Omega \times [t,T] \rightarrow \mathbb R^k  \mid  \phi(\cdot)\  \mbox{is }  \mathbb F\mbox{-adapted},\  {\rm c\grave{a}dl\grave{a}g},
		\mbox{ and}\ \mathbb{E} \Big[ \sup_{s \in[t,T]} | \phi_s |^2 \Big] < \infty \Big\};\\
		\ns\ds \mathcal{M}_\mathbb F ^2 (t,T;\mathbb R ^k) :=\Big\{ \phi:\Omega\times[t,T]\rightarrow \mathbb {R}^k\ \mid
		\phi(\cdot)\   \mbox{is }  \mathbb F \mbox{-progressively}\   \mbox{measurable, and }   \mathbb{E}\Big [ \int_t^T |\phi(s)|^2 \mathrm ds \Big] <  \infty  \Big\};\qq\qq\qq\qq \\
	\end{array}
	$$
		$$
		\begin{array}{ll}
		\ns\ds  \mathcal K _ \mathbb F ^2 (t,T;\mathbb  R  )  := \Big\{  K:\Omega\times[t,T]\times E \rightarrow \mathbb  R  \mid  K(\cdot)\ \mbox{is}\ \mathcal{P}_{t,T} \footnotemark
		\otimes \mathcal{B}(E) \mbox{-measurable},\\
		\ns\ds \hskip 7.5 cm
		\mbox { and}\ 
		\Vert K(\cd)\Vert_{\mathcal{K}_\dbF^2(t,T;\mathbb{R})}^2=\mathbb{E}\Big[\int_t^T\int_{E}|K_s(e)|^2\nu(\mathrm de)\mathrm  ds\Big]< \infty \Big\};\\
		\ns\ds  \mathcal{L}_\n ^p (E;\mathbb{R})  := \Big\{ K:E\rightarrow \mathbb  R   \mid  K(\cdot)\ \mbox{is}\	\mathcal{B}(E)\mbox{-measurable, and }  \Vert K(\cd) \Vert_{\nu,p}^p = \int_{E} |K(e)|^p\nu(\mathrm de)  < \infty \Big\};\\
		%
		\ns\ds\mathcal A_\mathbb F ^2(t,T;\mathbb R) := \Big\{ \phi:\Omega\times
		[t,T]\rightarrow\mathbb R\mid \phi (\cdot) \mbox{ is } \mathbb
		F\mbox{-adapted, }{\rm c\grave{a}dl\grave{a}g}  \mbox{ and increasing, }\phi(t)=0,\ \mathbb E\big[|\phi(T)|^2 \big]< \infty  \Big\}.\\
		%
	\end{array}
	$$
	\footnotetext{$\mathcal{P}_{t,T}$ denotes the $\sigma$-field of $\dbF$-predictable subsets of $\O\times [t,T]$.}
		For any $t\in[0,T]$, we put
	$$
	\sS_ \mathbb F ^2[t,T] :=          \mathcal{S}_{ \mathbb F }^2(t,T;\mathbb{R})          \times         \mathcal{M}_{ \mathbb F }^2(t,T;\mathbb{R}^d)  
	\times\mathcal{K}_\mathbb F ^2(t,T;\mathbb{R})\times \mathcal A_\mathbb F ^2(t,T;\mathbb R).  
	$$

	\subsection{Reflected backward stochastic differential equations with jumps}
 In this part, we recall some known results about 	RBSDEs with jumps. Consider
	\begin{equation}\label{rbsdep}\left\{
		\begin{array}{ll}
			\ns\ds\!\!\!\! 			{\rm(\romannumeral1)} \  	(Y,Z,V,A)\in \sS_ \mathbb F ^2[0,T];\\
			\ns\ds \!\!\!\! 		{\rm(\romannumeral2)} \ 
			Y_s = \xi + \int_s^T  \!\!  f\big(  r, Y_r,  Z_r, \int_E l(e)V_r(e)\nu(\mathrm de)  \big) \mathrm dr
			- (A_{T} - A_{s})  \\
			\ns\ds \!\!\!\! \hskip 1.42 cm  - \int_s^T   Z_r   \mathrm dB_r - \int_s^T \!\int_E   V_r (e)   \tilde N(\mathrm dr,\mathrm de), \q s\in[0,T];\\
			\ns\ds\!\!\!\! 			{\rm(\romannumeral3)} \  Y_s \leqslant S_s,\q   \mbox{a.e.}\ s \in [0,T];\\
			\ns\ds\!\!\!\! 			{\rm(\romannumeral4)} \  \int_0^T(S_{s-}-Y_{s-})  \mathrm d A_s=0.
		\end{array}\right.
	\end{equation}
In the above,  $f$, $\xi$ and $S$ is the driver coefficient, the terminal condition and the obstacle term  of \eqref{rbsdep}, resp. We sometimes use the triple  $(f,\xi,S)$ to represent \eqref{rbsdep}. The   condition on $(f,\xi,S)$ is given as follows.

 \begin{description}   \item[(A$_1$)]     (\romannumeral1)   $f:\Omega \times [0, T] \times \mathbb{R} \times \mathbb{R}^{d} \times \dbR  \rightarrow \mathbb{R}$ is $\cP_{0,T}$-measurable for every fixed $(y,z,v)\in \dbR\times\dbR^d\times \dbR $, and    $f(\cdot, 0,0,0) \in \mathcal{M}_\mathbb F ^2 (0,T;\mathbb R) $; 
  
	 (\romannumeral2)    $f$ is  Lipschitz continuous  in $(y,z,v )\in\dbR \times\dbR^d \times  \dbR  $,  uniformly with respect to $(\o,s)\in\O\times[0,T]$, $\dbP$-a.s.;
	 

 (\romannumeral3)   $\xi\in L^2_{\mathcal F_T}( \Omega ; \mathbb R)$,   $S $ is a real-valued   $\mathbb F$-progressively measurable, ${\rm c\grave{a}dl\grave{a}g}$ stochastic process satisfying
	$ 
	 \mathbb{E} \Big[ \sup\limits_{s \in[0,T]}  |S_s|  ^2 \Big] < \infty,
 $ 	and
   $\xi \leqslant S_T,$ $\dbP $-a.s.;

	(\romannumeral4)  there exists some constant $\k>0$ such that  $0 \les l(  e) \les \k(1\land |e|)$;
	
	\item[(A$_2$)] $v\to f(t,y,z,v)$ is non-decreasing, for all $(t,y,z)\in[0,T]\times \mathbb{R} \times \mathbb{R}^{d}$.
 \end{description}

Now we present some results about the wellposedness of RBSDE with jumps, which can be refer  to \cite{E-2008}. Notice that  the obstacle  in \eqref{rbsdep} is   the upper one, so   we   need to adapt the lower   obstacle case studied in \cite{E-2008} to our case.  Such a transformation between the two cases is natural and easy, so that we shall not repeat it here and  present the results directly. 
 
First, for each $n\in\dbN$, we introduce the following BSDE with jumps,
	%
	\begin{equation}\label{BSDEP}
	\ba{ll}
\ns\ds 	Y_s^n = \xi  +  \int_s^T   f \big(  r,  Y_r^n,  Z_r^n,  \int_E    l(e) V_r^n(e)  \nu(\mathrm de)     \big )  
	  \mathrm dr -(A_T^n-A_s^n)  -  \int_s^T    Z_r^n   \mathrm dB_r   \\
	  \ns\ds\hskip0.95cm -  \int_s^T    \int_E    V_r^n (e)   \tilde N(\mathrm dr,\mathrm de) ,\q s \in [0,T],
	\ea\end{equation} 
where $\ds A_s^n:=\int_0^s   n\big(  S_r - Y_r^n \big)^-   \mathrm dr $.
	From \cite{E-2008}, we know  \eqref{BSDEP} is the penalized equation  of RBSDE with jumps \eqref{rbsdep}.
  	Clearly, for each $n\in\dbN$, the condition {\bf (A$_1$)} guarantees the uniquely existence of $(Y ^n,  Z^n,   V ^n) \in \mathcal{S}^2_{ \mathbb F }( t,T; \mathbb R  ) \times  \mathcal{M}^2_{ \mathbb F }( t,T; \mathbb R ^d ) \times \mathcal K _\dbF^2 (t,T;\mathbb  R  ) $ satisfying \eqref{BSDEP}, refer to   \cite{Tang-1994, LP-2009}.
	With the help of  Theorems 4.2 and  5.1 in \cite{E-2008}, we have the following approximation  result from  \eqref{BSDEP} to \eqref{rbsdep},  and thereby the wellposedness of \eqref{rbsdep}.   
	\begin{lemma}\label{well}\sl
		Assume {\rm \textbf{(A$_1$)}} and  {\rm \textbf{(A$_2$)}} hold. Then the following two results hold true.

{\rm(i)}
 The sequence $\{(Y^n, Z^n, V^n, A^n)\}_{n\ges 0}$ has a limit $(Y,Z,V,A)$ such that  $Y^n $  converges  decreasingly  to $ Y \in \mathcal{S}^2_{ \mathbb F }( 0,T; \mathbb R  ) $, and $(Z,V,A)$ is the weak limit  of  $(Z^n,V^n,A^n)$ in $\mathcal{M}^2_{ \mathbb F }( 0,T; \mathbb R ^d ) \times   \mathcal K _\dbF^2 (0,T;\mathbb  R  )\times \cA^2_\dbF(0,T;\dbR) $.

{\rm(ii)} The limit  $(Y,Z,V,A)\in\sS_\dbF^2[0,T]$ is the unique solution of  RBSDE  with jumps \eqref{rbsdep}.

\end{lemma}
	%

	Next, we present the  estimate and  the continuous dependence of the solutions of RBSDEs with jumps. 

	\begin{lemma}
		\sl Assume {\rm \textbf{(A$_1$)}}  holds.
	
{\rm(i)}	Let $ (Y , Z ,V , K ) \in\sS_\dbF^2[0,T]$ be the solution of RBSDE with jumps \eqref{rbsdep}, then there exists a constant $C>0$ such that
$$
			\begin{array}{ll}
				\ns\ds\!\!\! \displaystyle \mathbb{E}^{\cF_t}\Big[
				\underset{s\in  [t,T]}{\sup} |Y_s|^2  +\int_t^T( |Z_s|^2+\|V_s(\cd)\|_{\n,2}^2) \mathrm ds
				    + \underset{s\in  [t,T]}{\sup}|A_s|^2    \Big]  \\ 
				\ns\ds\!\!\! \leqslant C \mathbb{E}^{\cF_t}\Big[ |\xi|^2   +     \(\int_t^T  |f(s, 0, 0, 0)|^{2}  \mathrm ds\) 
				+\underset{s\in  [t,T]}{\sup}  |S_s|^2  \Big],\quad \dbP\mbox{-a.s.} 
			\end{array}
$$

{\rm(ii)}  For any given   data   $( g',\xi',  S')$  satisfying  {\rm \textbf{(A$_1$)}}, denote  $(Y', Z',V', A')$  by  the solution of \eqref{rbsdep} with $(g', \xi', S')$.
Then there exists some constant $C>0$ such that
$$
			\begin{array}{ll}
				\ns\ds\!\!\!    \mathbb	{E}^{\cF_t}\big[ \underset{s\in  [t,T]}{\sup} 
				|\widehat Y_{s}|^{2}     +\int_t^T(|\widehat Z_{s}|^{2}  +\|\widehat  V_s(\cd)\|_{\n,2}^2)\mathrm ds   + |\widehat A_{T}-\widehat A_{t}|^{2}  \big] \\
				\ns\ds\!\!\!     \leqslant  C \mathbb	{E}^{\cF_t}\Big[|\widehat  \xi|^{2}+   \Big(\int_t^T \big|\widehat f \big(s, Y_{s}, Z_{s}, \int_E l(e)V_s(e)\nu(\mathrm de) \big) \big|^2  \mathrm ds\Big)  \Big]
				+C\Big(	\mathbb	{E}^{\cF_t}\big[\underset{s\in  [t,T]}{\sup} |\widehat S_{s}|^{2} \big]\Big)^{\frac{1}{2}} \Psi_{t, T}^{\frac{1}{2}},\quad \dbP\mbox{-a.s.},\\
			\end{array}
$$
		where  $\widehat \xi:=\xi-\xi'$, $(\widehat Y, \widehat Z,\widehat V,\widehat A):=(Y-Y',Z-Z',V-V',A-A')$, $\widehat S:=S-S'$, $\widehat f:=f-f'$ and
		\begin{equation*}
				\Psi_{t, T}:= \dbE^{\cF_t}\Big[|\xi|^{2}+|\xi'|^{2}  
				+ \int_t^T |f(s, 0,0,0)|^2 \mathrm ds     +  \int_t^T |f'(s, 0,0,0)|^2 \mathrm ds  
				 +\underset{s\in  [t,T]}{\sup} |S_{s}|^{2} 
				 + \underset{s\in  [t,T]}{\sup}|S_{s}'|^{2}  \Big].
		\end{equation*}
	\end{lemma}
  Note that, (i) can be referred to Proposition 2.2  in \cite{E-2008}. Moreover, the detail of (ii) is similar to (i), so we omit here.
For the later study, we also need the  comparison theorem of \eqref{rbsdep}. 

 \begin{lemma}\label{com}\sl

    	(Comparison Theorem)
     	Let  $( f_i,\xi_i,S^i)$, $ i=1,2 $  satisfy {\rm \textbf{(A$_1$)}} and {\rm \textbf{(A$_2$)}}, and   $(Y^i,Z^i,V^i, A^i),\ i=1,2$ be the unique solution  of     RBSDE  with jumps \eqref{rbsdep}  associated  with $(f_i,\xi_i,S^i)$,   $i=1,2$, resp.
     	 Then $Y^1_s \leqslant Y^2_s,$ $\dbP$-a.s.,    $s\in [0,T]$, whenever
    	$\xi_1 \leqslant \xi_2$,
    	$f_1(s, y,z,v) \leqslant f_2(s,y,z,v) $,  $S^1_s \les S^2_s$,
    	 $s \in [0,T]$,  $(y,z,v) \in  \mathbb{R} \times \mathbb{R}^{d} \times \dbR $,   
    	 $ \dbP$-a.s.  
    	
    \end{lemma}

     Note that the comparison theorem in   \cite{E-2008} (Theorem 5.2) is with the same obstacle process. Our Lemma \ref{com}  generalizes it to the case with different barriers by observing
     $$   f_1(s,y,z,v)-n( S_s^1-y)^-\les f_2(s,y,z,v)-n( S_s^2-y)^- ,\  (s,y,z,v) \in [0,T]\times \mathbb{R} \times \mathbb{R}^{d} \times \dbR,\ \dbP\mbox{-a.s.}  
     $$

\section{Reflected  stochastic recursive control problems   with jumps}\label{RSRCP}

\subsection{Formulation of the control problem}

Let $U\subset\dbR^m$ be compact. For $t\in[0,T],$   the admissible control $u(\cd)$ on $[t,T]$ is introduced   as the $U$-valued, $\dbF$-progressively measurable stochastic process. We denote by $\cU_{t,T}$  all the admissible controls on $[t,T]$. 
 
The state process    is described by the following controlled  SDE  with jumps,
	\begin{equation} \label{state}
		\left\{ 
		\begin{array}{ll}
			\ns\ds\!\!\!\!  \mathrm dX_s=b( s, X_s,u_s )\mathrm  ds + \sigma (s, X_s,u_s)  \mathrm dB_s
			+\int_{E} \gamma (s,    X_{s-}, u_s,e) \tilde N (\mathrm ds,\mathrm de),\q s\in[t,T],\\
			\ns\ds\!\!\!\!  X_t=x_t,

		\end{array}\right.
	\end{equation}
where $(t,x_t )\in[0,T]\times L^2_{\mathcal F_t}( \Omega ; \mathbb R^n)$ is the initial pair, $u(\cd)\in \cU_{t,T}$, 
and the above involved coefficients  
	$ b:[0,T] \times \mathbb R^n \times U  \rightarrow  \mathbb R^n,$ $
		\sigma: [0,T] \times \mathbb R^n \times U  \rightarrow  \mathbb R^{n\times d}, $ $
		\gamma : [0,T] \times \mathbb R^n \times U\times E   \rightarrow \mathbb R^n 
	 $
are assumed to satisfy the following condition.

 \begin{description}
   \item[(H$_1$)] (\romannumeral1)   For all $ x  \in \dbR^n,\ e\in E $, $b(\cd,x,\cd)$, $\si(\cd,x,\cd)$ and $\gamma(\cd,x,\cd,e)$  are continuous in $(r,u)\in[0,T]\times U$;
	
  (\romannumeral2)  $b,\sigma$ are  Lipschitz continuous  in  $x\in\dbR^n$, uniformly with respect to $(r,u)\in [0,T]\times U$. And 
	 there exists a map  
$\ell(\cd) \in \cL_\nu^2(E;\dbR)$
	   such that for all
	$(r,u)\in [0,T]\times U$,   $x_0,x_1\in\dbR^n$, $e\in E$,
	$$  
	|\gamma(r,x_0,u,e)-\gamma(r,x_1,u,e)| \les \ell(e)|x_0-x_1|.
	$$
  
 \end{description}

From the classical theory of SDE with jumps, we get  the following results (refer to \cite{Barles-1997,BHL-2011,LP-2009}). 

\bl\label{Le-SDE}
\sl Under  {\bf (H$_1$)}, for any $(t,x_t )\in[0,T]\times L^2_{\mathcal F_t}( \Omega ; \mathbb R^n)$, $u(\cd)\in \cU_{t,T}$, the SDE with jumps \eqref{state}  admits the unique  solution  $X \equiv X^{t,x_t;u} \in \cS^2_{\mathbb F}(t,T;\mathbb R^n)$.  
Moreover, there exists  some constant  $C>0$ such that, for   $t\in[0,T]$, $x_t  ,x_t^\prime\in   L^2_{\mathcal F_t}( \Omega ; \mathbb R^n) $ and  $u(\cd)\in \cU_{t,T}$,  the following estimates hold, $\dbP$-a.s.,
$$\ba{ll} 
		\ns\ds\!\!\!  	{\rm (i)} \ \dbE^{\cF_t} \[  \sup \limits_{r\in[t,T]}|X^{t,x_t;u}_r|^2 \] \les C (1+ |x_t|^2 ),\\
 		\ns\ds\!\!\!		{\rm (ii)} \ \dbE ^{\cF_t} \[  |X^{t,x_t;u}_s  -x_t|^2 \] \les C(s-t)(1+|x_t|^2),\\
		\ns\ds\!\!\!  	{\rm (iii)} \ \dbE ^{\cF_t}\[\sup\limits_{r\in[t,T]}|X^{t,x_t ;u}_r -X^{t,x_t^\prime;u}_r |^2\]\les C  |x_t-x_t^\prime|^2.
	\ea
$$
\el

  Next,  with the state $X^{t,x_t;u} $ from \eqref{state}, we consider the following  RBSDE  with jumps,
	\begin{equation}\label{RBSDEP}
		\left\{\begin{array}{ll}
			\ns\ds \!\!\!\! 		{\rm(\romannumeral1)} \   (Y ,Z,V,A) \in  \sS_\mathbb F ^2[t,T];\\
			\ns\ds\!\!\!\! 			{\rm(\romannumeral2)} \     Y_s\! =   \Phi   ( X_T ^{t,x_t;u})  
			+   \int_s^T   f \big(  r, X_r^{t,x_t;u}, Y_r, Z_r, \int_El( e) V_r(e)\n(\mathrm  de) , u_r \big) \mathrm dr- ( A_T  - A_s ) 
			\\
			\ns\ds\!\!\!\! \hskip 1.38cm  -   \int_s^T Z_r \mathrm dB_r- \int_s^T   \int_E V_r(e) \tilde N(\mathrm dr,\mathrm de) ,\q   s\in [t,T];\\
			\ns\ds\!\!\!\! 			{\rm(\romannumeral3)} \   Y_s \leqslant   h  ( s,X_s^{t,x_t;u} ),\q   \mbox{a.e. } s \in  [t,T];\\
			\ns\ds\!\!\!\! 			{\rm(\romannumeral4)} \   \int_t^T \big(  h (s, X_{s-}^{t,x_t;u} )-Y_{s-} \big) \mathrm dA_s =0, 
		\end{array}\right.
	\end{equation} 
	 where  the driver $f    :    [0,T] \times \dbR^n \times \dbR \times \dbR^d  \times \dbR \times U \rightarrow \dbR$, the terminal   $\F  :    \dbR^n \to \dbR$ and the obstacle   $h:[0,T] \times \dbR^n \rightarrow \dbR$ are assumed to satisfy

 \begin{description}
   \item[(H$_2$)]  (\romannumeral1) for all $(x,y,z,v)\in  \dbR^n \times \dbR \times \dbR^d  \times \dbR$, $f(\cd,x,y,z,v,\cd)$ is continuous in $(r,u)\in[0,T]\times U$, $h(\cd,x)$ is continuous in $r\in [0,T]$;

  (\romannumeral2)  $f,h,\Phi$ is Lipschitz continuous in $ (x,y,z,v ) \in\dbR^n \times \dbR \times \dbR^d \times \dbR $, uniformly with respect to $(r,u)\in[0,T]\times U$;

	(\romannumeral3)  there exists some constant $\k>0$ such that  $0 \les l(  e) \les \k(1\land |e|)$;

 
	 %

  (\romannumeral4)  for all $x\in\dbR^n$, $ \F(x)\les h(T,x)$.
  
 \item[(C)] $v\to f(t,x,y,z,v)$ is non-decreasing, for all $(t,x,y,z)\in[0,T]\times\dbR^n\times \mathbb{R} \times \mathbb{R}^{d}$.  
 \end{description}

%
%
%
We remark that the obstacle process in \eqref{RBSDEP} is $h  ( \cd,X_\cd^{t,x_t;u} )$, which obviously satisfies {\bf(A$_1$)}-(iii). Hence, under {\bf (H$_1$)} and {\bf (H$_2$)}, for any $(t,x_t )\in[0,T]\times L^2_{\mathcal F_t}( \Omega ; \mathbb R^n)$, $u(\cd)\in \cU_{t,T}$, Lemma \ref{well} can be applied here to guarantee
\eqref{RBSDEP} to admit  the unique solution $(Y,Z, V, A)\equiv(Y^{t,x_t;u},Z^{t,x_t;u},V^{t,x_t;u},A^{t,x_t;u}) \in  \sS_\mathbb F ^2[t,T] $. 
Moreover, under {\bf (H$_1$)}, {\bf (H$_2$)} and {\bf (C)}, by adopting the technique of Proposition 6.1 in \cite{BL-2011} and using Lemma  \ref{Le-SDE},   the following   estimates hold, for $t\in[0,T]$,  $x_t,x_t^\prime \in   L^2_{\mathcal F_t}( \Omega ; \mathbb R^n) $ and $u(\cd)\in \cU_{t,T}$, $\dbP$-a.s., 
	 \begin{equation}\label{esti-rbsde}
	 	\begin{array}{ll}
	 		\ns\ds\!\!\!   {\rm (\romannumeral1)}\ \dbE ^{\mathcal F_t}\[   \sup\limits_{r\in[t,T]}  |Y^{t,x_t;u}_r |^2 + \int_t^T \big(   |Z^{t,x_t;u}_r |^2 +  \|V^{t,x_t;u}_{r}(\cd) \|_{\n,2 } ^2      \big)\mathrm dr 
	 		+\sup\limits_{r\in[t,T]}   |A_r^{t,x_t;u}|^2    \]
	 		\les C (1+|x_t|^2);\\
	 		\ns\ds\!\!\! 	{\rm (\romannumeral2)}\  \dbE  ^{\mathcal F_t} \[ \sup\limits_{r\in[t,T]} |Y^{t,x_t;u}_r  -Y^{t,x_t^\prime;u}_r |^2  \] \les C  |x_t-x_t^\prime|^2.\\
	 	\end{array}
	 	\end{equation}
 
With the above preparation,    for any initial pair $(t,x)\in[0,T]\times\dbR^n$ and the admissible control $u(\cd)\in\cU_{t,T}$, we  define the cost functional of the control problem as follows, 
	 \begin{equation*}
	 J(t,x;u(\cd)) := Y^{t,x;u}_t,
	 \end{equation*}
	 which is of recursive form.

\no Notice that it is classical in the theory of BSDEs that, for $(t,x_t)\in[0,T]\times   L_{\cF_t}^2(\O;\dbR^n)$ and $u(\cd)   \in\cU_{t,T}$,  
	 $ 
	 J(t,x_t;u(\cd) )=J(t,x;u (\cd)) |_{x=x_t} =Y^{t,x_t;u}_t  $ holds true, $ \dbP\mbox{-a.s.}$, for which we   refer to   \cite{BL-2009, BL-2011, WY-2008}, etc.
	%


\ms

Now, we formulate the following  reflected stochastic recursive  control problem with jumps which is parameterized by the initial pair $(t,x)\in [0,T] \times \dbR^n$.

	  {\bf Problem (C)$_{t,x}$}  {\sl For any $(t,x)\in[0,T]\times\dbR^n$, find $\bar u(\cd) \in\cU_{t,T}$ such that
	 	\begin{equation}\label{VF}
	 	J(t,x;\bar u (\cd))=\essinf_{u(\cd) \in \cU_{t,T}}J(t,x;u (\cd)):= W(t,x).
	   \end{equation}}
	  The control $\bar u(\cd) $ satisfying \eqref{VF} is called  as an optimal control of Problem (C)$_{t,x}$,  the corresponding $\bar X =X^{t,x;\bar u} $ is the  corresponding optimal state process,   and $W:[0,T]\times\dbR^n\to\dbR$ is said to be  the value function of   Problem (C)$_{t,x}$. 
 
 
	 

	 	 Note that, as the essential infimum of the $\mathcal{F}_t$-measurable cost functional $J(t,x;u(\cd) )$ over a family of control processes, for all $(t,x)\in [0,T]\times \mathbb R^n$, $W(t,x)$ is an $\mathcal{F}_t$-measurable random variable. 
However,  we can   prove it   to be deterministic under our conditions, similar to Proposition 3.3 in \cite{BL-2008}, and Proposition 3.1 in \cite{BL-2011}, etc. 
\vspace{-0.03cm}
	 	\begin{proposition}\sl
	 		For all $(t,x ) \in [0,T] \times \mathbb{R}^n  $, the value function $W(t,x)$ is a deterministic function, i.e., $\mathbb{E}[W(t,x)]=W(t,x)$, $\dbP$-a.s. 
	 	\end{proposition}

 \vspace{-0.05cm}
	From the definition of the value function and the estimates in \eqref{esti-rbsde}, we get the following result directly.
	 	\begin{proposition}\sl
	 		\label{val-lip}
	 		 Under the conditions {\rm\textbf{(H$_1$)}}, {\rm \textbf{(H$_2$)}} and {\rm \textbf{(C)}}, $W(\cd,\cd)$ is of linear  growth and Lipschitz continuous in $x\in\dbR^n$, i.e., for all $t\in[0,T]$, $x,x^\prime\in\dbR^n$, 
	$$  	|W(t,x)|\leqslant C(1+|x|),  \qq   |W(t,x)-W(t,x')|\leqslant C|x-x'|.	$$
	 	\end{proposition}
 
\ms 	 
 
To study 
 the dynamic programming principle (DPP, for short) of   Problem (C)$_{t,x}$, we need to  generalize the   notation of the stochastic backward semigroup introduced initially  by Peng \cite{P-1997} to our framework.

\begin{definition}\sl
 Given the initial data $(t,x)\in [0,T]\times \mathbb R^n,$ $\d\in[0,T-t]$,   $u(\cd)  \in \mathcal{U}_{t, t+\d}$ and a real-valued random variable $\eta \in   L^{2}_ {\mathcal{F}_{t+\delta}}(\Omega; \mathbb{R}^n)$, we define the backward stochastic semigroup  $G_{s, t+\delta}^{t, x ; u}[\cd]$ as
	$$
	G_{s, t+\delta}^{t, x ; u}[\eta]:= \tilde{Y}_{s}^{t, x ; u} , \quad s \in[t, t+\delta],
	$$
	where $(\tilde{Y} ^{t, x ; u}, \tilde{Z} ^{t, x ; u},\tilde{V} ^{t, x ; u}, \tilde{A} ^{t, x ; u})   $ is the solution of the following RBSDE with jumps on $[t,t+\delta]$:
	\begin{equation*}\left\{
		\begin{array}{ll}
			\ns\ds\!\!\! {\rm(\romannumeral1)} \ (\tilde{Y} ^{t, x ; u}, \tilde{Z} ^{t, x ; u},\tilde{V} ^{t, x ; u}, \tilde{A} ^{t, x ; u})\in\sS^2_ \mathbb F[t,t+\d];\\
			\ns\ds\!\!\!\! 			{\rm(\romannumeral2)} \     \tilde{Y}_{s}^{t, x ; u}\! =   \eta  
			+   \int_s^{t+\delta}   f \big(  r, {X}_{r}^{t, x ; u}, \tilde{Y}_{r}^{t, x ; u}, \tilde{Z}_{r}^{t, x ; u},   \int_El( e) \tilde{V}_{r}^{t, x ; u}(e)\n(\mathrm  de) , u_r \big) \mathrm dr 
			-  ( \tilde{A}_{t+\delta}^{t, x ; u} - \tilde{A}_{s}^{t, x ; u} )
			\\
			\ns\ds\!\!\!\! \hskip 1.9cm -   \int_s^{t+\delta} \tilde{Z}_{r}^{t, x ; u} \mathrm dB_r
			- \int_s^{t+\delta}   \int_E \tilde{V}_{r}^{t, x ; u}(e) \tilde N(\mathrm dr,\mathrm de) ,\quad  s\in [t,t+\d];\\
			\ns\ds\!\!\!\! 			{\rm(\romannumeral3)} \    \tilde{Y}_{s}^{t, x ; u} \leqslant   h  ( s,{X}_{s}^{t, x ; u} ),\q  a.e. \ s \in  [t,t+\d];\\
			\ns\ds\!\!\!\! 			{\rm(\romannumeral4)}  \int_t^{t+\delta} \big(  h (r, {X}_{r-}^{t, x ; u} )-\tilde{Y}_{r-}^{t, x ; u} \big) \mathrm d\tilde{K}_{r}^{t, x ; u}=0 .
		\end{array}
	\right.\end{equation*}

\end{definition}

Note that for RBSDE with jumps \eqref{RBSDEP}, we have the following
 $$J(t, x ; u(\cd))  =Y_{t}^{t, x ; u}=G_{t, T}^{t, x ; u}\big[\Phi(X_{T}^{t, x ; u})\big]=G_{t, t+\delta}^{t, x ; u}[Y_{t+\delta}^{t, x ; u}] =G_{t, t+\delta}^{t, x ; u}\big[J\big(t+\delta, X_{t+\delta}^{t, x ; u} ; u(\cd) \big)\big].$$

With the help of   backward stochastic semigroup, we obtain the dynamic programming principle of Problem (C)$_{t,x}$   as follows.  
	\begin{proposition}\label{Th-DPP}\sl
		(DPP) Under {\rm\textbf{(H$_1$)}},  {\rm\textbf{(H$_2$)}} and {\rm\textbf{(C)}},  for any $0\leqslant t<t+\delta \leqslant T,\ x \in \mathbb{R}^{n}$,
		\begin{equation}\label{DPP}
			\begin{array}{ll}
				W(t, x)=\underset{u \in \mathcal{U}_{t, t+\delta}}{\operatorname{essinf}} G_{t, t+\delta}^{t, x ; u}\big[W(t+\delta, X_{t+\delta}^{t, x ; u})\big].
			\end{array}
		\end{equation}		
	\end{proposition}  
	 
Due to the lack of the continuity of the coefficients in control variable,  the method in \cite{P-1997} can not be used here.  However, we can  adopt the approach introduced in \cite{BL-2008, BL-2009} to prove \eqref{DPP}.  The detail is similar to Theorem 3.1 in  \cite{BL-2009}, so that we skip it here.

Thanks to the DPP, 	  we can also get the  continuity of $W(\cd,\cd)$ in $t\in[0,T]$ as follows.
\bp\label{}\sl
		Under {\rm\textbf{(H$_1$)}}, {\rm \textbf{(H$_2$)}} and {\rm \textbf{(C)}},    the value function $W(\cd,\cd)$ is continuous in $t\in[0,T]$. 
	\ep
 For the proof, we may refer to Theorem 3.2 in \cite{BL-2011}.
	\subsection{The Obstacle Problems of partial integral-differential equations of HJB type}\label{VIS}
	
In this subsection, 
we aim to associate Problem (C)$_{t,x}$ with a kind of partial differential equations. To simplified the notations, we first put 
	$$
	\begin{array}{ll}
\ns\ds\!\!\! \cL^{u} \Psi(t,x) :=   \Psi_x(t, x).  b(t, x, u)  + \frac{1}{2} \tr \big(\sigma \sigma^\top(t, x, u)  \Psi_{xx}(t, x)\big),\\
		\ns\ds\!\!\! \cB^{u} \Psi(t, x) :=\int_{E}\big[\Psi \big(t, x+\gamma(t, x, u,e)\big)-\Psi(t, x)- \Psi_x(t, x). \gamma(t, x, u, e)\big]  \nu(\mathrm de),	\\
		\ns\ds\!\!\!	\cC^{u} \Psi(t, x) :=\int_{E}  l( e) \big[ \Psi\big(t, x+\gamma(t, x, u, e)\big)-\Psi(t, x)\big]  \nu(\mathrm d e),\\
 	\ns\ds\!\!\!   \mathbb H(t,x, (\Psi , \Psi_x , \Psi_{xx})(t,x), u)	:= \cL^{u} \Psi(t,x) + \cB^{u} \Psi(t,x)	 \\
\ns\ds\hskip 4.6cm		+f\big(t,x,\Psi(t,x),  \Psi_x(t,x). \sigma(t,x,u)   , \cC^{u} \Psi(t,x),  u \big),
	\end{array}
	$$
where $(t,x,u)\in [0,T]\times \dbR^n\times U$ and $\Psi\in C^{1,2}([0,T]\times\dbR^n;\dbR)$.
Then we consider  the following   obstacle problem   for PIDE   of HJB type,
	\begin{equation}\label{HJB}
		\left\{
		\ba{ll}
		\ns\ds\!\!\!\! \max\Big\{ \! W(t,x)-h(t,x),- \frac{\partial}{\partial t}W(t,x) \!-\!\inf_{u\in U}  \mathbb H \big(
		t,x,(W,W_x, W_{xx})(t,x),u \big)  \Big\}=0,\ 
		(t,x)\in[0,T]\times\dbR^n,\\
		\ns\ds\!\!\!\! W(T,x) = \Phi(x),\quad x\in \mathbb R^n.
		\ea
		\right.
	\end{equation}
%
%
%
For convenience, we call  \eqref{HJB} as PIDE. The aim is to associate the value function $W(\cd,\cd)$  (in \eqref{VF}) with \eqref{HJB}. 
 Under our condition, 
$W(\cd,\cd)$ is not necessarily smooth, which pushes us to resort to a kind of weak solution, i.e., viscosity solution, 
introduced firstly by  Crandall, Lions \cite{CL}; we also refer to Crandall, Ishii, and Lions \cite{CIL}.
 We first  generalize  the notion of the viscosity solution  to adapt  to   PIDE \eqref{HJB}.

	\begin{definition}\label{vis-def}\sl
	Let  $W\in C\big([0, T] \times \mathbb{R}^{n};\dbR\big)$.   
	 
		{\rm(\romannumeral1)}  
	We call $W(\cd,\cd)$ as a viscosity subsolution of    \eqref{HJB} if $W(T,x)\leqslant\Phi(x)$, for all $x\in \mathbb{R}^{n}$, 
	 and if for all
		functions $\varphi \in C_{l, b}^{3}\big([0, T] \times \mathbb{R}^{n}\big)$ and any sufficiently small $\delta>0$,
		%
		\begin{equation*}
			\begin{array}{ll}
				\ns\ds\!\!\!\! 	
				\max \Big\{W(t, x)-h(t, x),-\frac{\partial \varphi}{\partial t}(t, x)-\underset{u \in U}{\inf}\big\{ \cL^{u} \varphi(t,x)
				
				+\cB^{\delta, u}(W, \varphi)(t, x)\\
				\ns\ds\!\!\!\! 	\hskip 3.6cm
				+	f \big (t, x, W(t, x),   \varphi_x(t, x).   \sigma(t, x, u ), \cC^{\delta, u}(W, \varphi)(t, x), u \big)\big\}\Big\} \leqslant 0,
			\end{array}
		\end{equation*}
		holds at any local maximum point $(t,x)$ of $W-\varphi$, where
		\begin{equation*}
			\begin{array}{ll}
				%
				%
				\ns\ds\!\!\!\! 	
				\cB^{\delta, u}(W, \varphi)(t, x):=  \int_{E_{\delta}}\big[ \varphi(t, x+\gamma(t, x, u, e))-\varphi(t, x) -\varphi_x(t, x) .  \gamma(t, x, u, e)\big] \nu(\mathrm d e)
				\\
				\ns\ds\!\!\! 
				\hskip 2.825cm
				+\int_{E_{\delta}^{c}}  \big[  W(t, x+\gamma(t, x, u, e))-W(t, x)- \varphi_x(t, x).   \gamma(t, x, u, e)\big]\nu(\mathrm d e),\\
				\ns\ds\!\!\!\! 	
				\cC^{\delta, u}(W, \varphi)(t, x):= \int_{E_{\delta}}l( e)\big[ \varphi(t, x+\gamma(t, x, u,   e))-\varphi(t, x)\big]  \nu(\mathrm d e) \\
				\ns\ds\!\!\! 
				\hskip 2.825cm
				+\int_{E_{\delta}^{c}}l(  e) \big[W(t, x+\gamma(t, x, u,  e))-W(t, x)\big] \nu(\mathrm d e),
				
			\end{array}
		\end{equation*}
		with $E_{\delta}:=\{e \in E   \mid |e|  <\delta\} $.

		{\rm(\romannumeral2)} 
	We call $W(\cd,\cd)$ as a  viscosity supersolution of    equation \eqref{HJB} if $W(T,x)\geqslant\Phi(x)$, for all $x\in \mathbb{R}^{n}$, and if for all functions $\varphi \in C_{l, b}^{3}\big([0, T] \times \mathbb{R}^{n}\big)$, sufficiently small $\d>0$ and the local minimizer $(t,x) \in[0, T) \times \mathbb{R}^{n}$ of $W-\varphi$,  
		%
			\begin{equation}\label{VIS-sup}
			\begin{array}{ll}
				\ns\ds\!\!\!\! 	
				\max \Big\{W(t, x)-h(t, x),-\frac{\partial \varphi}{\partial t}(t, x)-\underset{u \in U}{\inf}\big\{ \cL^{u} \varphi(t,x)
				
				+\cB^{\delta, u}(W, \varphi)(t, x)\\
				\ns\ds\!\!\!\! 	\hskip 3.6cm
				+	f \big (t, x, W(t, x),   \varphi_x(t, x).   \sigma(t, x, u), \cC^{\delta, u}(W, \varphi)(t, x), u \big)\big\}\Big\} \ges 0.
			\end{array}
		\end{equation}

		{\rm(\romannumeral3)} 
	If $W(\cd,\cd)$	is both   a viscosity subsolution and a viscosity supersolution of   \eqref{HJB}, it is called  as a viscosity solution of  PIDE \eqref{HJB}.
	\end{definition}

	\begin{remark}\label{Re-1}\sl  We point out that, due to the linear growth of the value function $W(\cd,\cd)$ in Proposition \ref{val-lip},    the local maximum (resp. minimum) of $W-\varphi$ in {\rm Definition \ref{vis-def}-(i)} (resp. \rm {Definition \ref{vis-def}-(ii)}) can be replaced by a global one in the proof of Theorem \ref{vis-solution}.
Moreover, 
		$ 
		\cB^{\delta, u}(W, \varphi)(t, x)$  and $\cC^{\delta, u}(W, \varphi)(t, x)$  in Definition \ref{vis-def}   can be replaced by $ \cB^{u} \varphi(t, x)$ and  $ \cC^{u} \varphi(t, x)$, resp. 
		Such an observation can be referred to \cite{Barles-1997,BHL-2011}, etc.
	\end{remark}

	%
	%
	  To prove the value function $W(\cd,\cd)$ in \eqref{VF} to be a viscosity solution of   \eqref{HJB}, we introduce the penalized equations  of     RBSDE with jumps \eqref{RBSDEP}.  
	 Precisely, for each $n\in \dbN$,  
 %
	\begin{equation}\label{BSDEP-pen}
		\begin{array}{ll}
			\ns\ds  Y_s^{n,t,x;u} = \Phi ( X_T^{t,x;u}) + \int_s^T   f\big(  r, X_r^{t,x;u}, Y_r^{n,t,x;u},  Z_r^{n,t,x;u}, \int_El( e) V_r^{n,t,x;u} ( e)\n(\mathrm de),u_r  \big)  \mathrm dr\\
			\ns\ds \hskip 1.6cm -( A_T^{n,t,x;u}- A_s^{n,t,x;u}) - \int_s^T   Z_r^{n,t,x;u}   \mathrm dB_r  - \int_s^T \int_E   V_r^{n,t,x;u} (e)   \tilde N(\mathrm dr,\mathrm de) ,\quad s \in [t,T],
		\end{array}
	\end{equation} 
where $\ds A_s^{n,t,x;u}:=\int_t^sn \big (h( r,X_r^{t,x;u}  )- Y_r^{n,t,x;u}  \big)^- \mathrm dr $ and   $X^{t,x ;u}$ satisfies  \eqref{state} with $x_t=x \in \dbR^n$.
It is obvious that, for each $n\in\dbN$,  \eqref{BSDEP-pen} admits the unique solution $(Y^{n,t,x;u},Z^{n,t,x;u},$ $V^{n,t,x;u} )\in\mathcal{S}^2_{ \mathbb F }( t,T; \mathbb R  ) \times  \mathcal{M}^2_{ \mathbb F }( t,T; \mathbb R ^d ) \times \mathcal K _\nu^2 (t,T;\mathbb  R  ) $.
Moreover, we have  the following approximation result from  Lemma \ref{well}.

	\begin{lemma}\label{well-C}\sl
	Assume {\rm\textbf{(H$_1$)}}, {\rm \textbf{(H$_2$)}} and {\rm \textbf{(C)}} hold.
 The sequence $\{(Y^{n,t,x;u},Z^{n,t,x;u},V^{n,t,x;u},A^{n,t,x;u})\}_{n\in\dbN}$ has a limit  $(Y^{t,x;u},Z^{t,x;u},V^{t,x;u},A^{t,x;u})$, which is indeed the solution of RBSDE with jumps \eqref{RBSDEP}. Moreover, 
 the convergence of  $Y^{n,t,x;u} $ to $Y^{t,x;u}$ in $\mathcal{S}^2_{ \mathbb F }( t,T; \mathbb R  )$ is  decreasingly.

\end{lemma}

With  \eqref{state}
 and \eqref{BSDEP-pen}, we  formulate a sequence of control problems as follows.
For  each $n\in\dbN$, we introduce the cost functional
	\begin{equation*}
			J_{n}(t,x;u(\cd)):=Y^{n,t,x;u}_{t},\q  (t,x) \in [0,T]\times \dbR^n, \  u(\cd)\in\mathcal{U}_{t,T},
	\end{equation*} 
	and the value function
	\begin{equation}\label{W-n}
			W^{n}(t,x)=\underset{u(\cd)\in \mathcal{U}_{t, T}}{\operatorname{essinf}}J_{n}(t,x;u(\cd)),\q (t,x) \in [0,T]\times \dbR^n.
	\end{equation}
According to Theorems 4.1 and 5.1 in \cite{BHL-2011}, for each $n\in\dbN$, $W^n(\cd,\cd)$ is the unique viscosity solution of the following HJB equation,
	\begin{equation}\label{HJB*}
		\left\{
		\ba{ll}
		\ns\ds\!\!\!\! - \frac{\partial}{\partial t}W^{n}(t,x) \!-\!\inf_{u\in U}  \mathbb H^{n} \big(
		t,x,(W^{n} , W^{n}_x ,  W^{n}_{xx})(t,x),u \big)  =0,\quad (t,x)\in[0,T]\times\dbR^n,\\
		\ns\ds\!\!\!\! W^{n}(T,x) = \Phi(x) ,\quad x\in \mathbb R^n,
		\ea
		\right.
	\end{equation}
	where
	$$
			\mathbb H^{n}(t, x, (W^{n} , W^{n}_x ,  W^{n}_{xx})(t,x), u)=  
			\mathbb H (t, x, (W^{n} , W^{n}_x ,  W^{n}_{xx})(t,x), u)-n\big (h(t,x)-W^{n}(t,x)\big)^{-}.
$$
  Notice that the uniqueness  holds in the following space:
  \begin{equation*}
  	\begin{array}{ll}
  		\ns\ds\!\!\! 
  		\Theta=  \Big\{\varphi\in C\big([0, T] \times \mathbb{R}^{n};\mathbb{R} \big)\mid \exists \widetilde{A}>0 \text { such that } \\
  		\ns\ds\!\!\!  \hskip 1.325cm 
  		\underset{|x| \rightarrow \infty}{\lim} \varphi(t, x) \exp \Big\{-\widetilde{A}\big[\log ((|x|^{2}+1)^{\frac{1}{2}})\big]^{2}\Big\}=0, \text { uniformly in } t \in[0, T]\Big\} .
  	\end{array}
  \end{equation*}
  
	 From Lemma \ref{well-C} and the definition of $W^n(\cd,\cd)$, we give  the following result without the proof. 
	 The readers may refer to Lemma 4.3  and Remark 4.4 in \cite{WY-2008}.  
	\begin{lemma}\label{Le-Vn}\sl
		For all $(t,x)\in[0,T]\times\dbR^n,$ $ n\in\dbN$, 
		$$W^{1}(t, x) \geqslant \cdots \geqslant W^{n}(t, x) \geqslant W^{n+1}(t, x) \geqslant \cdots \geqslant W(t, x).$$
Moreover, for all $(t,x)\in[0,T]\times\dbR^n$, $\lim\limits_{n\to\i}W^n(t,x)=W(t,x)$, which is also uniform on compacts sets.
	\end{lemma}
	 
	 Going further, the following result is true, referring to Lemma 4.6 in \cite{WY-2008}.
 	\begin{lemma}\label{Le-inf-lim}\sl
		For all $(t,x)\in[0,T]\times\dbR^n$ with $(t_n,x_n)\to(t,x)$ as $n\to\i$,  and $\f\in C_{l, b}^{3}  ( [0, T] \times \mathbb R ^n  )$, we have 
		$$	\lim\limits_{n\to\i}\inf_{u\in U}\mathbb H (t_n, x_n, (W^n , \f _x , \f _{xx})(t_n,x_n), u)=\inf_{u\in U} \mathbb H (t , x , (W , \f _x , \f _{xx})(t ,x ), u).$$
	\end{lemma}

With the above auxiliary results, we get the first main result of this work.
	\begin{theorem}\label{vis-solution}\sl
		Under  {\rm \textbf{(H$_1$)}}, {\rm\textbf{(H$_2$)}} and {\rm \textbf{(C)}}, the value function $W(\cd,\cd)$ introduced in \eqref{VF} is the unique  viscosity solution of   \eqref{HJB} in $\Theta$.

	\end{theorem}
 \begin{proof}
First, we show $W(\cd,\cd)$ is a viscosity supersolution of    \eqref{HJB}. 
	For this,   consider any   $ \varphi \in C_{l, b}^{3}  ( [0, T] \times \mathbb R ^n  )$ such that  $W-\varphi$ attains the 
		minimum at $(t,x)$. Without lose of generality, we assume  that $W(t, x) = \varphi(t, x)$.
 Notice that \eqref{VIS-sup} is obvious if $W(t,x) \ges h(t,x)$. Therefore, we just consider the case when $W(t,x) < h(t,x)$.

	According to the continuity of  $W$ and Lemma \ref{Le-Vn}, there exists some sequence $\{(t_n, x_n)\}_{n \in\dbN}$  such that (at least along a subsequence),

	{\rm(a)} 
	$(t_n, x_n) \to  (t, x)$, as $n \to \i$; 
 
	{\rm(b)} 
	$(W^n -\varphi)(s,y) \ges ( W^n   -  \varphi)(t_n, x_n)$  in a neighborhood of $(t_n, x_n )$, for all $n \ges 1$; 
	 
	{\rm(c)} 
	$W^n (t_n, x_n) \to  W(t, x)$, as $n \to \i$.

  We claim  that, for  any sufficiently large  $n$,
	$W^n (t_n, x_n) \les  h(t_n, x_n)$ and hence $ \big( h(t_n, x_n)- W^n (t_n, x_n)  \big)^-=0$. 
	In fact, from $W(t,x)<h(t,x)$ and Lemma \ref{Le-Vn}, for any $\epsilon>0$, there exists some integer $N_1$ such that   for all $n \ges N_1$ ,
	$$
	\begin{array}{ll}
		\ns\ds\!\!\!  W^n(t_n,x_n) - h(t_n,x_n)  <  W^n(t_n,x_n) - W(t,x)   + h(t,x) - h(t_n,x_n)  \les 2\epsilon. \\
		%
	\end{array}
	$$
Due to the arbitrariness of $ \epsilon$, we obtain $W^n(t_n,x_n) \les  h(t_n,x_n),\ \forall n\ges N_1$.


On the other hand,  $ W^n (\cd,\cd ) $ being  a supersolution of HJB equation \eqref{HJB*} implies us that, for all $n \ges  1$,
	\begin{equation*}\label{ineq2}
			\begin{array}{ll}
					\ns\ds\!\!\!
					  \varphi_t  (t_n, x_n)
					+  \inf_{u \in U}  \mathbb H^n  (t_n, x_n, (W^n , \f _x , \f _{xx})(t_n,x_n), u)  \les 0.\\
				\end{array}
	\end{equation*}

	\no For all $n \ges  N_1$, combined with   $\big( h(t_n, x_n) - W^n (t_n, x_n) \big)^- =0$, we get 
	\begin{equation*}\label{ineq3}
		\begin{array}{ll}
			\ns\ds\!\!\!
			  \varphi_t  (t_n, x_n)
			+  \inf_{u \in U} \mathbb H  (t_n, x_n, (W^n , \f _x , \f _{xx})(t_n,x_n), u)\les 0.\\
		\end{array}
	\end{equation*}

	\no Then, letting $n \to \i$ and  using Lemma \ref{Le-inf-lim}, we obtain
	\begin{equation*}   
			 \varphi_t  (t ,x)
			+  \inf_{u \in U}  \mathbb H  (t , x , (W , \f _x , \f _{xx})(t,x), u)\les 0.  
	\end{equation*}
Combined with Remark \ref{Re-1},    $W(\cd,\cd)$ is a viscosity supersolution of   \eqref{HJB}. 
	
 \ms
	Next, we  prove $W(\cd,\cd)$ to be a viscosity subsolution of   \eqref{HJB}. 
	Let  $ \f  \in C_{l, b}^{3}  ( [0, T] \times \mathbb R ^n  )$ and  $W-\f $ attain a  
		maximum  at $(t,x)$.
Then, there also exists some sequence $\{(t_n, x_n)\}_{n \in  \dbN}$  
	such that
 
 {\rm(a)} 
	$(t_n, x_n) \rightarrow(t, x)$, as $n \to  \i$;

	{\rm(b)} 
	$W^n(s,y)-\f (s,y) \les W^n(t_n, x_n)-\f (t_n, x_n)$  in a neighborhood of $(t_n, x_n)$, for all $n \ges 1$; 
 
	{\rm(c)} 
	$W^n(t_n, x_n) \rightarrow W(t, x)$, as $n \to  \i.$ 

 Using that $W^n(\cd,\cd)$ is    a subsolution of HJB equation \eqref{HJB*}, for all $n\ges  1$, we get
	\begin{equation}\label{ine1}  
			 \f _t  (t_n, x_n)
			+  \inf_{u \in U}  \mathbb H^n  (t_n, x_n, (W^n , \f _x , \f _{xx})(t_n,x_n), u)  \ges 0.\\
	\end{equation}

	\no Due to  $n\big(  h(t_n, x_n)-W^n (t_n, x_n)  \big)^- \ges 0$, the following holds,
	\begin{equation*} 
			  \f _t  (t_n, x_n)
			+  \inf_{u \in U}   \mathbb H   (t_n, x_n, (W^n , \f _x , \f _{xx})(t_n,x_n), u)  \ges 0.\\
	\end{equation*}
	
 Letting  $n \to \i$ and  using Lemma \ref{Le-inf-lim} again, we have
	\begin{equation}\label{ine2}
		\begin{array}{ll}
			\ns\ds\!\!\!
			  \f _t  (t ,x)
			+  \inf_{u \in U}  \mathbb H  (t , x , (\f , \f _x , \f _{xx})(t,x), u) \ges 0.\\
		\end{array}
	\end{equation}
Further,  we need to  prove that $W(t, x) \les h(t, x)$.
	From  \eqref{ine1} and the definition of $f^n$,we get
	\begin{equation*}  	 n\big(h(t_n, x_n) - W^n(t_n, x_n)   \big)^- 
			\les   \f _t (t_n, x_n) 
			+\inf_{u \in U}   \mathbb H   (t_n, x_n, (W^n , \f _x , \f _{xx})(t_n,x_n), u) . 
	\end{equation*}
 We remark that  the limit of the right-hand side of the above inequality exists. It means   $	n\big(h(t_n, x_n) -W^n (t_n, x_n)  \big)^-$ can't tend to infinity, as  $n \to \i$. 
Therefore, $\lim\limits_{n\to\i}\big( h(t_n, x_n)-W^n (t_n, x_n)   \big)^-=0     $ which leads to   $W(t, x) \les h(t, x)$.
Combined with \eqref{ine2}, we get 
	\begin{equation*}  	
				\max \Big\{W(t, x)-h(t, x),- \varphi_t(t, x)-\underset{u \in U}{\inf}\mathbb H  (t , x , (W , \f _x , \f _{xx})(t,x), u) \Big\}\leqslant 0.
		\end{equation*}
Hence, $W$ is a viscosity subsolution of   \eqref{HJB}. Based on the above, we complete the proof of existence.  Furthermore, by the subsequent result (Theorem \ref{1geq2}), the uniqueness is also true in $\Theta$.
\end{proof}

In the following, we focus on the study of the comparison theorem of PIDE \eqref{HJB}. 
	\begin{theorem}\label{1geq2}\sl
		Assume that {\rm\textbf{(H$_1$)}} and {\rm \textbf{(H$_2$)}} hold. Let the continuous functions $W_1(\cd,\cd) \in \Theta$ and $W_2(\cd,\cd)$  $ \in \Theta$ be a viscosity supersolution and a viscosity subsolution of PIDE \eqref{HJB}, respectively, then 
		$$
		W_1(t, x) \geqslant W_2(t, x),\quad \text { for all }(t, x) \in[0, T] \times \mathbb{R}^{n}.
		$$
	\end{theorem}
	In order to prove Theorem \ref{1geq2}, we just need to generalize the proof in \cite{Barles-1997} by considering the treatment of the obstacle term, which can be refer to Theorem 4.9 in \cite{WY-2008}. We can go such a way smoothly once the following two auxiliary results hold true.
	\begin{lemma}\label{omega}\sl
		Let $W_{1}(\cd,\cd) \in \Theta$ be a viscosity supersolution and $W_2(\cd,\cd) \in \Theta$ be a viscosity subsolution of \eqref{HJB}.  Then the function $\widehat W:= W_1-W_2$ is a viscosity supersolution of the following equation
		\begin{equation*}
			\left\{
			\ba{ll}
			\ns\ds\!\!\!\! 
			\max\Big\{ \! \widehat W(t,x),- \frac{\partial}{\partial t}\widehat W(t,x) \!-\!\inf_{u\in U}  \Big[\cL^{u} \widehat W(t,x)	+ \cB^{u} \widehat W(t,x)-K| \widehat W(t,x)|-K|D\widehat W(t,x)  \cd  \sigma(t, x, u)|\\
			\ns\ds\!\!\! 
			\hskip 7.5cm
			-K\big(\cC^{u} \widehat W(t,x)\big)^{-}\Big] \Big\}=0,\q
			(t,x)\in[0,T]\times\dbR^n,\\
			\ns\ds\!\!\!\! 
			\widehat W(T,x) = 0,\quad x\in \mathbb R^n,
			\ea
			\right.
		\end{equation*}
	where $ K$ is a Lipschitz constant of $f(t, x, \cdot, \cdot)$  in $ ( y,z,v ) \in  \dbR \times \dbR^d \times \dbR $.
	\end{lemma}
	\begin{proof}
		We notice that $\widehat W(T,x)=W_1(T, x)-W_2(T, x) \geqslant   \Phi(x)-\Phi(x)=0$. Let $\varphi \in C_{l, b}^{3}\big([0, T] \times \mathbb{R}^{n}\big)$ and $(t_{0}, x_{0}) \in(0, T) \times \mathbb{R}^{n}$ be such that $(t_{0}, x_{0})$  be a strict global minimum point of $\widehat W-\varphi$. We introduce the function
		\begin{equation*}
			\begin{array}{ll}
				\ns\ds\!\!\!
				\Phi_{\varepsilon, \alpha}(t, x, s, y)=W_1(t, x)-W_2(s, y)+\frac{|x-y|^{2}}{\varepsilon^{2}}+\frac{(t-s)^{2}}{\alpha^{2}}-\varphi(s, y),\q (t,x,s,y) \in \big(  [0,T]\times \dbR^n \big)^2,
			\end{array}
		\end{equation*}
		where $\varepsilon,\ \alpha$ are positive parameters which are devoted to tending to $0$.
		\par Then, there exists a sequence $(\bar{t}, \bar{x}, \bar{s}, \bar{y})$ depending on $(\varepsilon, \alpha)$,  such that\\
		(\romannumeral 1)  $\Phi_{\varepsilon, \alpha}$  reaches its minimum in $[0, T] \times \bar{B}_{r} \times [0, T] \times \bar{B}_{r}$ at $(\bar{t}, \bar{x}, \bar{s}, \bar{y})$, where $B_{r}$ is a ball with a radius of $r$;\\
		(\romannumeral 2) 
		$(\bar{t}, \bar{x})$, $(\bar{s}, \bar{y}) \rightarrow(t_{0}, x_{0}) \text { when }(\varepsilon, \alpha) \rightarrow 0$;\\
		(\romannumeral 3) 
		$\ds \frac{|\bar{x}-\bar{y}|^{2}}{\varepsilon^{2}}$, 
		$\ds \frac{(\bar{t}-\bar{s})^{2}}{\alpha^{2}}$ are bounded and tend to zero as $(\varepsilon, \alpha) \rightarrow 0$.

		\vspace{0.2cm}
		 Since $W_1(\cd,\cd)$, $W_2(\cd,\cd)$ is continuous, we get immediately from (ii) that $\underset{(\varepsilon, \alpha) \rightarrow 0}{\lim} W_1(\bar{s}, \bar{y})=W_1(t_{0}, x_{0})$ and $ \underset{(\varepsilon, \alpha) \rightarrow 0}{\lim} W_2(\bar{s}, \bar{y}) = W_2(t_{0}, x_{0})$.

		 Since $(\bar{t}, \bar{x}, \bar{s}, \bar{y})$  is a local minimum point of $\displaystyle \Phi_{\varepsilon, \alpha}, \ W_2(s, y)-\frac{|\bar{x}-y|^{2}}{\varepsilon^{2}}-\frac{(\bar{t}-s)^{2}}{\alpha^{2}}+\varphi(s,y)$ gets a local maximum in $(\bar{s},\bar{y})$. Moreover, by the definition of  viscosity subsolution of \eqref{HJB}, we know $W_2(\bar{s}, \bar{y}) \leqslant h(\bar{s}, \bar{y})$. Thus, we have $W_2(t_{0}, x_{0}) \leqslant h(t_{0}, x_{0})$. If $W_1(t_{0}, x_{0}) \geqslant h(t_{0}, x_{0})$, we have
		$$
		\widehat W(t_{0}, x_{0})=W_1(t_{0}, x_{0})-W_2(t_{0}, x_{0}) \geqslant 0,
		$$
		and the proof is complete. Consequently, we only need to consider the case when $W_1(t_{0}, x_{0}) < h(t_{0}, x_{0})$ in the sequel. Due to the continuity  of $h$ and $W_1$  in $(t,x)$, for $\varepsilon>0$ and $\alpha>0 $ sufficiently small, we have 
		$	W_1(\bar{t}, \bar{x})<h(\bar{t}, \bar{x})$.

		After dealing with the obstacle term, it suffices to prove
		\begin{equation*}
			\begin{array}{ll}
				\ns\ds\!\!\! 
				\!-  \varphi_t(t_{0},x_{0}) \!-\!  \underset{u\in U} {\inf} \Big[\cL^{u} \varphi(t_{0}, x_{0})	+ \cB^{u} \varphi(t_{0}, x_{0})-K| \widehat W( t_{0},x_{0})|
				-K|  \varphi_x   \sigma(t_{0}, x_{0}, u )|-K\big(\cC^{u} \varphi(t_{0}, x_{0})\big)^{-}\Big] \!\geqslant\!  0.\\
				%
			\end{array}
		\end{equation*}
		In fact, the above can be derived by the following procedures of Lemma 3.7 in \cite{Barles-1997}, or Lemma 5.1 in \cite{LW-2015}. To avoid the repetition, we omit it here.
	\end{proof}
	%
	Similar to Lemma 3.8 in \cite{Barles-1997} (or Lemma 5.2 in \cite{BHL-2011}), we have the following results.
	\begin{lemma}\label{chi}\sl
	For any $\widetilde{A}>0$, given $\ds t_{1}=T-\frac{\widetilde{A}}{C_{1}}$, there exists $C_{1}>0$ such that for all $(t,x)\in[t_1,T]$, the function $\chi(t, x)$ satisfies
	$$ 
	 \frac{\partial}{\partial t} \chi (t, x)+\underset{u \in U}{\inf}   \Big\{\cL^{u} \chi(t, x)+\cB^{u} \chi(t, x)+K \chi(t, x)
	-K\big|  \chi_x (t, x)  \sigma(t, x, u )\big|-K(C^{u} \chi(t, x)\big)^{-}\Big\}>0,   \\
	$$
	where  
	$\chi(t, x)=-\exp \Big[\big(C_{1}(T-t)+\widetilde{A}\big) \psi(x)\Big]$
	with
	$
	\psi(x)=\Big[\log \big((|x|^{2}+1)^{\frac{1}{2}}\big)+1\Big]^{2},\ x\in\mathbb{R}^{n}. 
	$
\end{lemma}

With Lemmas \ref{omega} and \ref{chi},  the aforementioned comparison theorem of PIDE \eqref{HJB} (Theorem \ref{1geq2}) can be proved smoothly, refer to \cite{BL-2008,WY-2008,LW-2015}. 

\section{More properties of the value function}\label{more-pro}
	
	Based on the previous study, we further explore   more properties of the value function $W(\cd,\cd)$ under some additional conditions, such as  the semi-concavity  and  the  joint Lipschitz continuity. The importance of   these issues will be revealed in the study of stochastic verification theorem in Section \ref{SVT}.
 
 The research of  these   issues will involve the comparison between stochastic systems with differential initial times and initial states, which is very complex. To make the study more directly, we first try to transform these systems to operate  on the same time interval. The appearance of Brownian motion and Poisson random measure at the same time increases the  difficulty. 
 A kind of time-stretching  transformations of the jump noise,  called as Kulik's transformation, turns out to be a key tool in our research.  For the initial literatures about the Kulik's transformation, we can refer to \cite{K-2001,K-2006,J-2013}. 
We first recall some known results about the Kulik's transformation in the following. 

 For any $t_0,t_1\in[0,T) $,  $\lambda\in[0,1]$, we introduce  $t_\lambda:=(1-\lambda)t_0+\lambda t_1$  and the following transformation 
\begin{equation}\label{trans}
	\tau_i^\lambda(s): = t_\lambda+\frac{T-t_\lambda}{T-t_i}(s-t_i),\q s\in[t_i,T],\ i=0,1.
\end{equation}
Obviously, for $i=0,1$, $\tau_i^\lambda$ maps $ [t_i,T]$ to $[t_\lambda,T]$, and   $\dot{\tau}_i^\lambda = \displaystyle \frac{\mathrm d}{\mathrm  ds}  \tau_i^\lambda(s)  =  \frac{T-t_\lambda}{T-t_i}$.

In the following, we also use    the inverse time changes $\varrho_i^\lambda $ of $\tau_i^\lambda$, that is,  
$$\displaystyle \varrho_i^\lambda(s) = t_i+\frac{T-t_i}{T-t_\lambda}(s-t_\lambda),$$ 
which maps $ [t_\lambda,T] $ to $ [t_i,T],$   $ i=0,1$.

For the above time changes, we have the following results by the directly computation. It can also be referred to \cite{BHL-2012,J-2013}.

\begin{lemma}\label{tau_pro}\sl
	For any $\d\in(0,T)$, 
	there exists the constant $C_{\delta}>0$ only depending on $\delta$ and $T$, such that, for all $s\in [t_\lambda,T-\d]$,
	$$\ba{ll}
	\ns\ds 
	|\varrho_1^\lambda(s)-\varrho_0^\lambda(s)| + \Big|\frac{1}{\dot \tau_1^\lambda}-\frac{1}{\dot \tau_0^\lambda}\Big| + \Big|\frac{1}{\sqrt{\dot \tau_1^\lambda}}-\frac{1}{\sqrt{\dot \tau_0^\lambda}}\Big|  \les  C_\delta |t_1-t_0|,\\
	\ns\ds \lambda\Big|1-\frac{1}{\sqrt{\dot \tau_1^\lambda}}\Big|+(1-\lambda)\Big|1-\frac{1}{\sqrt{\dot \tau_0^\lambda}}\Big|\les \frac{1}{2\d}\lambda(1-\lambda)|t_1-t_0|, \\
	\ns\ds  \Big| \lambda \Big(1-\frac{1}{\sqrt{\dot \tau_1^\lambda}}\Big) +(1-\lambda)\Big(1-\frac{1}{\sqrt{\dot \tau_0^\lambda}}\Big)\Big|\les \frac1{8\d^2}\lambda(1-\lambda) |t_1-t_0|^2 ,\\
	\ns\ds \lambda \Big(1-\frac{1}{ {\dot \tau_1^\lambda}}\Big) =-(1-\lambda)   \Big(1-\frac{1}{ {\dot \tau_0^\lambda}}\Big) =\frac{\lambda(1-\lambda)}{T-t_\lambda}(t_1-t_0),\\
	\ns\ds \lambda\varrho_1^\lambda(s)+(1-\lambda)\varrho_0^\lambda(s)=s,\q s\in[t_\lambda,T].
	\ea $$
	
\end{lemma}

 Denote  $B_s^\lambda:=B_s-B_{t_\lambda}$, $s\in [t_\lambda,T]$
 and $N^{\lambda} $ by  the restriction of $N$  from $[0,T]\times E$ to  $[t_\lambda,T]\times E$.  The filtration generated by $B^\lambda $ and $ N^{\lambda}$ is denoted by  $\dbF^\lambda=\{\cF_s ^{\lambda}\}_{s\in[t_\lambda,T]}$.
Then, for $i=0,1$, $s\in[t_i,T],$ $A\in \cB(E),$  we introduce 
\begin{equation}\label{dbB-dbN}
	\begin{array}{ll}
\ns\ds    \dbB^{i }_s := \frac{1}{  \sqrt{\dot \tau_i^\lambda}  }  B_{\tau_i^\lambda(s)}^\lambda,  \q 	  	\dbN^{i } (p, (t_i,s] \times A ) :=  \tau_i^\lambda(N^{\lambda}) (p, (t_i,s] \times A )   =N^{\lambda}\big(p, (t_\lambda,\tau_i^\lambda(s)] \times A \big),   \\
\ns\ds g_{\t_i}:=\exp\Big\{-\ln\(\frac{T-t_\lambda}{T-t_i}\) \t_i^\lambda(N^{ \lambda})([t_i,T]\times E)+(t_i-t_\lambda)\n(E)\Big\},\q \dbQ_{\tau_i^\lambda}:=g_{\tau_i^\lambda} \dbP.
	\end{array}
\end{equation}
Note that, for $i=0,1$,  $\dbQ_{\tau_i^\lambda}$ is a new probability measure, which depends strongly on the following  additional condition on the measure $\n$.


\ss

{\bf (H$_3$)} $\n(E)<\i$.

\ss

\no Thus, from now on, {\bf (H$_3$)} is assumed.

   According to \cite{K-2001,K-2006,J-2013},    for $i=0,1$,
	$ \{\dbB^{i }_s \}_{s \in[t_i,T]  } $ is Brownian motion under the probability measures $\dbP$ and $\dbQ_{\t_i}$,  and
	the point process $\dbN^{i}$ defined on $[t_i,T] \times E$ is a Poisson random measure  under   $\dbQ_{\t_i}$, which    has the same distribution as $N^{\lambda}$ under $\dbP $. Moreover,   its compensator is still $\nu (\mathrm de)\mathrm ds  $.
 The compensated Poisson random measure for $\dbN^{i }$ under $\dbQ_{\t_i}$ is introduced as follows,
  $$\wt{\dbN }^{i}(\mathrm ds,\mathrm de):=\dbN^{ i} (\mathrm ds,\mathrm de)-\n(\mathrm de)\mathrm ds.$$
  Moreover, for $i=0,1$, $ \dbB^{i }  $ and $\dbN^{i } $ are independent under $\dbP$ and $\dbQ_{\t_i}$.
 In this case, we denote ${\bf F}^i=\{\sF_s ^{i}\}_{s\in[t_i,T]}$ by the filtration generated by $\dbB^i $ and $\dbN^{i}$, i.e., 
 $$\sF_s^{i}=\si\Big\{\dbB^i _r,\dbN^i ([t_i,r]\times A):\  r\in[t_i,s], \ A\in \cB(E) \Big\}\vee \cN_{\dbQ_{\t_i}}, \q s\in[t_i,T] .$$

 For convenience, for $\lambda\in[0,1]$ and  $i=0,1,$ we set $ \cU^{\lambda}_{t_\lambda,T}$ (resp., $\cU^{i}_{t_i,T}$) as the set of all  $U$-valued, $\dbF^\lambda$ (resp., ${\bf F}^i$)-predictable  stochastic processes on $ [t_\lambda,T ] $ (resp., $[ t_i,T] $).
Then, for any $u^\lambda(\cd)\in \cU^{\lambda}_{t_\lambda,T}$, we have
  $u^i(\cd):=u^\lambda \big(\t_i^\lambda(\cd) \big)\in \cU^{i}_{t_i,T}$, $i=0,1$.


   \subsection{The semi-concavity  of $W(\cd,\cd)$ in $(t,x)$}\label{semi-concavity}

First, we study  the semi-concavity  of the value function. Let us recall its definition as follows. 
		\begin{definition}\sl
			A function $\varphi:\mathbb R^n \to \mathbb R$ is said to be semi-concave, if there exists a constant $C\ges 0$ such that $ \varphi(x)-C|x|^2$ is concave in $x\in\dbR^n$.
		\end{definition}

	 For this, we need some additional assumptions.

\begin{description}
	\item[(H$_4$)] (i) $ b,\ \sigma,\ f$ and $ h$  are Lipschitz continuous with respect to $t\in [0,T]$;
	 for all $t, t' \in[0,T]$ and $(x,u,e) \in \dbR^n \times U \times E$,
	$ 
	|\gamma(t,x,u,e) -\gamma(t',x,u,e)| \les  C(1 \wedge|e|) |t-t'|;	
	$ \\
	(ii)  $ b,\ \sigma,\ f,\ h$ and $ \Phi$ are bounded in $(x,y,z)\in \dbR^n  \times \dbR \times \dbR^d $; 
	for all $x\in\dbR^n$ and $e\in E$,
	$ 
	|\gamma(t,x,u,e)| \les   C(1 \wedge|e|);	
	$ \\
	(iii)  $ b,\ \sigma,\ \gamma$ are differentiable in $t$ and $x$, and  
	$$
	\left\{ 
	\begin{array}{ll}
		\ns\ds\!\!\! 	|b_x(t,x,u)-b_x(t',x',u)|+	|\sigma_x(t,x,u)-\sigma_x(t',x',u)| \les C\big( |t-t'|+ |x-x'|\big) ,	\\
		\ns\ds\!\!\! 	|b_t(t,x,u)-b_t(t',x',u)|+	|\sigma_t(t,x,u)-\sigma_t(t',x',u)| \les C\big( |t-t'|+ |x-x'|\big),\\
		\ns\ds\!\!\! 	|\gamma_x(t,x,u,e)-\gamma_x(t',x',u,e)|	+ |\gamma_t(t,x,u,e)-\gamma_t(t',x',u,e)| \les    C(1 \wedge|e|)\big( |t-t'|+ |x-x'| \big),	\\
		%
		%
	\end{array}
	\right.
	$$
	where $C$ is the nonnegative constant;
	\item[(H$_5$)] $f$ is semi-concave in $(t,x,y,z,v)\in [0,T]\times \dbR^n \times \dbR \times \dbR^d \times \dbR$, uniformly with respect to $u\in U$; $\Phi$ is semi-concave in $x\in\dbR^n$.  
\end{description}

	\begin{remark}\sl
		 In the following, for $p\ges 2$, $\ds \int_E ( 1 \wedge|e| ) ^p \nu(\mathrm de)  < \i$ is needed, which holds naturally under \rm\textbf{(H$_3$)}.
		 For convenience, we always denote $\ell(e)=C(1 \wedge|e|),\ e\in E$  in the following, even with different $C$.
	\end{remark}

	\begin{theorem}\label{semicon}\sl
		Assume that the conditions {\rm \textbf{(H$_1$)-(H$_5$)}} and {\rm \textbf{(C)}}  hold true.
		 For  all $\delta>0$, the value function $W(\cdot,\cdot)$ is semi-concavity on $[0, T-\delta] \times \mathbb R^n$, that is, there exists some constant $C_\delta>0$ such that for any $(t_0,x_0),$ $(t_1,x_1) \in [0, T-\delta] \times \mathbb R^n$, and $\lambda\in[0,1]$,
		$$
		\lambda W(t_1,x_1) + (1-\lambda) W(t_0,x_0)
		\les W (t_\lambda,x_\lambda)+C_\delta \lambda(1-\lambda)\big(|t_0-t_1|^2+|x_0-x_1|^2\big),
		$$
		where $(t_\lambda,x_\lambda):=\lambda(t_1,x_1)+(1-\lambda)(t_0,x_0)$.
	\end{theorem}

Indeed, combined with the property of $W(\cd,\cd)$ in Lemma \ref{Le-Vn}, Theorem \ref{semicon} can be proved clearly if  
the following result is true.

	\bp\label{Pro-convex-W-n}\sl
	Let {\rm \textbf{(H$_1$)-(H$_5$)}} and {\rm \textbf{(C)}} hold true. Then, for  all $\delta>0$,  there exists some constant $C_\delta>0$ such that for any $(t_0,x_0),$ $(t_1,x_1) \in [0, T-\delta] \times \mathbb R^n$  and $\lambda\in[0,1]$,
		$$
		\lambda W^n(t_1,x_1) + (1-\lambda) W^n(t_0,x_0)
		\les W^n (t_\lambda,x_\lambda)+C_\delta \lambda(1-\lambda)\big(|t_0-t_1|^2+|x_0-x_1|^2\big)+\mathscr{A}^n,
		$$
	 with $\lim\limits_{n\to\i}\sA^n=0.$
	\ep
	%

The proof of Proposition \ref{Pro-convex-W-n} is very complex and occupies an extensive length. Therefore, we first make the following abbreviations,
$$\ba{ll}
\ns\ds \f^{\lambda,u}:= \f^{t_\lambda,x_\lambda;u },\q
 \f^{i,u}:= \f^{t_i,x_i;u},\ 
  \mbox{where }\f=X,Y,Z,V,A,\ i=0,1, \mbox{ resp.};\q  \cX^u:= \lambda X^{1,u}+(1-\lambda)X^{0,u} ;\\ 
  %
%
\ns\ds
  (\cY^u,\cZ^u,\cV^u,\cA^u):= (\lambda Y^{1,u}+(1-\lambda)Y^{0,u}  ,\lambda Z^{1,u}+(1-\lambda)Z^{0,u} ,\lambda V^{1,u}+(1-\lambda)V^{0,u}  ,\lambda A^{1,u}+(1-\lambda)A^{0,u}  ).
\ea$$

 \no For the above processes, we can recognize the  equations satisfied by them   as follows.
For $\lambda\in[0,1]$ and  $u^\lambda(\cd)\in \cU_{t_\lambda,T}^\lambda$,
	\begin{equation}\label{t-lambda-SDEP}
		\left\{
		\begin{array}{ll}
			\ns\ds\!\!\!  \mathrm dX_s^{\lambda,u^\lambda} =b(s,X_s^{\lambda,u^\lambda},u_s^\lambda)\mathrm ds+\sigma (s,X_s^{\lambda,u^\lambda},u_s^\lambda)\mathrm dB_s^\lambda
			+\int_{E}\gamma(s, X_{s-}^{\lambda,u^\lambda},u_s^\lambda,e)\tilde N ^\lambda(\mathrm ds,\mathrm de), \quad  s\in [t_\lambda,T],\\
			\ns\ds\!\!\!  X_{t_\lambda}^{\lambda,u^\lambda}=  x_\lambda, \quad  (t_\lambda,x_\lambda) \in [0,T-\delta] \times \mathbb{R}^n,
		\end{array}
		\right.
	\end{equation}
	and
	\begin{equation}\label{t-lambda-RBSDEP}
		\left\{
		\begin{array}{ll}
			\ns\ds \!\!\!\! 		{\rm(\romannumeral1)}\   (Y^{\lambda,u^\lambda} ,Z^{\lambda,u^\lambda} ,V^{\lambda,u^\lambda} ,A^{\lambda,u^\lambda} ) \in  \sS_ {\mathbb F^\lambda} ^2[t_\lambda,T];\\
			\ns\ds\!\!\!\! 			{\rm(\romannumeral2)}\     Y_s^{\lambda,u^\lambda} \! =   \Phi   ( X_T^{\lambda,u^\lambda})  
			+   \int_s^T   f \big(  r, X_r^{\lambda,u^\lambda}, Y_r^{\lambda,u^\lambda}, Z_r^{\lambda,u^\lambda}, \int _E l(e) V_r^{\lambda,u^\lambda}(e)  \nu(\mathrm de), u_r^\lambda \big) \mathrm dr 
			\\
			\ns\ds\!\!\!\! \hskip 1.9  cm -  ( A_T^{\lambda,u^\lambda}  - A_s^{\lambda,u^\lambda} )-   \int_s^T Z_r^{\lambda,u^\lambda} \mathrm dB_r^\lambda- \int_s^T   \int_E V_r^{\lambda,u^\lambda}(e) \tilde N^\lambda(\mathrm dr,\mathrm de) ,\quad  s\in [t_\lambda,T];\\
			\ns\ds\!\!\!\! 			{\rm(\romannumeral3)}\   Y_s^{\lambda,u^\lambda}  \leqslant   h  ( s,X_s^{\lambda,u^\lambda}),\q  \mbox{a.e. } s \in  [t_\lambda,T];\\
			\ns\ds\!\!\!\! 			{\rm(\romannumeral4)}\   \int_t^T \big(  h (s, X_{s-}^{\lambda,u^\lambda} )-Y_{s-}^{\lambda,u^\lambda} \big) \mathrm dA_s^{\lambda,u^\lambda} =0.
		\end{array}
		\right.
	\end{equation} 
For $i=0,1$ and $u^i(\cd)\in \cU_{t_i,T}^i$,
	\begin{equation}\label{ti-SDEP}
		\left\{
		\begin{array}{ll}
			\ns\ds\!\!\!  \mathrm dX_s^{i,u^i}=b(s,X_s^{i,u^i},u_s^i)\mathrm  ds+\sigma (s,X_s^{i,u^i},u_s^i)\mathrm d\dbB_s^{i}
			+\int_{E}\gamma(s,X_{s-}^{i,u^i},u_s^{i },e)\tilde \dbN ^{i }(\mathrm ds,\mathrm de), \quad  s\in [t_i,T],\\
			\ns\ds\!\!\!  X_{t_i}^{i,u^i}=x_i, \quad  (t_i,x_i) \in [0,T-\delta] \times \mathbb{R}^n,
		\end{array}
		\right.
	\end{equation}
and	 
	\begin{equation}\label{ti-RBSDEP}
	\left\{	\begin{array}{ll}
			\ns\ds \!\!\!\! 		{\rm(\romannumeral1)} \   (Y^{i,u^i},Z^{i,u^i},V^{i,u^i},A^{i,u^i}) \in \sS^2_ {{\bf F}^{i }} [t_i,T];\\
			\ns\ds\!\!\!\! 			{\rm(\romannumeral2)} \     Y_s^{i,u^i} \! =   \Phi   ( X_T^{i,u^i} )  
			+   \int_s^T   f \big(  r, X_r^{i,u^i}, Y_r^{i,u^i}, Z_r^{i,u^i}, \int _E l( e)  V_r^{i,u^i}(e)   \nu(\mathrm de), u_r^i \big) \mathrm dr 
			\\
			\ns\ds\!\!\!\! \hskip  1.82cm -  ( A_T^{i,u^i}  - A_s^{i,u^i} ) -   \int_s^T Z_r^{i,u^i} \mathrm d\dbB^i_r- \int_s^T   \int_E V_r^{i,u^i}(e) \tilde \dbN^{i } (\mathrm dr,\mathrm de) ,\quad  s\in [t_i,T];\\
			\ns\ds\!\!\!\! 			{\rm(\romannumeral3)} \   Y_s^{i,u^i}  \leqslant   h  ( s,X_s^{i,u^i} ),\q  \mbox{a.e. } s \in  [t_i,T];\\
			\ns\ds\!\!\!\! 			{\rm(\romannumeral4)} \   \int_t^T \big(  h (s, X_{s-}^{i,u^i} )-Y_{s-}^{i,u^i} \big) \mathrm dA_s^{i,u^i}=0.
		\end{array}\right.
 	\end{equation} 
	
Under our conditions, 	 the above equations are all well-posed.  Moreover, for  RBSDE with jumps \eqref{t-lambda-RBSDEP} and \eqref{ti-RBSDEP},  their    penalized equations are as follows,  for $n\in\dbN$, $\lambda\in[0,1],$
\begin{equation}\label{t-lambda_BSDEP}
		\left\{
		\begin{array}{ll}
			\ns\ds\!\!\!  \mathrm dY_s^{n,\lambda,u^\lambda} =    - f\big(  s, X_s^{\lambda,u^\lambda}, Y_s^{n,\lambda,u^\lambda},  Z_s^{n,\lambda,u^\lambda},  \int _E  l( e) V_s^{n,\lambda,u^\lambda} (e)   \nu(\mathrm de),u_s^\lambda  \big) \mathrm ds 
			+ \mathrm dA_s^{n,\lambda,u^\lambda}  +  Z_s^{n,\lambda,u^\lambda}   \mathrm dB_s^\lambda
		 \\
			\ns\ds\!\!\!  \hskip 1.825cm  +   \int_E   V_s^{n,\lambda,u^\lambda} (e)   \tilde N^\lambda(\mathrm ds,\mathrm de) ,\quad s \in [t_\lambda,T],\\
			\ns\ds\!\!\! Y_T^{n,\lambda,u^\lambda}=\Phi(X_T^{\lambda,u^\lambda});
		\end{array}
		\right.
	\end{equation} 
and for  $i=0,1$,  
	\begin{equation}\label{t-i_BSDEP}
		\left\{
		\begin{array}{ll}
			\ns\ds\!\!\!  \mathrm dY_s^{n,i,u^i} =    - f\big(  s, X_s^{i,u^i}, Y_s^{n,i,u^i},  Z_s^{n,i,u^i},  \int _E l( e) V_s^{n,i,u^i} (e) \nu(\mathrm de),u_s^i  \big) \mathrm ds + \mathrm d A_s^{n,i,u^i} +  Z_s^{n,i,u^i}   \mathrm d\dbB_s^i\\
			\ns\ds\!\!\!  \hskip 1.725cm  +   \int_E   V_s^{n,i,u^i} (e)   \tilde \dbN^i(\mathrm ds,\mathrm de) ,\quad s \in [t_i,T],\\
			\ns\ds\!\!\! Y_T^{n,i,u^i}=\Phi(X_T^{i,u^i}),
		\end{array}
		\right.
	\end{equation} 
where $\ds A_\cd^{n,\lambda,u^\lambda}:= n\int_{t_\lambda}^\cd   \big(  h( r ,X_r^{\lambda,u^\lambda})-Y_r^{n,\lambda,u^\lambda}  \big)^ - \mathrm dr$
and 
$\ds A_\cd^{n,i,u^i}:= n\int_{t_i}^\cd   \big( h( r ,X_r^{i,u^i} ) - Y_r^{n,i,u^i} \big)^ - \mathrm dr$.
Moreover, the  wellposedness of the above equations is obvious.

Recall the definition of $W^n(\cd,\cd)$ (in \eqref{W-n}), to prove Proposition \ref{Pro-convex-W-n}, 	we  need to prove formally 
	 \begin{equation}\label{est12}
	 	\lambda     Y_{t_1}^{n,1,u^1}  + (1-\lambda)   Y_{t_0}^{n,0,u^0} - Y_{t_\lambda}^{n,\lambda,u^\lambda}  \les C_\d \lambda (1-\lambda)\big( |t_1-t_0|^2 + |x_1-x_0|^2  \big)
	 	+   \sA^n,
	 \end{equation}
	 with $\lim\limits_{n\to\i}\sA^n=0$. 
	 Note  that, the processes involved in  \eqref{est12} are fixed with three different times, so that the direct proof of \eqref{est12} is infeasible.
	 Therefore,  for $i=0,1$,  by
introducing
$$\tilde X_s^{i,u^\lambda} :=X_{ \varrho_i^\lambda(s) }^{i,u^\lambda},\ 
\tilde \f_s^{n,i,u^\lambda} :=\f_{ \varrho_i^\lambda(s) }^{n,i,u^i}, \  \f= Y,V,A, \ \mbox{resp., and }
\tilde Z_s^{n,i,u^\lambda} := \frac{1}{\sqrt{\dot \tau_i^\lambda }} Z_{ \varrho_i^\lambda(s) }^{n,i,u^i},
\  s\in[t_\lambda,T],
\ \dbP\mbox{-a.s.},\\
$$
and applying  the  inverse time change to  SDE \eqref{ti-SDEP} and BSDE \eqref{t-i_BSDEP}, we get (note \eqref{dbB-dbN})
%
\begin{equation}\label{tilde_SDEP_i}
	\left\{
	\begin{array}{ll}
		\ns\ds\!\!\!  \mathrm d \tilde X_s^{i,u^\lambda} = \frac{1}{\dot \tau_i^\lambda} b \big( \varrho_i^\lambda(s), \tilde X_s^{i,u^\lambda}, u_s^\lambda \big) \mathrm ds
		+ \frac{1}{ \sqrt{\dot \tau_i^\lambda}}\sigma \big( \varrho_i^\lambda(s), \tilde X_s^{i,u^\lambda},u_s^\lambda \big) \mathrm dB_s^\lambda \\
		\ns\ds\!\!\!\hskip 1.55cm   +\int_{E}\gamma \big( \varrho_i^\lambda(s) , \tilde X_{s-}^{i,u^\lambda}, u_s^\lambda, e \big)
		\( \tilde N ^\lambda(\mathrm ds,\mathrm de)  +(1-\frac{1}{\dot \tau_i^\lambda})\nu(\mathrm de)\mathrm ds\), \quad  s\in [t_\lambda,T],\\
		\ns\ds\!\!\!  \tilde X_{t_\lambda}^{i,u^\lambda} = x_i, \quad  (t_\lambda,x_i) \in [0,T-\delta] \times \mathbb{R}^n,
	\end{array}
	\right.
\end{equation}
and 
\begin{equation}\label{tilde_BSDEP_in}
	\left\{
	\begin{array}{ll}
		\ns\ds\!\!\!  \mathrm d \tilde Y_s^{n,i,u^\lambda}
		=  \!  -\frac{1}{\dot\tau_i^\lambda}      f\big(  \varrho_i^\lambda(s), \tilde X_s^{i,u^\lambda}, \tilde Y_s^{n,i,u^\lambda},  \sqrt{\dot \tau_i^\lambda} \tilde Z_s^{n,i,u^\lambda},  \int _E l(e) \tilde V_s^{n,i,u^\lambda} (e) \nu(\mathrm de),u_s^\lambda  \big)        \mathrm ds+  \mathrm d \tilde A_s^{n,i,u^\lambda}  \  \\
		\ns\ds\!\!\!  \hskip 1.675cm  +   \tilde Z_s^{n,i,u^\lambda}   \mathrm dB_s^\lambda  + \int_E   \tilde V_s^{n,i,u^\lambda} (e)   \Big( \tilde  N^\lambda(\mathrm ds,\mathrm de) + (1-\frac{1}{\dot\tau_i^\lambda})\nu(\mathrm de)\mathrm ds  \Big), \quad s \in [t_\lambda,T],\\ 
		\ns\ds\!\!\!  \tilde Y_T^{n,i,u^\lambda}=\Phi(\tilde X_T^{i,u^\lambda}),
	\end{array}
	\right.
\end{equation}
where  $u_s^\lambda=u^i_{ \varrho_i^i(s)}$ and $\ds \tilde  A_s^{n,i,u^\lambda} = \frac{n}{\dot \tau_i^\lambda}  \int_{t_\lambda}^s   \big( h(  \varrho_i^\lambda(r) ,\tilde X_r^{i,u^\lambda} ) - \tilde Y_r^{n,i,u^\lambda} \big)^ - \mathrm dr$.

\ms 

By the above transformation, \eqref{est12} turns out to be equivalent to
\begin{equation}\label{est11}
	\lambda   \tilde Y_{t_\lambda}^{n,1,u^\lambda}  + (1-\lambda) \tilde Y_{t_\lambda}^{n,0,u^\lambda} - Y_{t_\lambda}^{n,\lambda,u^\lambda}  \les C_\d \lambda (1-\lambda)\big( |t_1-t_0|^2 + |x_1-x_0|^2  \big)
	+   \sA^n.
\end{equation}
%
%
Once \eqref{est11} is true,  Proposition \ref{Pro-convex-W-n} can be proved as follows.
Note that,
for any $n\ges 1$, $(t_0,x_0), (t_1,x_1),(t_\lambda,x_\lambda) \in [0,T-\d]\times \dbR^n$,  
\begin{equation}\label{est1111}
	\begin{array}{ll}
		\ns\ds\!\!\!  W^n(t_i,x_i )
		= \underset{u^i(\cd)\in \mathcal{U}_{t_i, T}  ^i }   \essinf   J_n(t_i,x_i;u^i(\cd) )  
		=\underset{u^\lambda(\cd)\in \mathcal{U}_{t_\lambda, T}  ^\lambda }   \essinf    \tilde Y_{t_\lambda}^{n,i,u^\lambda },\q i=0,1,\\
		\ns\ds\!\!\!  W^n(t_\lambda,x_\lambda )
		=\underset{u^\lambda(\cd)\in \mathcal{U}_{t_\lambda, T}  ^\lambda }   \essinf    J_n(t_\lambda,x_\lambda;u^\lambda(\cd) )  
		=\underset{u^\lambda(\cd)\in \mathcal{U}_{t_\lambda, T}  ^\lambda }   \essinf      Y_{t_\lambda}^{n,\lambda,u^\lambda }. \\
	\end{array} 
\end{equation}
Then, for any $\varepsilon >0$, there exists some $ u^{\lambda,\varepsilon}(\cd)\in \mathcal{U}_{t_\lambda, T}  ^\lambda $ such that
\begin{equation}\label{est2222}
	\begin{array}{ll}
		\ns\ds\!\!\! \lambda  W^n(t_1,x_1 )+ (1-\lambda) W^n(t_0,x_0 )  \les   \lambda  \tilde Y_{t_\lambda}^{n,1,u^{\lambda,\varepsilon}}  +(1-\lambda )  \tilde Y_{t_\lambda}^{n,0,u^{\lambda,\varepsilon}  }
		\\
		\ns\ds\!\!\! \les Y_{t_\lambda}^{n,\lambda,u^{\lambda,\varepsilon} }  + C_\d \lambda (1-\lambda)\big( |t_1-t_0|^2 + |x_1-x_0|^2  \big)
		+   \sA^n  \\
		\ns\ds\!\!\! \les \varepsilon+ W^n(t_\lambda,x_\lambda )  + C_\d \lambda (1-\lambda)\big( |t_1-t_0|^2 + |x_1-x_0|^2  \big)
		+  \sA^n.  \\
	\end{array}  
\end{equation}
Due to the arbitrariness  of $\e>0$, we complete the proof of Proposition \ref{Pro-convex-W-n}.
 
%
\newpage

From  now on, we focus on the study of \eqref{est11}. The proof of \eqref{est11} is complicated and needs some techniques. To make it clear, we make a flow chart to present the logical relationship of the auxiliary Lemmas as follows.
%
%
%
\begin{figure}[h] 
	\centering %
	\includegraphics[width=0.925\linewidth]{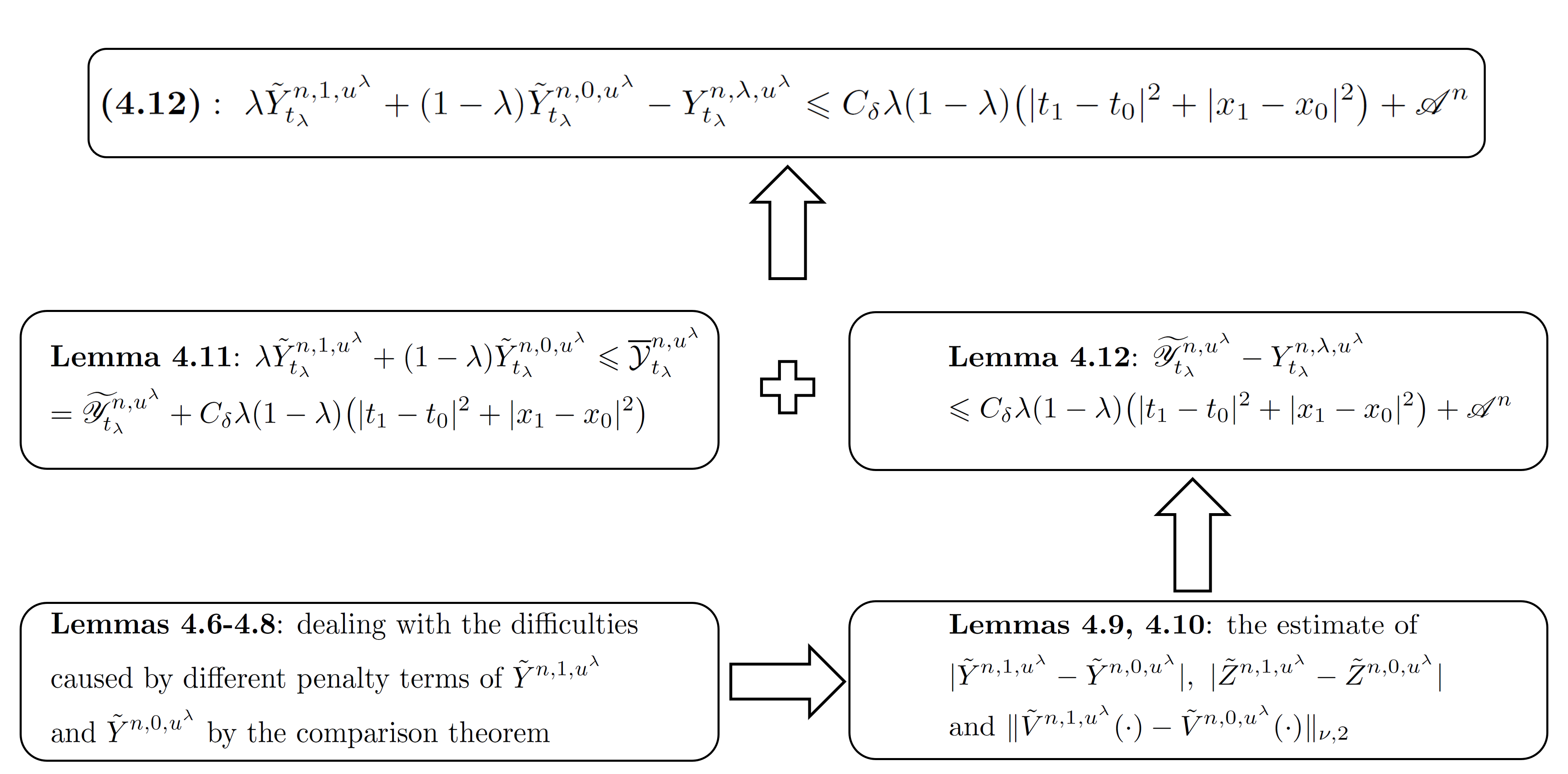} 
	\renewcommand{\figurename}{Figure}
	\caption{Relations among the auxiliary Lemmas} 
	\label{fig_lc} 
\end{figure}

Besides the above,
we point out that the study of $|\tilde X^{1,u^\lambda}-\tilde X^{0,u^\lambda}|$, $|\lambda\tilde X^{1,u^\lambda}+(1-\lambda)\tilde X^{0,u^\lambda} - X^{\lambda,u^\lambda} |$ in Lemmas \ref{CV-Le-X-1}, \ref{x-cx}  and the penalized equations \eqref{t-i_BSDEP} in Lemmas \ref{Y-esti}, \ref{Le-h-Y-0}  are all the basis of the above inferences.

\ms

Now 
 we study some properties of the  penalized equations \eqref{t-i_BSDEP}.

\begin{lemma}\label{Y-esti}\sl
		Under the conditions {\rm \textbf{(H$_1$)}, \textbf{(H$_2$)}}  and {\rm \textbf{(H$_4$)}-(ii)},  for all $n \geqslant 1$ and $i=1,2$,
		$$
			|Y_s^{n,i,u^i}|\leqslant C,\quad \forall s\in[t_i,T],\ \dbP\mbox{-a.s.} 
		$$
		Furthermore, for any $p \geqslant 1$, there exists some constant $C_p>0$  such that  
		$$
		\mathbb E ^{  \mathcal F^i_s}\Big[   \( \int_{s}^T |Z_r^{n,i,u^i}|^2 \mathrm dr \)^p	+ \( \int_s^T \|V_r^{n,i,u^i}(\cd)\|^2_{\n ,2 }  \mathrm dr \)^p  +\(  A_T^{n,i,u^i}-  A_s^{n,i,u^i}\)^{2p} 
   \Big]  \leqslant C_p, \quad  \forall s\in[t_i,T],\ \dbP\mbox{-a.s.},
		$$
		where the constant $C_p$ depends on $p$ as well as the bounds of $f$, $h$ and $\Phi$.
	\end{lemma}
	\begin{proof}
Firstly, similar to  the proof of Theorem 4.2 in \cite{E-2008}, we know the existence of some constant $C>0$ independent of $n$ such that
\begin{equation}\label{p=1}
	\mathbb E^{ \mathcal F^i_s} \Big[\sup\limits_{s\in [t_i,T]}|Y_s^{n,i,u^i}|^2+   \int_{s}^T \Big(|Z_r^{n,i,u^i}|^2 + \|V_r^{n,i,u^i}(\cd)\|^2_{\n ,2 } \Big) \mathrm dr   +  \( A_T^{n,i,u^i}-  A_s^{n,i,u^i}  \)^2 
	      \Big]  \leqslant C, \q  \dbP\mbox{-a.s.} 
		\end{equation}
Especially, we get, for all $n \geqslant 1$, $i=1,2$ and $s\in[t_i,T]$,
		$$
			|Y_s^{n,i,u^i}|\leqslant C,\quad   \dbP\mbox{-a.s.} 
		$$

Next,   according to \eqref{p=1} and the boundedness of $f$, for \eqref{t-i_BSDEP}, we obtain
	$$
\ba{ll}
\ns\ds   A_T^{n,i,u^i}-  A_s^{n,i,u^i} = \Phi (X_T^{i,u^i}) - Y_s^{n,i,u^i}    + \int_s^T   f\big(  r, X_r^{i,u^i}, Y_r^{n,i,u^i},  Z_r^{n,i,u^i}, \int_E  l( e) V_r^{n,i,u^i}(e)  \n(\mathrm de) ,u_r^{i }  \big)  \mathrm dr\\
			\ns\ds \hskip 2.9cm - \int_s^T Z_r^{n,i,u^i}   \mathrm d\dbB_r^i  - \int_s^T \int_E   V_r^{n,i,u^i} (e)   \tilde \dbN^i (\mathrm dr,\mathrm de) \\
\ns\ds \les C + \Big|\int_s^T Z_r^{n,i,u^i}   \mathrm d\dbB_r^i \Big|  + \Big|\int_s^T \int_E   V_r^{n,i,u^i} (e)   \tilde \dbN^i (\mathrm dr,\mathrm de) \Big|. \\
\ea$$
Thereby,  for all $p\ges 1$,	from the   Burkholder-Davis-Gundy inequality,
		\begin{equation}\label{pen1}
			\begin{array}{ll}
				\ns\ds\!\!\!	\mathbb E^{\mathcal F_s^i} \Big[ \(   A_T^{n,i,u^i}-  A_s^{n,i,u^i}  \)^{2p}   \Big]\\
				\ns\ds\!\!\!  
				\les	 C_p +C_p  \mathbb E^{\mathcal F_s^i}  \Big[ \(  \int_s^T |  Z_r^{n,i,u^i}|^2   \mathrm dr \)^{p }  + \(  \int_s^T \int_E  |  V_r^{n,i,u^i}(e)|^2 \dbN^i (\mathrm dr,\mathrm de)   \)^{p} \Big]
				,\q \dbP\mbox{-a.s.} 
			\end{array}
		\end{equation}

		Apply the It\^{o}'s formula to $|Y_s^{n,i,u^i}|^2$, we have
		\begin{equation}\label{est5}
			\begin{array}{ll}
				\ns\ds\!\!\!  |Y_s^{n,i,u^i}|^2 +   \int_s^T |Z_r^{n,i,u^i}|^2 \mathrm dr  + \int_s^T  \int_E |V_r^{n,i,u^i}(e)|^2  \dbN^i  (\mathrm dr,\mathrm de)    \\
				\ns\ds\!\!\!  =     |\Phi ( X_T^{i,u^i}) |^2   + 2 \!  \int_s^T \!   Y_r^{n,i,u^i}     f\big(  r, X_r^{i,u^i}, Y_r^{n,i,u^i},  Z_r^{n,i,u^i},  \int _E \! l(  e) V_r^{n,i,u^i} (e) \nu(\mathrm de),u_r^i  \big)   \mathrm dr
				  \\
				\ns\ds\!\!\!  \hskip0.425cm -2  \!  \int_s^T  \!  Y_r^{n,i,u^i}  \mathrm d A_r^{n,i,u^i}     -2  \int_s^T   Y_r^{n,i,u^i}  Z_r^{n,i,u^i} \mathrm d\dbB_r^i -2  \int_s^T  \! \int_E   Y_r^{n,i,u^i}  V_r^{n,i,u^i}(e) \tilde \dbN^i (\mathrm dr,\mathrm de), \q \dbP\mbox{-a.s.} \\
			\end{array}
		\end{equation}
 Then,  from  the inequalities \eqref{p=1} and \eqref{est5}, for all $p\ges 1$,
 	\begin{equation*} 
			\begin{array}{ll}
				\ns\ds\!\!\!    \mathbb E^{\mathcal F_s^i} \Big[ \Big( \int_s^T |Z_r^{n,i,u^i}|^2 \mathrm dr \Big) ^{2p}  +   \Big( \int_s^T  \int_E |V_r^{n,i,u^i}(e)|^2  \dbN^i (\mathrm dr,\mathrm de)  \Big) ^{2p}  \Big] \\
				\ns\ds\!\!\!  \les \mathbb E^{\mathcal F_s^i} \Big[ \Big( \int_s^T |Z_r^{n,i,u^i}|^2 \mathrm dr + \int_s^T  \int_E |V_r^{n,i,u^i}(e)|^2  \dbN^i (\mathrm dr,\mathrm de)  \Big) ^{2p} \Big]\\
				\ns\ds\!\!\! \les    C_p   
				+C_p   \mathbb E ^{\mathcal F_s^i}\Big[  \Big( A_T^{n,i,u^i}-A_s^{n,i,u^i}   \Big)^{2p} + \Big |   \int_s^T   Y_r^{n,i,u^i}  Z_r^{n,i,u^i} \mathrm d\dbB_r^i \Big|^{2p}+ \Big|   \int_s^T  \! \int_E   Y_r^{n,i,u^i}  V_r^{n,i,u^i}(e) \tilde \dbN^i (\mathrm dr,\mathrm de)\Big |^{2p} \Big]     \\
				\ns\ds\!\!\! \les C_p 
				 +C_p   \mathbb E^{\mathcal F_s^i} \Big[  \Big(  A_T^{n,i,u^i}-A_s^{n,i,u^i}      \Big)^{2p} +    \Big( \int_s^T     |Z_r^{n,i,u^i}|^2 \mathrm dr \Big) ^{  p}+   \Big( \int_s^T    \int_E     |V_r^{n,i,u^i}(e)|^2   \dbN^i (\mathrm dr,\mathrm de) \Big) ^{ p} \Big], \q \dbP\mbox{-a.s.}	\\
				%
			\end{array}
		\end{equation*}

Further, combined with \eqref{pen1}, we have 
		\begin{equation*}
			\begin{array}{ll}
				\ns\ds\!\!\! 
				\mathbb E ^{\mathcal F_s^i}\Big[ \Big( A_T^{n,i,u^i}-A_s^{n,i,u^i} \Big)^{4p} + \Big( \int_s^T |Z_r^{n,i,u^i}|^2 \mathrm dr \Big) ^{2p}  
				+   \Big( \int_s^T  \int_E |V_r^{n,i,u^i}(e)|^2  \dbN^i (\mathrm dr,\mathrm de)  \Big) ^{2p}      \Big]\\
				\ns\ds\!\!\!  \les C_p+  C_p \mathbb E^{\mathcal F_s^i} \Big[ \Big( \int_s^T |Z_r^{n,i,u^i}|^2 \mathrm dr \Big) ^{2p}  +   \Big( \int_s^T  \int_E |V_r^{n,i,u^i}(e)|^2  \dbN^i (\mathrm dr,\mathrm de)  \Big) ^{2p}  \Big]\\
				\ns\ds\!\!\! 
				\les   C_p +C_p   \mathbb E ^{\mathcal F_s^i}\Big[  \Big( A_T^{n,i,u^i}-A_s^{n,i,u^i}  \Big)^{2p} + \Big(  \int_s^T |  Z_r^{n,i,u^i}|^2   \mathrm dr \Big)^{p} + \Big(  \int_s^T \int_E  |  V_r^{n,i,u^i}(e)|^2 \dbN^i (\mathrm dr,\mathrm de)   \Big)^{p}  \Big]
				,\q \dbP\mbox{-a.s.}
			\end{array}
		\end{equation*}

		\no Consequently, due to \eqref{p=1}, for all integers of the form $p=2^k$ with $ k=0,1,2,\cds$,  
		$$
		\mathbb E ^{\mathcal F_s^i} \Big[ \Big( \int_s^T |Z_r^{n,i,u^i}|^2 \mathrm dr \Big) ^{p}  
		+   \Big( \int_s^T  \int_E |V_r^{n,i,u^i}(e)|^2  \dbN^i (\mathrm dr,\mathrm de)  \Big) ^p   +\Big(  A_T^{n,i,u^i}-A_s^{n,i,u^i}    \Big)^{2p}   \Big]
		\les C_p,\q \dbP\mbox{-a.s.}
		$$
		Hence, for any $p\ges 1$, the above conclusion  is valid.

		Finally, using Lemma 3.1 in \cite{LW-2014-lp}, for all $p\ges 1$,
		\begin{equation*}
			 \begin{array}{ll}
			 	\ns\ds\!\!\!  \mathbb E  ^{\mathcal F_s^i}\Big[   \( \int_{s}^T |Z_r^{n,i,u^i}|^2 \mathrm dr \)^p 
			 	+ \( \int_s^T \Vert V_r^{n,i,u^i}(\cd)\Vert ^2_{\n ,2 }     \mathrm dr \)^p  +\big(  A_T^{n,i,u^i}-A_s^{n,i,u^i}   \big)^{2p}     \Big] \\
			 	\ns\ds\!\!\!  \les  C \mathbb E ^{\mathcal F_s^i} \Big[ \( \int_s^T |Z_r^{n,i,u^i}|^2 \mathrm dr \) ^{p}   
			 	+   \( \int_s^T  \int_E |V_r^{n,i,u^i}(e)|^2  \dbN^i (\mathrm dr,\mathrm de)  \) ^p    +\Big(  A_T^{n,i,u^i}-A_s^{n,i,u^i}     \Big)^{2p}   \Big]   \les C_p,\q \dbP\mbox{-a.s.} \\
			 	%
			 \end{array}
		\end{equation*}
		\end{proof}

\ms
Similar to Lemma 6.1 in \cite{El-KPPQ},  the following result holds. For the readers'   convenience, we sketch the proof.
%
\bl\label{Le-h-Y-0}\sl   
Under {\rm \textbf{(H$_1$)}, \textbf{(H$_2$)}}, {\rm \textbf{(H$_4$)}-(ii)} and {\rm \textbf{(C)}},   for any $p\ges 2$,  $\dbP$-a.s.,
$$\ba{ll} 
\ns\ds {\rm (i)} \ \lim\limits_{n\to\i}  \dbE^{\mathcal F_s^i} \[\sup\limits_{r\in [s,T]} |(  h( r ,X_r^{i,u^i} ) -Y_r^{n,i,u^i} )^-|^p\]=0 ,\q s\in[t_i,T],\\
\ns\ds   {\rm (ii)} \ \lim\limits_{n\to\i} \dbE^{\mathcal F_s^\lambda}\[  \Big| \int_{s}^T (  h( \rho_i^\lambda(r) ,\tilde X_r^{i,u^\lambda} ) -\tilde Y_r^{n,i,u^\lambda} )^-\mathrm d \tilde A_r^{n, j ,u^\lambda}\Big|^\frac p2 \]=0,\q i,j=0,1,\ s\in[t_\lambda,T]. 
\ea$$

\el
\begin{proof}
It is obvious that, for all $ s\in[t_i,T]$, $\{Y_s^{n,i,u^i}\} $ is decreasing in $n$, thereby $Y_s^{n,i,u^i}\les Y_s^{0,i,u^i},\ \dbP$-a.s. Next we shall  prove $Y_s^{n,i,u^i}\les h( s ,X_s^{i,u^i} ), \ s\in [t_i,T],\ \dbP$-a.s. By the comparison theorem of BSDE with jumps, we know $Y_s^{n,i,u^i}\les \dbY_s^{n,i,u^i}$, where  
$$
		\left\{
		\begin{array}{ll}
			\ns\ds\!\!\!  \mathrm d\dbY_s^{n,i,u^i} =    - \[  f\big(  s, X_s^{i,u^i}, \dbY_s^{n,i,u^i},  \dbZ_s^{n,i,u^i},  \int _E  l( e)\dbV_s^{n,i,u^i} (e)   \nu(\mathrm de),u_s^i  \big)  
			-n\big(\dbY_s^{n,i,u^i}-h( s ,X_s^{i,u^i} )\big) \] \mathrm ds \\
			\ns\ds\!\!\!  \hskip 1.65cm    +  \dbZ_s^{n,i,u^i}   \mathrm d\dbB_s^i+   \int_E   \dbV_s^{n,i,u^i} (e)   \tilde \dbN^i(\mathrm ds,\mathrm de) ,\quad s \in [t_i,T],\\
			\ns\ds\!\!\! \dbY_T^{n,i,u^i}=\Phi(X_T^{i,u^i}).
		\end{array}
		\right.
$$

 Let $\vartheta$ be a stopping time such that $s\les \vartheta\les T$, then 
$$
		\begin{array}{ll}
			\ns\ds\!\!\!  \dbY_\vartheta^{n,i,u^i} = \dbE  ^{\mathcal F_\vartheta^i}\[ e^{-n(T-\vartheta)}\Phi(X_T^{i,u^i})+\int_\vartheta^T  e^{-n(r-\vartheta)}f\big(  r, X_r^{i,u^i}, \dbY_r^{n,i,u^i},  \dbZ_r^{n,i,u^i},  \int _E  l( e) \dbV_r^{n,i,u^i} (e)   \nu(\mathrm de),u_r^i  \big) \mathrm dr \\
\ns\ds\qq\qq \qq
 +\int_\vartheta^T  n  e^{-n(r-\vartheta)}  h( r ,X_r^{i,u^i} ) \mathrm dr \] .
		\end{array}
$$
As $n\to\i$, it is easy to check that,
$$
		\begin{array}{ll}
			\ns\ds\!\!\!     e^{-n(T-\vartheta)}\Phi(X_T^{i,u^i}) 
 +\int_\vartheta^T  n  e^{-n(r-\vartheta)}  h( r ,X_r^{i,u^i} ) \mathrm dr  \to \Phi(X_T^{i,u^i}) {\bf I}_{\{\vartheta=T\}}
 + h( \vartheta ,X_\vartheta^{i,u^i} )  {\bf I}_{\{\vartheta<T\}} ,
		\end{array}
$$
$\dbP$-a.s. and   in $L^2$ sense. Furthermore, due to $f$ is bounded, we also have 
$$ \ba{ll}
\ns\ds \dbE  ^{\mathcal F_\vartheta^i}\[ \int_\vartheta^T  e^{-n(r-\vartheta)}f\big(  r, X_r^{i,u^i}, \dbY_r^{n,i,u^i},  \dbZ_r^{n,i,u^i},  \int _E  l( e) \dbV_r^{n,i,u^i} (e)  \nu(\mathrm de),u_r^i  \big) \mathrm dr  \] \to 0, 
\ea
$$
  in $L^2$ sense. Consequently, we get $\dbY_\vartheta^{n,i,u^i} \to  \Phi(X_T^{i,u^i}) {\bf I}_{\{\vartheta=T\}}
 + h( \vartheta ,X_\vartheta^{i,u^i} )  {\bf I}_{\{\vartheta<T\}} $ in mean square and $Y_\cd^{n,i,u^i}\les h(   \cd ,X_ \cd^{i,u^i} ).$ Combined with the section theorem in Dellacherie and Meyer \cite{DM-1975}, we have  
  $Y_\cd^{n,i,u^i}\les h( \cd ,X_ \cd^{i,u^i} ) $,  $s\in [t_i,T],\ \dbP$-a.s.
 Therefore, $ ( h(  s ,X_ s^{i,u^i} )- Y_s^{n,i,u^i})^- \downarrow 0$, $s\in [t_i,T]$, $\dbP$-a.s., as $n\to\i$.
     Moreover, based on the continuity and monotonicity of power functions,  for any $p\ges2$, we have  $| ( h(  s ,X_ s^{i,u^i} )- Y_s^{n,i,u^i})^-|^p \downarrow 0$,   $s\in [t_i,T]$, $\dbP$-a.s., as $n\to\i$. 
  From the Dini Theorem, we know such a convergence is uniform in $t$. Further,  using the dominated convergence theorem, we get (i) holds true.
%

  Now, let's prove the second result. For any $p\ges 2$ and $i,j=0,1$,
 \begin{equation*}
 	\begin{array}{ll}
 		\ns\ds\!\!\!  	\dbE^{\mathcal F_s^\lambda}\[  \Big| \int_{s}^T (  h( \rho_i^\lambda(r) ,\tilde X_r^{i,u^\lambda} ) -\tilde Y_r^{n,i,u^\lambda} )^-\mathrm d \tilde A_r^{n, j ,u^\lambda}\Big|^\frac p2 \]  \\
 		\ns\ds\!\!\! 	\les \dbE^{\mathcal F_s^\lambda}\[    \sup_{r \in[s,T]} \Big|  (  h( \rho_i^\lambda(r) ,\tilde X_r^{i,u^\lambda} ) -\tilde Y_r^{n,i,u^\lambda} )^- \Big|^\frac p2  \cd \Big|   \tilde A_T^{n, j ,u^\lambda}-\tilde A_s^{n, j ,u^\lambda} \Big|^\frac p2 \] 	 \\
 		\ns\ds\!\!\! 	 \les  \Big(  \dbE^{\mathcal F_s^\lambda}\[    \sup_{r \in[s,T]} \Big|  (  h( \rho_i^\lambda(r) ,\tilde X_r^{i,u^\lambda} ) -\tilde Y_r^{n,i,u^\lambda} )^- \Big|^p  \]  \Big)^\frac 12
 		 \Big(  \dbE^{\mathcal F_s^\lambda}\[     \tilde A_T^{n, j ,u^\lambda}-\tilde A_s^{n, j ,u^\lambda} \Big|^p \] \Big)^\frac 12 \\
 		\ns\ds\!\!\! 	 \to 0,\q \mbox{as } n\to \i,
 		  \\
 	\end{array}
 \end{equation*}
 where  the last inequality have used  (i) and Lemma \ref{Y-esti}.
\end{proof}

Now we estimate the difference of  $\tilde X_\cd^{0,u^\lambda} $ and $ \tilde  X_\cd ^{1,u^\lambda} $.

\begin{lemma}\label{CV-Le-X-1}\sl
 	Suppose {\rm \textbf{(H$_1$)}}-{\rm \textbf{(H$_3$)}} and {\rm \textbf{(H$_4$)}-(i), (ii)}   hold. 
		 Then,   for all $p\ges 1$ and $s\in[t_\lambda,T]$, we have
		$$ 
		\begin{array}{ll}
			\ns\ds\!\!\!  	
			\mathbb E^{\mathcal F_{s }^\lambda  } \Big[ \underset{ \t \in [s   ,T] }{ \sup }  | \tilde X_\t^{1,u^\lambda} -   \tilde X_\t^{0,u^\lambda}  |^p +\int_s^T| \tilde X_r^{1,u^\lambda} -   \tilde X_r^{0,u^\lambda}  |^p\mathrm dr  \Big] \leqslant C_{T,p,\d}  \(|t_1-t_0|^p+ | \tilde X_s^{1,u^\lambda} -   \tilde X_s^{0,u^\lambda}  |^p \), \quad \dbP\mbox{-a.s.}
			%
		\end{array}
		$$
 
		%
	\end{lemma}
	
\begin{proof} 
	From \eqref{tilde_SDEP_i}, we know
  \begin{equation*}\label{ }
		\left\{
		\begin{array}{ll}
			\ns\ds\!\!\!  \mathrm d (\tilde X_s^{1,u^\lambda}-\tilde X_s^{0,u^\lambda}) = \(\frac{1}{\dot \tau_1^\lambda} b \big( \varrho_1^\lambda(s), \tilde X_s^{1,u^\lambda}, u_s^\lambda \big)- \frac{1}{\dot \tau_0^\lambda} b \big( \varrho_0^\lambda(s), \tilde X_s^{0,u^\lambda}, u_s^\lambda \big)\)\mathrm ds\\
\ns\ds\!\!\!  \q
			+\( \frac{1}{ \sqrt{\dot \tau_1^\lambda}}\sigma \big( \varrho_1^\lambda(s), \tilde X_s^{1,u^\lambda},u_s^\lambda \big)-\frac{1}{ \sqrt{\dot \tau_0^\lambda}}\sigma \big( \varrho_0^\lambda(s), \tilde X_s^{0,u^\lambda},u_s^\lambda \big) \)\mathrm dB_s^\lambda\\
			\ns\ds\!\!\!  \q  +\int_{E}\(\gamma \big( \varrho_1^\lambda(s) , \tilde X_{s-}^{1,u^\lambda}, u_s^\lambda, e \big)-\gamma \big( \varrho_0^\lambda(s) , \tilde X_{s-}^{0,u^\lambda}, u_s^\lambda, e \big)\)
		  \tilde N ^\lambda(\mathrm ds,\mathrm de)  \\
\ns\ds\!\!\!  \q   +\int_{E}\(\gamma \big( \varrho_1^\lambda(s) , \tilde X_{s-}^{1,u^\lambda}, u_s^\lambda, e \big)
			 (1-\frac{1}{\dot \tau_1^\lambda})-\gamma \big( \varrho_0^\lambda(s) , \tilde X_{s-}^{0,u^\lambda}, u_s^\lambda, e \big)
			 (1-\frac{1}{\dot \tau_0^\lambda})\)\nu(\mathrm de)\mathrm ds , \q  s\in [t_\lambda,T],\\
			\ns\ds\!\!\!  \tilde X_{t_\lambda}^{1,u^\lambda}- \tilde X_{t_\lambda}^{0,u^\lambda} = x_1-x_0. 
		\end{array}
		\right.
	\end{equation*} 

\no From Lemma \ref{tau_pro}, the boundedness and Lipschitz continuity of $b,\si,\g$, for any  $p\ges 1$, we have 
  \begin{equation}\label{b-1}
		\begin{array}{ll}
			\ns\ds\!\!\!   \dbE^{\cF_s^\lambda}  \[\sup\limits_{\t\in[s,T]}  \Big|\int_s^\t \(\frac{1}{\dot \tau_1^\lambda} b \big( \varrho_1^\lambda(r), \tilde X_{r}^{1,u^\lambda}, u_r^\lambda \big)- \frac{1}{\dot \tau_0^\lambda} b \big( \varrho_0^\lambda(r), \tilde X_{r}^{0,u^\lambda}, u_r^\lambda \big)\) \mathrm dr\Big|^p\]\\
%
%
\ns\ds\!\!\!    \les C_{T,p}\dbE^{\cF_s^\lambda}  \[ \! \int_s^T \Big| \(\! \frac{1}{\dot \tau_1^\lambda} \! - \! \frac{1}{\dot \tau_0^\lambda}\! \) b \big( \varrho_1^\lambda(r), \tilde X_{r}^{1,u^\lambda}, u_r^\lambda \big)
\!+\! \frac{1}{\dot \tau_0^\lambda}\( \!  b \big( \varrho_1^\lambda(r), \tilde X_{r}^{1,u^\lambda}, u_r^\lambda \big)\! -\! b \big( \varrho_0^\lambda(r), \tilde X_{r}^{0,u^\lambda}, u_r^\lambda \big) \! \)\Big|^p \mathrm dr  \! \] \\
%
%
%
\ns\ds \!\!\!   \les C_{T,p,\d}\dbE^{\cF_s^\lambda}  \[ \int_s^T\( |t_1-t_0|^p  +| \tilde X_{r}^{1,u^\lambda}- \tilde X_{r}^{0,u^\lambda}|^p   \) \mathrm dr\].
		\end{array}
	\end{equation} 
Using the  Burkholder-Davis-Gundy inequality, we get
\begin{equation*} 
		\begin{array}{ll}
			\ns\ds \dbE^{\cF_s^\lambda}\[ \sup\limits_{\t\in[s,T]} \Big|\int_s^\t\( \frac{1}{ \sqrt{\dot \tau_1^\lambda}}\sigma \big( \varrho_1^\lambda(r), \tilde X_{r}^{1,u^\lambda},u_r^\lambda \big)-\frac{1}{ \sqrt{\dot \tau_0^\lambda}}\sigma \big( \varrho_0^\lambda(r), \tilde X_{r}^{0,u^\lambda},u_r^\lambda \big) \)\mathrm dB_r^\lambda\Big|^p\]\\
%
		%
  \ns\ds\les \left\{ \ba{ll} \ns\ds \!\!\! C_{T,p,\d}\dbE^{\cF_t^\lambda}\[ \(\int_s^T\big( |t_1-t_0|^2 +| \tilde X_{r}^{1,u^\lambda}- \tilde X_{r}^{0,u^\lambda}|^2   \big)\mathrm d  r \)^\frac p2\],\q 1\les p< 2  ,\\
\ns\ds\!\!\!  C_{T,p,\d}  \dbE^{\cF_{t}}\[\int_{s}^T\big( |t_1-t_0|^p +| \tilde X_{r}^{1,u^\lambda}- \tilde X_{r}^{0,u^\lambda}|^p   \big)\mathrm d r\], \qq\q \  p\ges  2  , \ea\right. 
		\end{array}
	\end{equation*} 
and 
  \begin{equation*}\label{ }
		\begin{array}{ll}

			\ns\ds \dbE^{\cF_s^\lambda}\[ \sup\limits_{t\in[s,T]} \Big|\int_s^t\int_{E}\(\gamma \big( \varrho_1^\lambda(r) , \tilde X_{r-}^{1,u^\lambda}, u_r^\lambda, e \big)-\gamma \big( \varrho_0^\lambda(r) , \tilde X_{r-}^{0,u^\lambda}, u_r^\lambda, e \big)\)
		  \tilde N ^\lambda(\mathrm dr,\mathrm de) \Big|^p \]\\

\ns\ds \les\left\{ \ba{ll} \ns\ds \!\!\!  C_{T,p}\dbE^{\cF_s^\lambda}\[ \( \int_{s}^T\int_{E}|\gamma \big( \varrho_1^\lambda(r) , \tilde X_{r-}^{1,u^\lambda}, u_r^\lambda, e \big)-\gamma \big( \varrho_0^\lambda(r) , \tilde X_{r-}^{0,u^\lambda}, u_r^\lambda, e \big)|^2 \n(\mathrm de)\mathrm dr\)^\frac p2\],\  1\les p< 2  ,\\
 \ns\ds\!\!\!  \ba{ll} \ns\ds\!\!\!C_{T,p} \dbE^{\cF_s^\lambda}\[ \( \int_{s}^T\int_{E}|\gamma \big( \varrho_1^\lambda(r) , \tilde X_{r-}^{1,u^\lambda}, u_r^\lambda, e \big)-\gamma \big( \varrho_0^\lambda(r) , \tilde X_{r-}^{0,u^\lambda}, u_r^\lambda, e \big)|^2 \n(\mathrm de)\mathrm dr\)^\frac p2\\
 \ns\ds\!\!\!\hskip 1.45cm +  \int_{s}^T\int_{E}|\gamma \big( \varrho_1^\lambda(r) , \tilde X_{r-}^{1,u^\lambda}, u_r^\lambda, e \big)-\gamma \big( \varrho_0^\lambda(r) , \tilde X_{r-}^{0,u^\lambda}, u_r^\lambda, e \big)|^p \n(\mathrm de)\mathrm dr  \],\ea\q p\ges  2  ,\ea\right.  \\
 %
 %
  \ns\ds \les\left\{ \ba{ll} \ns\ds \!\!\!  C_{T,p,\d}\( \dbE^{\cF_s^\lambda}\[ \int_{s}^T    \(|t_1-t_0|^2 +| \tilde X_{r}^{1,u^\lambda}- \tilde X_{r}^{0,u^\lambda}|^2\)  \mathrm dr\]\)^\frac p2,\q 1\les p< 2  ,\\
 \ns\ds\!\!\!  \ba{ll} \ns\ds\!\!\! C_{T,p,\d} \dbE^{\cF_s^\lambda}\[  \int_{s}^T \(|t_1-t_0|^p +| \tilde X_{r}^{1,u^\lambda}- \tilde X_{r}^{0,u^\lambda}|^p\)   \mathrm dr  \],\ea\hskip0.95cm  p\ges  2  ,\ea\right.  
		\end{array}
	\end{equation*} 
and 
  \begin{equation}\label{gamma-2}
		\begin{array}{ll}	 
\ns\ds\!\!\!  \dbE^{\cF_s^\lambda}\[\sup\limits_{\t\in[s,T]}  \Big|\int_s^\t\int_{E}\(\gamma \big( \varrho_1^\lambda(r) , \tilde X_r^{1,u^\lambda}, u_r^\lambda, e \big)
			 (1-\frac{1}{\dot \tau_1^\lambda})-\gamma \big( \varrho_0^\lambda(r) , \tilde X_r^{0,u^\lambda}, u_r^\lambda, e \big)
			 (1-\frac{1}{\dot \tau_0^\lambda})\)\nu(\mathrm de)\mathrm dr \Big|^p\]\\

\ns\ds\!\!\! \les  \dbE^{\cF_s^\lambda}\[\Big( \int_{s}^T\int_{E} \Big| \(\frac{1}{\dot \tau_0^\lambda}-\frac{1}{\dot \tau_1^\lambda}\) \gamma \big( \varrho_1^\lambda(r) , \tilde X_r^{1,u^\lambda}, u_r^\lambda, e \big)\\
\ns\ds\!\!\! \qq\qq\qq\qq 
		+(1-\frac{1}{\dot \tau_0^\lambda})\(\gamma \big( \varrho_1^\lambda(r) , \tilde X_r^{1,u^\lambda}, u_r^\lambda, e \big)
			 -\gamma \big( \varrho_0^\lambda(r) , \tilde X_r^{0,u^\lambda}, u_r^\lambda, e \big)	 \)\Big|\nu(\mathrm de)\mathrm dr\Big)^p\]\\
%
%
\ns\ds\!\!\!   \les   C_{T,p,\d}\dbE^{\cF_s^\lambda}\[\int_{s}^T  \( | t_1 -t_0  |^p+ | \tilde X_{r}^{1,u^\lambda}- \tilde X_{r}^{0,u^\lambda}|^p \)\mathrm dr \]. 
		\end{array}
	\end{equation} 

Then, for all $  \t\in[s, T]$ and $p\ges 1$, we have 
  \begin{equation*} 
		\begin{array}{ll}
			\ns\ds\!\!\! \dbE^{\cF_s^\lambda}\[   |\tilde X_\t^{1,u^\lambda}-\tilde X_\t^{0,u^\lambda}|^p\]\\
			\ns\ds\!\!\! 
			\les  C_{T,p,\d} \(|\tilde X_s^{1,u^\lambda}-\tilde X_s^{0,u^\lambda}|^p +  | t_1 -t_0  |^p  +\dbE^{\cF_s^\lambda}\[\int_{s}^T     | \tilde X_{r}^{1,u^\lambda}- \tilde X_{r}^{0,u^\lambda}|^p\mathrm d  r \] \) 
			 \\
%
 %
 \ns\ds\!\!\!\q
 +  C_{T,p,\d} \dbE^{\cF_s^\lambda}\[ \(\int_s^T | \tilde X_{r}^{1,u^\lambda}- \tilde X_{r}^{0,u^\lambda}|^2  \mathrm d  r \)^\frac p2{\bf I}_{\{1\les p< 2 \}}\] \\
 \ns\ds\!\!\!   \les  C_{T,p,\d} \(|\tilde X_s^{1,u^\lambda}-\tilde X_s^{0,u^\lambda}|^p +  | t_1 -t_0  |^p  +\dbE^{\cF_s^\lambda}\[\int_{s}^T     | \tilde X_{r}^{1,u^\lambda}- \tilde X_{r}^{0,u^\lambda}|^p\mathrm d  r \] \)   \\
 \ns\ds\!\!\! 
 \q +  C_{T,p,\d} \(\dbE^{\cF_s^\lambda}\[\int_s^T | \tilde X_{r}^{1,u^\lambda}- \tilde X_{r}^{0,u^\lambda}|^2  \mathrm d  r \]\)^\frac p2{\bf I}_{\{1\les p< 2 \}}   . 
	\ea\end{equation*} 
 Therefore, for $ p\ges 2$,   using Gronwall's inequality, we get
  \begin{equation}\label{X-p>2}
		\begin{array}{ll}
			\ns\ds\!\!\!\sup\limits_{\t\in[s,T] } \dbE^{\cF_s^\lambda}\[   |\tilde X_\t^{1,u^\lambda}-\tilde X_\t^{0,u^\lambda}|^p\]\les  C_{T,p,\d} \(|\tilde X_s^{1,u^\lambda}-\tilde X_s^{0,u^\lambda}|^p +  | t_1 -t_0  |^p \). 
	\ea\end{equation} 
For $1\les p<2$, $\t\in[s,T]$ using \eqref{X-p>2} with $p=2$,   we have  
  $$
		\begin{array}{ll}
			\ns\ds\!\!\! \dbE^{\cF_s^\lambda}\[   |\tilde X_\t^{1,u^\lambda}-\tilde X_\t^{0,u^\lambda}|^p\]  
			\les  C_{T,p,\d} \(|\tilde X_s^{1,u^\lambda}-\tilde X_s^{0,u^\lambda}|^p +  | t_1 -t_0  |^p  +\dbE^{\cF_s^\lambda}\[\int_{s}^T     | \tilde X_{r}^{1,u^\lambda}- \tilde X_{r}^{0,u^\lambda}|^p\mathrm d  r\] \).
%
 %
	\ea
$$
 Using Gronwall's inequality again, \eqref{X-p>2}  also   holds for $1\les  p<  2$. 

Finally, from \eqref{b-1}-\eqref{gamma-2} and using \eqref{X-p>2} with $p\ges 1$,
   \begin{equation*} 
		\begin{array}{ll}
			\ns\ds\!\!\! \dbE^{\cF_s^\lambda}\[ \sup\limits_{\t\in[s,T]}  |\tilde X_\t^{1,u^\lambda}-\tilde X_\t^{0,u^\lambda}|^p\]\\
			\ns\ds\!\!\!  \les  C_{T,p,\d} \(|\tilde X_s^{1,u^\lambda}-\tilde X_s^{0,u^\lambda}|^p +  | t_1 -t_0  |^p  +\dbE^{\cF_s^\lambda}\[ \displaystyle \int_{s}^T     | \tilde X_{r}^{1,u^\lambda}- \tilde X_{r}^{0,u^\lambda}|^p\mathrm d  r \] \) 
			\\
			\ns\ds\!\!\!\hskip 0.425cm\displaystyle 
			+  C_{T,p,\d}  \(  \dbE^{\cF_s^\lambda}\[\int_s^T | \tilde X_{r}^{1,u^\lambda}- \tilde X_{r}^{0,u^\lambda}|^2  \mathrm d  r \]  \)^\frac p2 {\bf I}_{\{1\les p<2 \}}\\
  \ns\ds\!\!\! \les   C_{T,p,\d} \(|\tilde X_s^{1,u^\lambda}-\tilde X_s^{0,u^\lambda}|^p +  | t_1 -t_0  |^p \), 
	\ea\end{equation*} 
 holds for $p\ges 1$. The desired result is proved.
\end{proof}


Next, we denote by $\tilde\cX_\cd^{u^\lambda}$ the convex combination  of $\tilde X_\cd^{0,u^\lambda}$ and  $\tilde X_\cd^{1,u^\lambda}$, i.e.,  
  $\tilde\cX_s^{u^\lambda}:= \lambda \tilde X_s^{1,u^\lambda} +(1-\lambda)\tilde X_s^{0,u^\lambda} ,$ $s\in[t_\lambda,T]$.
The estimate of $|\tilde\cX _\cd^{u^\lambda} -   X_\cd^{\lambda,u^\lambda}|$ is given as follows.

 	\begin{lemma}\label{x-cx}\sl
		  Suppose {\rm \textbf{(H$_1$)}}-{\rm \textbf{(H$_4$)}} hold,  
 for all $    p\ges 1$ and $s\in[t_\lambda,T]$, we have, $\dbP$-a.s.,
 $$
\dbE^{\cF_{s}^\lambda}\[ \!\sup\limits_{\t\in[s,T]}\! | \tilde\cX _\t^{u^\lambda}-X_\t^{\lambda,u^\lambda}|^p  \]  
 \!\les\!  C_{T,p,\d} \Big(  |t_1-t_0 |^p  + | \tilde\cX _s^{u^\lambda} \! -  \! X_s^{\lambda,u^\lambda}|^p  \Big)  
 +  C_{T,p,\d}\lambda^p(1-\lambda)^p \(|t_1-t_0|^{2p} + | \tilde X_s^{1,u^\lambda}  \!-\!   \tilde X_s^{0,u^\lambda}  |^{2p}\).
$$
		%
	\end{lemma}
	
\begin{proof} First, $\tilde\cX_\cd  ^{u^\lambda} -   X_\cd ^{\lambda,u^\lambda} $ satisfies the following equation
	\begin{equation*} 
		\begin{array}{ll}
			\ns\ds\!\!\!   \tilde\cX _s^{u^\lambda}-X_s^{\lambda,u^\lambda}  =\lambda \tilde X^{1,u^\lambda}_s+(1-\lambda)\tilde X^{0,u^\lambda} _s-X_s^{\lambda,u^\lambda}  \\
\ns\ds\!\!\! =\int_{t_\lambda}^s\(J_b(r)+b(r,\tilde\cX_r^{ u^\lambda} ,u_r^\lambda)-b(r,X_r^{\lambda,u^\lambda},u_r^\lambda)\)\mathrm dr 
 +\int_{t_\lambda}^s  \(J_\si(r)+\si(r,\tilde\cX_r^{u^\lambda} ,u_r^\lambda)-\sigma (r,X_r^{\lambda,u^\lambda},u_r^\lambda)\)\mathrm dB_r^\lambda \\
%
\ns\ds\!\!\! \q 	    
			+\int_{t_\lambda}^s\!\int_{E}\(J_\g(r,e) + \gamma(r,\tilde\cX_{r-}^{u^\lambda},u_r^\lambda,e)-\gamma(r,X_{r-}^{\lambda,u^\lambda},u_r^\lambda,e)\)\tilde N ^\lambda(\mathrm dr,\mathrm de)\\
\ns\ds\!\!\!\q   +\int_{t_\lambda}^s\!\int_{E}\(\lambda(1-\frac{1}{\dot \tau_1^\lambda})\gamma \big( \varrho_1^\lambda(r) , \tilde X_{r}^{1,u^\lambda}, u_r^\lambda, e \big)
			 +(1-\lambda)(1-\frac{1}{\dot \tau_0^\lambda})\gamma \big( \varrho_0^\lambda(r) , \tilde X_{r}^{0,u^\lambda}, u_r^\lambda, e \big)
			\)\nu(\mathrm de)\mathrm dr ,
		\end{array}
	\end{equation*}
 where $ s\in [t_\lambda,T]$ and
   $$\ba{ll}
 \ns\ds J_b(r):=\frac{\lambda }{\dot \tau_1^\lambda} b \big( \varrho_1^\lambda(r), \tilde X_{r}^{1,u^\lambda}, u_r^\lambda \big)+   \frac{1-\lambda}{\dot \tau_0^\lambda} b \big( \varrho_0^\lambda(r), \tilde X_{r}^{0,u^\lambda}, u_r^\lambda \big)- b(r,\tilde\cX_r^{u^\lambda} ,u_r^\lambda),\\
  \ns\ds J_\si(r):= \frac{\lambda}{ \sqrt{\dot \tau_1^\lambda}}\si\big( \varrho_1^\lambda(r), \tilde X_{r}^{1,u^\lambda}, u_r^\lambda \big)+   \frac{1-\lambda}  { \sqrt{\dot \tau_0^\lambda}}  \si \big( \varrho_0^\lambda(r), \tilde X_{r}^{0,u^\lambda}, u_r^\lambda \big)-  \si(r,\tilde\cX_r^{u^\lambda} ,u_r^\lambda),\\
 \ns\ds 
 J_\g(r,e):=\lambda\gamma \big( \varrho_1^\lambda(r) , \tilde X_{r-}^{1,u^\lambda}, u_r^\lambda, e \big)+(1-\lambda)\gamma \big( \varrho_0^\lambda(r) , \tilde X_{r-}^0, u_r^\lambda, e \big) -\gamma(r,\tilde \cX_{r-}^{\lambda,u^\lambda},u_r^\lambda,e),\q r\in [t_\lambda,T],\  e\in E.\ea$$

From Lemma \ref{tau_pro}, for any $X_1,X_0\in\dbR^n$ and $u\in U$, we have  (set $X_\lambda:=\lambda X_1+(1-\lambda)X_0$)
$$\ba{ll}
 \ns\ds\!\!\! \Big| \frac{\lambda }{\dot \tau_1^\lambda}b\big( \varrho_1^\lambda(r), X_1, u   \big)+   \frac{1-\lambda}{\dot \tau_0^\lambda}b \big( \varrho_0^\lambda(r),X_0, u  \big)- b(r, X_\lambda  ,u )\Big|\\
  \ns\ds\!\!\!  \les \Big| {\lambda }  b \big( \varrho_1^\lambda(r),X_1, u  \big)+  (1-\lambda) b \big( \varrho_0^\lambda(r),X_0, u  \big)-   b(r,X_\lambda   ,u )\Big| \\
  
  \ns\ds\!\!\! \q +\Big| (\frac{1}{\dot \tau_1^\lambda} -1)\lambda b \big( \varrho_1^\lambda(r), X_1, u\big)+   ( \frac{1}{\dot \tau_0^\lambda}-1)(1-\lambda) b \big( \varrho_0^\lambda(r), X_0, u\big)\Big|\\
  %
  %
  \ns\ds\!\!\!  \les C\lambda(1-\lambda)   |  \varrho_1^\lambda(r)- \varrho_0^\lambda(r)     |   \cd \int_0^1    |b_s\big(r+\eta(1-\lambda)(\varrho_1^\lambda(r)- \varrho_0^\lambda(r) ),  X_\lambda  +\eta(1-\lambda)(   X_1-X_0), u\big)\\
  \ns\ds\!\!\! \hskip 5.15 cm-b_s\big(r-\eta\lambda(\varrho_1^\lambda(r)- \varrho_0^\lambda(r) ), X_\lambda  -\eta \lambda ( X_1-X_0), u \big) \big|\mathrm d\eta      \\
    \ea$$ 
$$\ba{ll}
   \ns\ds\!\!\!  \q+ C\lambda(1-\lambda)|  X_1-X_0| \cd \int_0^1   |b_x\big(r+\eta(1-\lambda)(\varrho_1^\lambda(r)- \varrho_0^\lambda(r) ),  X_\lambda  +\eta(1-\lambda)(   X_1-X_0), u\big)\\
  \ns\ds\!\!\! \hskip 4.75cm-b_x\big(r-\eta\lambda(\varrho_1^\lambda(r)- \varrho_0^\lambda(r) ), X_\lambda  -\eta \lambda ( X_1-X_0), u \big) \big|\mathrm d\eta      \\
  \ns\ds\!\!\! \q + \frac{\lambda(1-\lambda)}{T-t_\lambda}|t_1-t_0|  \big(  |\varrho_1^\lambda(r)- \varrho_0^\lambda(r)| +|  X_1-X_0| \big)\\
  %
      %
    %
      %
  \ns\ds \!\!\!\les C\lambda(1-\lambda)  \big(|t_1-t_0|^2+   |  X_1-X_0|^2 \big) .\ea$$ 
Therefore, \begin{equation}\label{Jb} 
  |J_b(r)| \les C\lambda(1-\lambda)  \(|t_1-t_0|^2+   | \tilde X_{r}^{1,u^\lambda}-\tilde X_{r}^{0,u^\lambda}|^2 \),\q r\in[t_\lambda,T] .   
           \end{equation}
Similarly, using Lemma \ref{tau_pro}-(i), (ii) and (iii), we also obtain
\begin{equation}\label{J-si-g} \ba{ll}
 \ns\ds |J_\si(r)| \les C\lambda(1-\lambda)  \Big(|t_1-t_0|^2+   | \tilde X_{r}^{1,u^\lambda}-\tilde X_{r}^{0,u^\lambda}|^2 \Big) ,\\
 \ns\ds |J_\g(r,e)| \les C \bar \ell(e)\lambda(1-\lambda)  \Big(|t_1-t_0|^2+   | \tilde X_{r}^{1,u^\lambda}-\tilde X_{r}^{0,u^\lambda}|^2 \Big),\q r\in[t_\lambda,T],\ e\in E.
  \ea    \end{equation}

 Based on the above estimates,  using the methods used in \eqref{b-1}-\eqref{gamma-2},  for $p\ges 1$, we get
 \begin{equation*} 
		\begin{array}{ll}
			\ns\ds  \dbE^{\cF_{s}^\lambda}\[\sup\limits_{\t\in [s,T]}\Big|\int_{s}^\t\(J_b(r)+b(r,\tilde\cX_r^{u^\lambda} ,u_r^\lambda)-b(r,X_r^{\lambda,u^\lambda},u_r^\lambda)\)\mathrm dr  \Big|^p\]  \\
%
\ns\ds \les  C_{T,p,\d}\dbE^{\cF_{s}^\lambda}\[ \int_{s}^T\(|J_b(r)|^p+| \tilde\cX_r^{u^\lambda} -X_r^{\lambda,u^\lambda}   |^p\)\mathrm dr \] ,
		\end{array}
	\end{equation*}
and 
\begin{equation*} 
		\begin{array}{ll}
		\ns\ds 	 \dbE^{\cF_{s}^\lambda}\[\sup\limits_{\t\in [s,T]}\Big|\int_{s}^\t\(J_\si(r)+\si(r,\tilde\cX_r^{u^\lambda} ,u_r^\lambda)-\sigma (r,X_r^{\lambda,u^\lambda},u_r^\lambda)\)\mathrm dB_r^\lambda\Big|^p\]\\%
%
\ns\ds\les \left\{ \ba{ll} \ns\ds \!\!\! C_{T,p,\d}\Big(\dbE^ {\cF_{s}^\lambda}\[\int_{s}^T \(|J_\si(r)|^2+|\tilde\cX_r^{u^\lambda}   -X_r^{\lambda,u^\lambda} |^2\)\mathrm d r\]\Big)^\frac p2,\q 1\les p<2  ,\\
\ns\ds \!\!\!  C_{T,p,\d}\dbE^ {\cF_{s}^\lambda}\[\int_{s}^T \(|J_\si(r)|^p+|\tilde\cX_r^{u^\lambda}   -X_r^{\lambda,u^\lambda} |^p\)\mathrm d r\], \qq\q   p\ges  2  , \ea\right.
		\end{array}
	\end{equation*}
and 
\begin{equation*} 
		\begin{array}{ll}
			\ns\ds \dbE^{\cF_{s}^\lambda}\[\sup\limits_{\t\in [s,T]}\Big|\int_{s}^\t\int_{E}\(J_\g(r,e) + \gamma(r,\tilde\cX_{r-}^{u^\lambda},u_r^\lambda,e)-\gamma(r,X_{r-}^{\lambda,u^\lambda},u_r^\lambda,e)\)\tilde N ^\lambda(\mathrm dr,\mathrm de)\Big|^p\]\\
 \ns\ds \les\left\{ \ba{ll} \ns\ds \!\!\!   C_{T,p,\d}\( \dbE^{\cF_{s}^\lambda}\[ \int_{s}^T\int_{E}\(|J_\g(r,e)|^2 +\ell^2(e)|  \tilde\cX_{r }^{u^\lambda}-X_{r }^{\lambda,u^\lambda} |^2\)  \n(\mathrm de)\mathrm dr\]\)^\frac p2,\q 1\les p<2  ,\\
 \ns\ds\!\!\!  \ba{ll} \ns\ds\!\!\! C_{T,p,\d}\dbE^{\cF_{s}^\lambda}\[ \( \int_{s}^T\int_{E} \( |J_\g(r,e)|^2 
 +\ell^2(e)|  \tilde\cX_{r }^{u^\lambda}-X_{r }^{\lambda,u^\lambda} |^2 \) \n(\mathrm de)\mathrm dr \)^\frac p2     \\
 \ns\ds\!\!\! \hskip1.7cm +  \int_{s}^T\int_{E}\(|J_\g(r,e)|^p +\ell^p(e)|  \tilde\cX_{r }^{u^\lambda}-X_{r }^{\lambda,u^\lambda} |^p\)  \n(\mathrm de)\mathrm dr  \],\ea\hskip0.65cm  p\ges 2  ,\ea\right.  \\
		\end{array}
	\end{equation*}
and 
\begin{equation*} 
		\begin{array}{ll}
			\ns\ds\dbE^{\cF_{s}^\lambda}\[\sup\limits_{\t\in [s,T]}\Big|\int_{s}^\t\int_{E}\(\lambda(1-\frac{1}{\dot \tau_1^\lambda})\gamma \big( \varrho_1^\lambda(r) , \tilde X_{r}^{1,u^\lambda}, u_r^\lambda, e \big)
			 +(1-\lambda)(1-\frac{1}{\dot \tau_0^\lambda})\gamma \big( \varrho_0^\lambda(r) , \tilde X_{r}^{0,u^\lambda}, u_r^\lambda, e \big)
			\)\nu(\mathrm de)\mathrm dr\Big|^p\]\\
%
%
%
\ns\ds   \les C_{T,p,\d}\lambda^p(1-\lambda)^p \dbE^{\cF_{s}^\lambda}\[\Big|\int_{t}^s\int_E  \ell(e)\big( | t_1 -t_0  | + | \tilde X_{r-}^{1,u^\lambda}- \tilde X_{r-}^{0,u^\lambda}|  \big) \nu(\mathrm de)\mathrm dr\Big|^p\] \\
 \ns\ds \les  C_{T,p,\d} \lambda^p(1-\lambda)^p\dbE^{\cF_{s}^\lambda}\[ \int_{t}^s  \big( | t_1 -t_0  |^p + | \tilde X_{r}^{1,u^\lambda}- \tilde X_{r}^{0,u^\lambda}|^p  \big)\mathrm dr  \]. 
		\end{array}
	\end{equation*}
%

Similar to the  proof of Lemma \ref{CV-Le-X-1}, for $p\ges 1$, we obtain
 %
  \begin{equation*} 
		\begin{array}{ll}
			\ns\ds\!\!\! \dbE^{\cF_{s}^\lambda}\[ \sup\limits_{\t\in[s,T]}| \tilde\cX _\t^{u^\lambda}-X_\t^{\lambda,u^\lambda}|^p\]   \\
\ns\ds\!\!\!    \les C_p | \tilde\cX _s^{u^\lambda}-X_s^{\lambda,u^\lambda}|^p +C_{T,p,\d}\dbE^{\cF_{s}^\lambda}\[ \int_{s}^T\(|J_b(r)|^p+|J_\si(r)|^p+\int_{E}\ |J_\g(r,e)|^p\n(\mathrm de) \)\mathrm dr \] \\
\ns\ds\!\!\!\q  +  C_{T,p,\d}\Big(\dbE^ {\cF_{s}^\lambda}\[\int_{s}^T \(|J_\si(r)|^2+\int_E |J_\g(r,e)|^2 \nu(\mathrm de)\)\mathrm d r\]\Big)^\frac p2\\
\ns\ds\!\!\!\q  + C_{T,p,\d} \lambda^p(1-\lambda)^p\dbE^{\cF_{s}^\lambda}\[ \int_{t}^s  \big( | t_1 -t_0  |^p + | \tilde X_{r}^{1,u^\lambda}- \tilde X_{r}^{0,u^\lambda}|^p  \big)\mathrm dr  \]\\
\ns\ds\!\!\!   \les  C_{T,p,\d} \big(| t_1 -t_0  |^p+ |\tilde\cX _s^{u^\lambda}-X_s^{\lambda,u^\lambda}|^p\big)  +  C_{T,p,\d}\lambda^p(1-\lambda)^p \(|t_1-t_0|^{2p} + | \tilde X_s^{1,u^\lambda} -   \tilde X_s^{0,u^\lambda}  |^{2p}\),
		\end{array}
	\end{equation*}
 where we have used \eqref{Jb}, \eqref{J-si-g} and Lemma \ref{CV-Le-X-1}. 
\end{proof}	
	 
	For convenience, we denote 
	 \begin{equation}\label{D}
	  \D_s:=C_\d\sup\limits_{r\in [t_\lambda,s]}\big(|t_1-t_0|  +|\tilde X_r^{1,u^\lambda}-\tilde X_r^{0,u^\lambda}|\big).
	 \end{equation} 
	   Without loss of generality, we also assume $t_0 \les t_1 $. Otherwise, the following study is made for $\tilde  Y_\cd^{n,1,u^\lambda}$.

\bl\label{Le-tilde-Y<bar}\sl Suppose {\rm \textbf{(H$_1$)}}, {\rm \textbf{(H$_2$)}}, {\rm \textbf{(H$_4$)}-(i), (ii)} and {\rm \textbf{(C)}}   hold, for any $s\in[t_\lambda,T]$, we have  
$$\tilde  Y_s^{n,0,u^\lambda}\les \overline   Y_s^{n,0,u^\lambda},\q \dbP\mbox{-a.s.},$$
where 
\begin{equation}\label{bar-Y-n0}
		\left\{
		\begin{array}{ll}
			\ns\ds\!\!\!  \mathrm d \overline   Y_s^{n,0,u^\lambda} \! =  \!  - \Big[ \!     \frac{1}{\dot \tau_1^\lambda }   f\Big(  \varrho_1^\lambda(s), \tilde X_s^{1,u^\lambda}, \overline   Y_s^{n,0,u^\lambda}-\D_s,  \sqrt{\dot \tau_1^\lambda} \overline   Z_s^{n,0,u^\lambda},  \int _E l(e)  \overline  V_s^{n,0,u^\lambda} (e) \nu(\mathrm de),u_s^\lambda  \Big) 
			\\
			\ns\ds\!\!\!   \hskip 1.98cm    -    \frac{n}{\dot \tau_1^\lambda }  \Big(  h( \varrho_1^\lambda(s), \tilde X_s^{1,u^\lambda} )+\D_s-\overline   Y_s^{n,0,u^\lambda}  \Big)^ -  
			-  (1-\frac{1}{\dot \tau_1^\lambda})  \int_E   \overline   V_s^{n,0,u^\lambda} (e)  \nu(\mathrm de) \\
			\ns\ds\!\!\!   \hskip 1.98cm    + C_\delta |t_1-t_0| \cd \Big( \Big|  \overline   Z_s^{n,0,u^\lambda} \Big| + \Big| \int_E  \overline   V_s^{n,0,u^\lambda}(e)\nu(\mathrm de) \Big|  \Big)   +  \D_s    \Big]  \mathrm ds\\
			\ns\ds\!\!\!   \hskip 1.65cm     +  \overline   Z_s^{n,0,u^\lambda}   \mathrm dB_s^\lambda  + \int_E   \overline   V_s^{n,0,u^\lambda} (e)   \tilde  N^\lambda(\mathrm ds,\mathrm de)  , \quad s \in [t_\lambda,T],\\
			\ns\ds\!\!\!  \overline Y_T^{n,0,u^\lambda}=\Phi(\tilde X_T^{1,u^\lambda})+  \D_T.
		\end{array}
		\right.
	\end{equation} 
 
\el

\begin{proof}	For all $s\in[t_\lambda,T]$, $(y,z,v(\cd))\in \dbR\times \dbR^d\times \cL_\n^2(E;\dbR)$, we have 
 \begin{equation*} 
		\begin{array}{ll}
\ns\ds  \frac{1}{\dot \tau_0^\lambda}      f\Big(  \varrho_0^\lambda(s), \tilde X_s^{0,u^\lambda}, y,  \sqrt{\dot \tau_0^\lambda}z,  \int _E l( e)  v (e) \nu(\mathrm de),u_s^\lambda  \Big) - (1-\frac{1}{\dot \tau_0^\lambda}) \int_E  v(e)   \nu(\mathrm de)  
		   \\
%
%
\ns\ds \les   \frac{1}{\dot \tau_1^\lambda }       f\Big(  \varrho_1^\lambda(s), \tilde X_s^{1,u^\lambda}, y- \D _s,  \sqrt{\dot \tau_1^\lambda} z,  \int _E l(  e)  v (e) \nu(\mathrm de),u_s^\lambda  \Big)    - (1-\frac{1}{\dot \tau_1^\lambda})  \int_E  v (e)   \nu(\mathrm de) \\
\ns\ds\q +   C_\delta |t_1-t_0| \cd \Big(  |z|+  \Big|\int_Ev(e)\n(\mathrm de)  \Big|\Big) 
+   \D_s ,
		\end{array}
	\end{equation*} 
where we have used the Lipschitz continuity of $f$ and Lemma \ref{tau_pro}. Furthermore, 
  $$\ba{ll} \ns\ds    \frac{1}{\dot \tau_0^\lambda} \big(  h( \varrho_0^\lambda(s), \tilde X_s^{0,u^\lambda} )-y \big)- \frac{1}{\dot \tau_1^\lambda} \big(  h( \varrho_1^\lambda(s), \tilde X_s^{1,u^\lambda} )-y \big)\\
   \ns\ds= \( \frac{1}{\dot \tau_0^\lambda}- \frac{1}{\dot \tau_1^\lambda}\) \big(  h( \varrho_0^\lambda(s), \tilde X_s^{0,u^\lambda} )-y \big)+ \frac{1}{\dot \tau_1^\lambda} \big(  h( \varrho_0^\lambda(s), \tilde X_s^{0,u^\lambda} )-  h( \varrho_1^\lambda(s), \tilde X_s^{1,u^\lambda} ) \big)
   \les \D_s (1+|y|).\\
   %
 \ea$$ 
 Note the boundedness of $  \tilde  Y_s^{n,0,u^\lambda} $ obtained in Lemma \ref{Y-esti}, we know
  $$\ba{ll} \ns\ds   \frac{n}{\dot \tau_1^\lambda} \big(   h( \varrho_1^\lambda(s), \tilde X_s^{1,u^\lambda} )+  \D_s-  \tilde  Y_s^{n,0,u^\lambda}  \big) ^- 
  \les  \frac{n}{\dot \tau_0^\lambda} \big( h( \varrho_0^\lambda(s), \tilde X_s^{0,u^\lambda} )  -  \tilde  Y_s^{n,0,u^\lambda}  \big)^- .
  \ea$$ 
 Moreover, 
 $\Phi(\tilde X_T^{0,u^\lambda}) \les \Phi(\tilde X_T^{1,u^\lambda})+C|\tilde X_T^{1,u^\lambda}-\tilde X_T^{0,u^\lambda}|\les  \Phi(\tilde X_T^{1,u^\lambda})+ \D_T.$

 Finally,   $t_0\les t_1$ guarantees us  to apply the comparison theorem of BSDE with jumps to \eqref{bar-Y-n0} and  \eqref{tilde_BSDEP_in} with $i=0$, we get 
 $ \tilde  Y_s^{n,0,u^\lambda}\les \overline   Y_s^{n,0,u^\lambda},\q \forall s\in [t_\lambda,T],\ \dbP\mbox{-a.s.}$ 
 \end{proof}

	\begin{remark}
 \sl  
		The process $\{\D_s \}_{s \in [t_0,T] }$ defined in \eqref{D} is an $\mathbb F^\lambda$-adapted, continuous increasing process, and from  
		 Lemma {\rm \ref{CV-Le-X-1}},  for all $p \geqslant 1$,
		\begin{equation}\label{A-esti}
			\mathbb E ^{\cF_s^\lambda}\Big[ \D_T^p-\D_s^p   \Big] \leqslant C_{p,\delta } \big( |t_0-t_1|^p + |\tilde X_s^{1,u^\lambda} - \tilde X_s^{0,u^\lambda}|^p \big),\quad \dbP\mbox{-a.s.}
		\end{equation}
	\end{remark}
	\begin{lemma}\label{esti-hatY1n}\sl
		Let {\rm \textbf{(H$_1$)}-\textbf{(H$_3$)}}, {\rm \textbf{(H$_4$)}-(i), (ii)} and {\rm \textbf{(C)}}  hold. Then 
		\begin{equation*}
				   -C  \leqslant \tilde  Y_s^{n,0,u^\lambda}\les \overline   Y_s^{n,0,u^\lambda} \leqslant C_\d+C_\d \D_s, \quad  s\in[t_\lambda,T], \  n\geqslant 1, \  \dbP\mbox{-a.s.} \\
				%
				%
		\end{equation*}
		Moreover, 
		$$
		 \mathbb E^{\mathcal F^\lambda_s} \Big[  \int_s^T  \big(|\overline Z_r^{n,0,u^\lambda}|^2  
						+     \|\overline V_r^{n,0,u^\lambda}(\cd)\|_{\nu,2}^2    \big) \mathrm dr 
						  \Big] \leqslant C_{ \delta}(1+ \D_s^2
		),\quad  s\in[t_\lambda,T],\ n\geqslant 1, \  \dbP\mbox{-a.s.} 
		$$
	\end{lemma}
	\begin{proof}
		%
		By  Lemmas  \ref{Y-esti} and \ref{Le-tilde-Y<bar}, for all $s\in[t_\lambda,T],$ $n\geqslant 1$, $\dbP$-a.s., we have 
$$ -C \leqslant\tilde  Y_s^{n,0,u^\lambda}\les \overline   Y_s^{n,0,u^\lambda} .$$
 Now, let us consider another  BSDE with jumps as follows,
$$  \left\{
		\begin{array}{ll}
			\ns\ds\!\!\!  \mathrm d \sY_s^{0,u^\lambda} 
			= -\[C_1   + C_\delta |t_1-t_0| \cd  \Big( \Big|\sZ_s^{0,u^\lambda}\Big|+ \Big|\int_E \sV_s^{0,u^\lambda}(e) \nu(\mathrm de) \Big| \Big ) +\D_s  \] \mathrm d s
			+\sZ_s^{0,u^\lambda}\mathrm d B_s^\lambda
			 \\
\ns\ds\!\!\!\hskip1.6cm    +\int_E \sV_s^{0,u^\lambda}(e)   \tilde N^\lambda(\mathrm d  s,\mathrm de),\q  s\in [t_\lambda,T],\\
\ns\ds  \!\!\!  \sY_T^{0,u^\lambda}= \Phi(\tilde X_T^{1,u^\lambda})+C_1\D_T.
\ea\right.$$ 
From the   boundedness of $f$ as well as  Lemma \ref{tau_pro},
		  we have,  $\dbP$-a.s., 
$$\ba{ll}
\ns\ds\!\!\!
  \frac{1}{\dot \tau_1^\lambda }   f\big(  \varrho_1^\lambda(s), \tilde X_s^{1,u^\lambda}, \overline   Y_s^{n,0,u^\lambda}\!\!-\! \D_s,  \sqrt{\dot \tau_1^\lambda} \overline   Z_s^{n,0,u^\lambda},  \int _E l(e)  \overline  V_s^{n,0,u^\lambda} (e) \nu(\mathrm de),u_s^\lambda  \big) 
-  (1-\frac{1}{\dot \tau_1^\lambda})  \int_E   \overline   V_s^{n,0,u^\lambda} (e)  \nu(\mathrm de) \\
\ns\ds\!\!\!       -    \frac{n}{\dot \tau_1^\lambda }  \big(  h( \varrho_1^\lambda(s), \tilde X_s^{1,u^\lambda} )+\D_s-\overline   Y_s^{n,0,u^\lambda}  \big)^ -  
 + C_\delta |t_1-t_0|\cd  \Big(   \Big|  \overline   Z_s^{n,0,u^\lambda} \Big| + \Big| \int_E  \overline   V_s^{n,0,u^\lambda}(e)\nu(\mathrm de) \Big|  \Big) +\D_s
\\
%
%
%
\ns\ds\!\!\! \les  C_1   + C_\delta |t_1-t_0| \cd \Big(   \Big|  \overline   Z_s^{n,0,u^\lambda}\Big| + \Big| \int_E  \overline   V_s^{n,0,u^\lambda}(e)\nu(\mathrm de) \Big|  \Big) +\D_s .
\ea$$
Therefore, using 
the comparison theorem, we get $$\overline   Y_s^{n,0,u^\lambda}\les \sY_s^{ 0,u^\lambda}, \q s\in[t_\lambda,T], \  n\geqslant 1, \  \dbP\mbox{-a.s.}$$

Next, we want to prove $\sY_s^{0,u^\lambda} \leqslant C_\d+C_\d \D_s, \ s\in[t_\lambda,T], \  n\geqslant 1, \  \dbP\mbox{-a.s.}$ 
For some constant $\b$, by applying the It\^{o}'s formula to $e^{\beta s} |\sY_s^{ 0,u^\lambda} |^2$, we obtain, for all $s\in[t_\lambda,T]$, $\dbP$-a.s.,
		\begin{equation*}
			\begin{array}{ll}
				\ns\ds\!\!\!   e^{\beta s} |\sY_s^{0,u^\lambda} |^2
				+\mathbb E ^{\mathcal F^\lambda_s}\Big[ \int_s^T  e^{\beta r}  \big(  \beta |\sY_r^{0,u^\lambda}|^2 + | \sZ_r^{0,u^\lambda}|^2   
				+ \|\sV_r^{0,u^\lambda}(\cd) \|^2_{\nu,2}    \big)  \mathrm dr   \Big]\\
				%
				\ns\ds\!\!\!   =\mathbb E^{\mathcal F^\lambda_s} \Big[  e^{\beta T} |\Phi(\tilde X_T^{1,u^\lambda})+C_1\D_T|^2\Big]
				+ 2 C_1    \mathbb E^{\mathcal F^\lambda_s} \Big[ \int_s^T    e^{\beta r}  \sY_r^{0,u^\lambda}        \mathrm dr   \Big] 
				+ 2    \mathbb E^{\mathcal F^\lambda_s} \Big[ \int_s^T    e^{\beta r}  \sY_r^{0,u^\lambda}  \D_r       \mathrm dr   \Big] 
				\\
				 \ns\ds\!\!\!\q 
				+ 2   C_\delta |t_1-t_0|  \cd  \mathbb E^{\mathcal F^\lambda_s} \Big[ \int_s^T    e^{\beta r}  \sY_r^{0,u^\lambda}   \Big( |\sZ_s^{0,u^\lambda}|+\Big|\int_E \sV_s^{0,u^\lambda}(e) \nu(\mathrm de)\Big| \Big )  \mathrm dr   \Big]. \\
				\ns\ds\!\!\!   \les   C_{ \delta}\big( 1+ \D_s^2  \big)
				+ C \mathbb E^{\mathcal F^\lambda_s} \Big[  \int_s^T    e^{\beta r} |\sY_r^{0,u^\lambda}|^2\mathrm dr   \Big]
				+\frac 12 \mathbb E \Big[ \int_s^T   e^{\beta r} \big( | \sZ_r^{0,u^\lambda}|^2   
				+ \|\sV_r^{0,u^\lambda}(\cd) \|^2_{\nu,2}  \big)  \mathrm dr  \Big]   . \\	    
				%
			\end{array}
		\end{equation*}
		%
%
%
		By choosing a sufficiently large  $\beta$, we obtain
		\begin{equation*}
			\label{estimate-Y2n}
			\begin{array}{ll}
				\ns\ds\!\!\!     | \sY_s^{0,u^\lambda} |^2
				+\mathbb E ^{\mathcal F^\lambda_s}\Big[ \int_s^T   \big( |\sY_r^{0,u^\lambda}|^2 + | \sZ_r^{0,u^\lambda}|^2   
				+ \|\sV_r^{0,u^\lambda}(\cd) \|^2_{\nu,2}   \big)  \mathrm dr \Big]
				\les C_{ \delta}(1+ \D_s^2 )
				,\q  \dbP\mbox{-a.s.}\\
				%
			\end{array}
		\end{equation*}
		In particular, 
		$ 
		|\sY_s^{0,u^\lambda} | \les  C_\delta(1+\D_s),\  s\in[t_\lambda,T],\ \dbP\mbox{-a.s.}
		$ 
Therefore,  
		$$ -C\les \overline   Y_s^{n,0,u^\lambda}\les \sY_s^{0,u^\lambda} \les C_\delta(1+\D_s)
		 , \q  \dbP\mbox{-a.s.}
		$$ 
		 %
	 
		Now, we proceed to establish the second result.
		 Let $C_h$ be the bound of $h$, applying the It\^{o}'s formula to $ (\overline   Y_s^{n,0,u^\lambda} + C_h - \D_s  )^2$, we obtain
		\begin{equation*}
			\begin{array}{ll}
				\ns\ds\!\!\!  \big(\overline   Y_s^{n,0,u^\lambda} + C_h - \D_s  \big)^2 
				+\mathbb E ^{\mathcal F^\lambda_s}\Big[ \int_s^T    \big(|\overline   Z_r^{n,0,u^\lambda}|^2+ \|\overline   V_r^{n,0,u^\lambda}(\cd)\|^2_{\nu,2}  \big) \mathrm dr \Big]\\
				\ns\ds\!\!\!   = \mathbb E^{\mathcal F^\lambda_s} \Big[  (\Phi( \tilde X_T^{1,u^\lambda})  + C_h  )^2     \Big]
				- \frac{2n}{\dot \tau_1^\lambda }\mathbb E \Big[ \int_s^T  (\overline   Y_r^{n,0,u^\lambda} +C_h - \D_r  )  
				\big(    h( \varrho_1^\lambda(r), \tilde X_r^{1,u^\lambda} ) +\D_r -\overline  Y_r^{n,0,u^\lambda}    \big)^ -  \mathrm dr \Big]\\
				\ns\ds\!\!\!    \q +  \frac{2}{\dot \tau_1^\lambda }\mathbb E^{\mathcal F^\lambda_s} \Big[ 
				\int_s^T  
				(\overline   Y_r^{n,0,u^\lambda}  \!\!\!+C_h \!-\! \D_r )  \   f\Big(  \varrho_1^\lambda(r), \tilde X_ r^{1,u^\lambda}, \overline   Y_r^{n,0,u^\lambda}\!\!\!-\D_r,  \sqrt{\dot \tau_1^\lambda} \overline   Z_r^{n,0,u^\lambda},  \int _E  l(  e) \overline V_r^{n,0,u^\lambda}\!\! (e)  \nu(\mathrm de),u_r^\lambda  \Big)  \mathrm dr  \Big]\\
				\ns\ds\!\!\!   \q + 2 C_\delta |t_1-t_0|\cd  \mathbb E^{\mathcal F^\lambda_s} \Big[ \int_s^T   (\overline   Y_r^{n,0,u^\lambda}\!\!  +C_h -\D_r )   \Big(  |  \overline   Z_r^{n,0,u^\lambda}|  + \Big|\int_E   \overline   V_r^{n,0,u^\lambda}(e)  \nu(\mathrm de)\Big|  \Big)  \mathrm dr  \Big]\\
				\ns\ds\!\!\!   \q + 2  \mathbb E^{\mathcal F^\lambda_s} \Big[ \int_s^T    (\overline   Y_r^{n,0,u^\lambda}\!\!  +C_h -\D_r )  \D_r  \mathrm dr  \Big]
				+2\mathbb E^{\mathcal F^\lambda_s} \Big[ \int_s^T  (\overline   Y_r^{n,0,u^\lambda} \!\!  + C_h - \D_r )   \mathrm d \D_r  \Big]\\
				\ns\ds\!\!\!   \q  + 2 (1-\frac{1}{\dot \tau_1^\lambda}) \mathbb E^{\mathcal F^\lambda_s} \Big[ \int_s^T \int_E  (\overline   Y_r^{n,0,u^\lambda} \!\!  + C_h - \D_r )    \overline   V_r^{n,0,u^\lambda} (e)  \nu(\mathrm de)      \mathrm dr  \Big] 
				.  \ 
			\end{array}
		\end{equation*}

			Further, according to the results of {\rm(\romannumeral1)}, \eqref{A-esti}, Cauchy inequality and the fact
		$$
		(\overline   Y_s^{n,0,u^\lambda} + C_h - \D_s  )  \big( h( \varrho_1^\lambda(s), \tilde X_s^{1,u^\lambda} )+\D_s -  \overline   Y_s^{n,0,u^\lambda}    \big)^ -
		\ges 0,\quad n\ges 1,\ s\in[t_\lambda,T],\ \dbP\mbox{-a.s.},
		$$
		we have, for all $ s\in[t_\lambda,T],\ \dbP\mbox{-a.s.}$,
		\begin{equation*}
			\begin{array}{ll}
				\ns\ds\!\!\!  \big(\overline   Y_s^{n,0,u^\lambda} + C_h - \D_s  \big)^2 
				+\mathbb E ^{\mathcal F^\lambda_s}\Big[ \int_s^T    \big(|\overline   Z_r^{n,0,u^\lambda}|^2+ \|\overline   V_r^{n,0,u^\lambda}(\cd)\|^2_{\nu,2}  \big) \mathrm dr \Big] \\
				\ns\ds\!\!\!  \les C_{ \delta} (1+\D_s^2) 
				+2\mathbb E^{\mathcal F^\lambda_s} \Big[ \int_s^T  (\overline   Y_r^{n,0,u^\lambda} \!\!  + C_h - \D_r )   \mathrm d \D_r  \Big]
				+ \frac 12  \mathbb E ^{\mathcal F^\lambda_s}\Big[ \int_s^T    \big(|\overline   Z_r^{n,0,u^\lambda}|^2+ \|\overline   V_r^{n,0,u^\lambda}(\cd)\|^2_{\nu,2}  \big) \mathrm dr \Big].  \\
			\end{array}
		\end{equation*}
		Thus, 
		\begin{equation}\label{est8}
		\begin{array}{ll}
			\ns\ds\!\!\! \big(\overline   Y_s^{n,0,u^\lambda} + C_h - \D_s  \big)^2 
			+\mathbb E ^{\mathcal F^\lambda_s}\Big[ \int_s^T    \big(|\overline   Z_r^{n,0,u^\lambda}|^2+ \|\overline   V_r^{n,0,u^\lambda}(\cd)\|^2_{\nu,2}  \big) \mathrm dr \Big] \\
			\ns\ds\!\!\!  \les C_{ \delta} (1+\D_s^2) 
			+2\mathbb E^{\mathcal F^\lambda_s} \Big[ \int_s^T  (\overline   Y_r^{n,0,u^\lambda} \!\!  + C_h - \D_r )   \mathrm d \D_r  \Big],\q  s\in[t_\lambda,T],\ \dbP\mbox{-a.s.}
		\end{array}
		\end{equation}

		Notice that, for any $1<p <2$ and $ q>2$ satisfying $\displaystyle \frac 1 p + \frac 1 q =1 $,  
		\begin{equation}\label{est7}
			\begin{array}{ll}
				\ns\ds\!\!\!	 \mathbb E^{\mathcal F^\lambda_s} \Big[ \int_s^T  (\overline   Y_r^{n,0,u^\lambda} \!\!  + C_h - \D_r )   \mathrm d \D_r  \Big]
					\les  \mathbb E^{\mathcal F^\lambda_s} \Big[  \sup_{r \in[s,T]}   |\overline   Y_r^{n,0,u^\lambda} \!\!  + C_h - \D_r |  (\D_T-\D_s) 	\Big]	\\
				%
				\ns\ds\!\!\!	\les  \(\mathbb E^{\mathcal F^\lambda_s} \Big[  \sup_{r \in[s,T]}   |\overline   Y_r^{n,0,u^\lambda} \!\!  + C_h - \D_r |^p   	\Big]  \)^{\frac 1 p} 
				\Big( \mathbb E^{\mathcal F^\lambda_s}  \Big[ \D_T ^q     \Big] \Big)^{\frac 1 q}	\\
				\ns\ds\!\!\!		\les \varepsilon \(\mathbb E^{\mathcal F^\lambda_s} \Big[  \sup_{r \in[s,T]}   |\overline   Y_r^{n,0,u^\lambda} \!\!  + C_h - \D_r |^p   	\Big]  \)^{\frac 2 p} 
				+C_ \varepsilon   \D_s ^2     	\\
				\ns\ds\!\!\!	\les \varepsilon M_{s,t}^{\frac 2 p} 
			   +C_\varepsilon    \D_s ^2.	\\ 
			\end{array}
		\end{equation}
%
		where  $\varepsilon$ will be given in detail later, and     $\displaystyle M_{ s,t} := \mathbb E^{\mathcal F^\lambda_s} \Big[  \sup_{r \in[t,T]}   |\overline   Y_r^{n,0,u^\lambda} \!\!  + C_h - \D_r |^p   	\Big],\     t_\lambda \les t\les s \les  T   $.
		We remark that  $(M_{s,t})_{s\ges t}$ is an $\mathbb{F}^\lambda$-martingale,
		and $\displaystyle\frac 2 p >1$. So from Doob's martingale inequality,
		\begin{equation}\label{estimate-M}
			\begin{array}{ll}
				\ns\ds\!\!\! \mathbb E^{\mathcal F^\lambda_t}  \Big[  \sup_{s\in[t,T]}  M_{s,t}^{\frac 2 p}   \Big] 
				\les  \Big({ \frac 2 {2-p}} \Big) ^{\frac 2 p}   \Big( \mathbb E^{\mathcal F^\lambda_t}  \Big[    M_{T,t}^{\frac 2 {2-p}}   \Big]  \Big)^{\frac  {2-p} p}
				\les  \Big({ \frac 2 {2-p}} \Big) ^{\frac 2 p}     \mathbb E ^{\mathcal F^\lambda_t} \Big[    M_{ T,t}^{\frac 2 p}  \Big] \\
				%
				%
				\ns\ds\!\!\! \les   \Big({ \frac 2 {2-p}} \Big) ^{\frac 2 p}        \mathbb E ^{\mathcal F^\lambda_t} \Big[    \sup_{s \in[t,T]}   |\overline   Y_s^{n,0,u^\lambda} \!\!  + C_h - \D_s |^2     	\Big].    \\
			\end{array}
		\end{equation}
		%

		Now, we go back to the inequality \eqref{est8}. Combined with \eqref{est7} and \eqref{estimate-M}, we obtain
		$$
		\begin{array}{ll}
			\ns\ds\!\!\! \mathbb E^{\mathcal F^\lambda_t}  \Big[  \sup_{s\in[t ,T]} \big(\overline   Y_s^{n,0,u^\lambda} + C_h - \D_s  \big)^2 \]
			 \les C_{ \delta} (1+\D_t^2) 
			+2  \varepsilon  \Big({ \frac 2 {2-p}} \Big) ^{\frac 2 p}        \mathbb E ^{\mathcal F^\lambda_t} \Big[    \sup_{s \in[t,T]}   |\overline   Y_s^{n,0,u^\lambda} \!\!  + C_h - \D_s |^2     	\Big].\\
			%
			%
		 
		\end{array}
		$$

		\no By choosing a suitable $\varepsilon>0$ such that $\displaystyle  2  \varepsilon  \Big({ \frac 2 {2-p}} \Big) ^{\frac 2 p}  <1 $, we get  
		$$
		\mathbb E^{\mathcal F^\lambda_t}  \Big[  \sup_{s\in[t ,T]} \big(\overline   Y_s^{n,0,u^\lambda} + C_h - \D_s  \big)^2 \]
		\les C_{ \delta} (1+\D_t^2).
		$$
		Thereby,
		$$
			\begin{array}{ll}
				\ns\ds\!\!\! \mathbb E ^{\mathcal F^\lambda_s}\Big[ \int_s^T    \big(|\overline   Z_r^{n,0,u^\lambda}|^2+ \|\overline   V_r^{n,0,u^\lambda}(\cd)\|^2_{\nu,2}  \big) \mathrm dr \Big] \\
				\ns\ds\!\!\!  \les C_{ \delta} (1+\D_s^2) 
				+2\mathbb E^{\mathcal F^\lambda_s} \Big[  \sup_{r \in[s,T]}  |\overline   Y_r^{n,0,u^\lambda} \!\!  + C_h - \D_r |^2        \Big]
				 \les C_{ \delta} (1+\D_s^2)
				,\q  s\in[t_\lambda,T],\  \dbP\mbox{-a.s.}
			\end{array}
		$$

		\end{proof}

		For all $s \in [t_{\lambda}, T ]$ and $n \ges 1$,
		we define
		$ 
	\cY_{s}^{n,0,u^\lambda}:=\overline Y_s^{n,0,u^\lambda}\! -\D_{s}.
		$ 
		\noindent Then, from \eqref{bar-Y-n0}, we get 
		\begin{equation}\label{Y2n}
			\left\{
			\begin{array}{ll}
				\ns\ds\!\!\!  \mathrm d \cY_{s}^{n,0,u^\lambda} \! =  \!  - \Big[ \!     \frac{1}{\dot \tau_1^\lambda }   f\Big(  \varrho_1^\lambda(s), \tilde X_s^{1,u^\lambda}, \cY_{s}^{n,0,u^\lambda} ,  \sqrt{\dot \tau_1^\lambda} \overline   Z_s^{n,0,u^\lambda},  \int _E l(e)  \overline  V_s^{n,0,u^\lambda} (e) \nu(\mathrm de),u_s^\lambda  \Big) 
				\\
				\ns\ds\!\!\!   \hskip 1.98cm    -    \frac{n}{\dot \tau_1^\lambda }  \big(  h( \varrho_1^\lambda(s), \tilde X_s^{1,u^\lambda} ) - \cY_{s}^{n,0,u^\lambda} \big)^ -  
				-  (1-\frac{1}{\dot \tau_1^\lambda})  \int_E   \overline   V_s^{n,0,u^\lambda} (e)  \nu(\mathrm de) \\
				\ns\ds\!\!\!   \hskip 1.98cm    + C_\delta |t_1-t_0| \cd \Big(   |  \overline   Z_s^{n,0,u^\lambda}| + | \int_E  \overline   V_s^{n,0,u^\lambda}(e)\nu(\mathrm de) |  \Big)    +\D_s    \Big]  \mathrm ds
				-\mathrm d\D_s\\
				\ns\ds\!\!\!   \hskip 1.65cm     +  \overline   Z_s^{n,0,u^\lambda}   \mathrm dB_s^\lambda  + \int_E   \overline   V_s^{n,0,u^\lambda} (e)   \tilde  N^\lambda(\mathrm ds,\mathrm de)  , \quad s \in [t_\lambda,T],\\
				\ns\ds\!\!\!  \cY_{T}^{n,0,u^\lambda}=\Phi(\tilde X_T^{1,u^\lambda}).
			\end{array}
			\right.
		\end{equation}

		\no Notice that the penalization term of the above BSDE \eqref{Y2n} is the same as BSDE \eqref{tilde_BSDEP_in} with $i=1$, and they are driven by the same Brownian motion and Poisson random measure in $[t_\lambda,T]$. 
		Thereafter, according to the standard estimates and Doob's martingale inequality, we give an estimate between them as follows.
		\begin{lemma}\label{esti-Y0n-Y2n}\sl
				Assume that {\rm \textbf{(H$_1$)}-\textbf{(H$_3$)}}, {\rm \textbf{(H$_4$)}-(i), (ii)} hold.
				 Then, there exists some constant  $C_{ \delta}$, such that, for all $n\ges 1$ and $s\in[t_\lambda,T]$,  $\dbP$-a.s.,
				$$
				\mathbb E^{\mathcal F_s^\lambda} \Big[ \sup_{r \in[s,T]} | \cY_{r}^{n,0,u^\lambda} -\tilde Y_r^{n,1,u^\lambda} |^2 + \int_s^T \( |\overline   Z_r^{n,0,u^\lambda} - \tilde Z_r^{n,1,u^\lambda}  |^2
				+ \| \overline   V_r^{n,0,u^\lambda} (\cd) -  \tilde V_r^{n,1,u^\lambda}(\cd) \|^2_{\nu,2} \)\mathrm dr  \Big] \les C_{ \delta} \D_s^2. 
				$$
		\end{lemma}
  
\begin{proof} 
 
Setting $ \big( \D Y_\cd,\D Z_\cd,\D V_\cd(\cd) \big):= \big(  \cY_{\cd}^{n,0,u^\lambda}-\tilde Y_\cd^{n,1,u^\lambda},\overline   Z_\cd^{n,0,u^\lambda}-   \tilde Z_\cd^{n,1,u^1} ,  \overline   V_\cd^{n,0,u^\lambda} (\cd)-\tilde V_\cd^{n,1 ,u^1} (\cd)\big)$, 
  by \eqref{tilde_BSDEP_in} and \eqref{Y2n}, we have
\begin{equation}\label{}
		\left\{
		\begin{array}{ll}

			\ns\ds\!\!\!  \mathrm d \D Y_s \! =  \!  - \Big[ \!     \frac{1}{\dot \tau_1^\lambda }   f\Big(  \varrho_1^\lambda(s), \tilde X_s^{1,u^\lambda}, \cY_{s}^{n,0,u^\lambda} ,  \sqrt{\dot \tau_1^\lambda} \overline   Z_s^{n,0,u^\lambda},  \int _E l(e)  \overline  V_s^{n,0,u^\lambda} (e) \nu(\mathrm de),u_s^\lambda  \Big) 
			\\
			\ns\ds\!\!\!  \hskip 1.6cm 
			-\frac{1}{\dot\tau_1^\lambda}      f\Big(  \varrho_1^\lambda(s), \tilde X_s^{1,u^\lambda}, \tilde Y_s^{n,1,u^\lambda},  \sqrt{\dot \tau_1^\lambda} \tilde Z_s^{n,1 ,u^\lambda},  \int _E l(e) \tilde V_s^{n,1,u^\lambda} (e) \nu(\mathrm de),u_s^\lambda  \Big)         
			   \  \\
			\ns\ds\!\!\!   \hskip 1.6cm   +  \frac{n}{\dot \tau_i^\lambda}     \big( h(  \varrho_1^\lambda(s) ,\tilde X_s^{1,u^\lambda} ) - \tilde Y_s^{n,1,u^\lambda} \big)^ - 
			-    \frac{n}{\dot \tau_1^\lambda }  \big(  h( \varrho_1^\lambda(s), \tilde X_s^{1,u^\lambda} ) - \cY_{s}^{n,0,u^\lambda} \big)^ -  
			 \\
			\ns\ds\!\!\!   \hskip 1.6cm    + C_\delta |t_1-t_0| \cd \Big(    |  \overline   Z_s^{n,0,u^\lambda}| + | \int_E  \overline   V_s^{n,0,u^\lambda}(e)\nu(\mathrm de) |  \Big)+\D_s
			- (1-\frac{1}{\dot\tau_1^\lambda}) \int_E   \D  V_s  (e)      \nu(\mathrm de)  \Big]  \mathrm ds
			\\
			\ns\ds\!\!\!   \hskip 1.25 cm   -\mathrm d\D_s  +  \D Z_s   \mathrm dB_s^\lambda  + \int_E   \D V_s(e)  \tilde  N^\lambda(\mathrm ds,\mathrm de),\quad s \in [t_\lambda,T],  \\
			\ns\ds\!\!\!  \D Y_T =0.
		\end{array}
		\right.
	\end{equation}

For some constant $\a>0$, applying It\^o's formula to $e^{\a s}|\D Y_s|^2$, we get 
\begin{equation*} 
		\begin{array}{ll}
			\ns\ds\!\!\!     e^{\a s}|\D Y_s |^2   +\dbE^{\cF_s^\lambda} \[\int_s^T \!  \( \a e^{\a r}|\D Y_r |^2 +   e^{\a r}   | \D Z_r |^2\)\mathrm dr \]
			+ \dbE^{\cF_s^\lambda} \[\int_ s^T   \int_E  e^{\a r}  | \D V_r  (e) |^2      N^\lambda(\mathrm dr,\mathrm de)\]  \\
\ns\ds\!\!\!  =   \frac{ 2}{\dot \tau_1^\lambda }  \dbE^{\cF_s^\lambda} \[ \int_s^T  e^{\a r} \D Y_r
\(    f\big(  \varrho_1^\lambda(r), \tilde X_r^{1,u^\lambda}, \cY_{r}^{n,0,u^\lambda} ,  \sqrt{\dot \tau_1^\lambda} \overline   Z_r^{n,0,u^\lambda},  \int _E l(e)  \overline  V_r^{n,0,u^\lambda} (e) \nu(\mathrm de),u_r^\lambda  \big)
\\
\ns\ds\!\!\!  \hskip 3.67cm 
-       f\big(  \varrho_1^\lambda(r), \tilde X_r^{1,u^\lambda}, \tilde Y_r^{n,1,u^\lambda},  \sqrt{\dot \tau_1^\lambda} \tilde Z_r^{n,1 ,u^\lambda},  \int _E l(e) \tilde V_r^{n,1,u^\lambda} (e) \nu(\mathrm de),u_r^\lambda  \big)   \) \mathrm dr \]        
   \\
\ns\ds\!\!\!    \q + \frac{2n}{\dot \tau_1^\lambda } \dbE^{\cF_s^\lambda} \[     \int_s^T  e^{\a r} \D Y_r   \( \big( h(  \varrho_1^\lambda(r) ,\tilde X_r^{1,u^\lambda} ) - \tilde Y_r^{n,1,u^\lambda} \big)^ - 
-      \big(  h( \varrho_1^\lambda(r), \tilde X_r^{1,u^\lambda} ) - \cY_{r}^{n,0,u^\lambda} \big)^ -  \) \mathrm dr\]
\\
\ns\ds\!\!\!   \q    + 2 C_\delta |t_1-t_0|   \cd   \dbE^{\cF_s^\lambda} \[   \int_s^T  e^{\a r} \D Y_r   \Big(   |  \overline   Z_r^{n,0,u^\lambda}| + | \int_E  \overline   V_r^{n,0,u^\lambda}(e)\nu(\mathrm de) |  \Big)     \mathrm dr \]
 + 2    \dbE^{\cF_s^\lambda} \[   \int_s^T  e^{\a r} \D Y_r    \D_r    \mathrm dr \]\\
\ns\ds\!\!\!\q 
- 2(1-\frac{1}{\dot\tau_1^\lambda}) \dbE^{\cF_s^\lambda} \[  \int_s^T    \int_E  e^{\a r} \D Y_r  \D  V_r  (e)      \nu(\mathrm de)   \mathrm dr \] 
+2\dbE^{\cF_s^\lambda} \[ \int_s^T  e^{\a r} \D Y_r\mathrm d\D_r\] \\
		\end{array} 
	\end{equation*}
\begin{equation*} 
		\begin{array}{ll}
\ns\ds\!\!\!\!\!\! \!\!\! \!\!\! \!\!\! \!\!\!   \les  \big( 11C^2_{ \d}+2 \big)  \dbE^{\cF_s^\lambda} \[ \int_s^T  e^{\a r} |\D Y_r|^2\mathrm dr\]
+ \frac12  \dbE^{\cF_s^\lambda} \[  \int_s^T  e^{\a r} \big(  |  \D Z_r|^2 +\|   \D V_r(\cd)  \|_{\nu,2} ^2    \big) \mathrm dr \]        
\\
\ns\ds\!\!\! \!\!\! \!\!\! \!\!\! \!\!\!   \q    +  C_{ \d}|t_1-t_0|^2    \dbE^{\cF_s^\lambda} \[   \int_s^T    \big(    |  \overline   Z_r^{n,0,u^\lambda}|^2 + \| \ \overline   V_r^{n,0,u^\lambda}(\cd)\|_{\nu,2}^2  \big)     \mathrm dr \]
+2\dbE^{\cF_s^\lambda} \[ \int_s^T  e^{\a r} \D Y_r\mathrm d\D_r\],\\
%
%
%
		\end{array} 
	\end{equation*}
	where we have used  $\D Y_r   \( \big( h(  \varrho_1^\lambda(r) ,\tilde X_r^{1,u^\lambda} ) - \tilde Y_r^{n,1,u^\lambda} \big)^ -
	 -      \big(  h( \varrho_1^\lambda(r), \tilde X_r^{1,u^\lambda} ) - \cY_{r}^{n,0,u^\lambda} \big)^ -  \) \les 0,\   r\in[t_\lambda,T],\ \dbP$-a.s.

For $\a$ large enough, using Lemma \ref{esti-hatY1n} and Lemma 3.1 in \cite{LW-2014-lp}, we get
\begin{equation*} 
		\begin{array}{ll}
			\ns\ds   \dbE^{\cF_s^\lambda}\[ |\D Y_s |^2   +\int_s^T \!  \(  |\D Y_r |^2  +     |  \D Z_s |^2+    \|   \D V_r(\cd)  \|_{\nu,2} ^2    \)\mathrm dr\ ]
			\les  C_{ \d}   \D_s^2 
			+2\dbE^{\cF_s^\lambda} \[ \int_s^T  e^{\a r} \D Y_r\mathrm d\D_r\]   .
		\end{array} 
	\end{equation*}
	The term $\ds\dbE^{\cF_s^\lambda} \[ \int_s^T  e^{\a r} \D Y_r\mathrm d\D_r\] $ can be dealt with by employing the technique in the proof of  Lemma \ref{esti-hatY1n}-(ii) (refer to \eqref{est7}, \eqref{estimate-M}). Then  the desired result can be proved.
\end{proof}
		
	\ms 

 Based on the above preparations, we have the following estimates.
	 %

	 \bl\label{CV-Le-Y-1}\sl
	 Assume the conditions of Lemma \ref{esti-hatY1n} hold ture, then
	 $$| \tilde Y_s^{n,1,u^\lambda} - \tilde  Y_s^{n,0,u^\lambda} |\les C_\delta  \D_s ,\q \dbP\mbox{-a.s.}$$

\el
\begin{proof}
 For all $n\ges 1$ and $s\in[t_\lambda,T]$, based on Lemma \ref{esti-hatY1n} and Lemma \ref{esti-Y0n-Y2n}, we have
		\begin{equation*}     
			 \tilde  Y_s^{n,0,u^\lambda}-\tilde Y_s^{n,1,u^\lambda} 
			\les\overline   Y_s^{n,0,u^\lambda} -\tilde Y_s^{n,1,u^\lambda} 
			= \cY_s^{n,0,u^\lambda}   + \D_s-\tilde Y_s^{n,1,u^\lambda} 
			\les  C_{ \delta} \D_s, \quad n\ges 1,\ s\in[t_\lambda,T],\ \dbP\mbox{-a.s.} 
		\end{equation*}
		Moreover, due to the symmetry, by the same argument, we also get
		\begin{equation*}
			\begin{array}{ll}
				\ns\ds\!\!\!  
				 \tilde  Y_s^{n,0,u^\lambda}-\tilde Y_s^{n,1,u^\lambda}  
				\ges - C_\delta  \D_s  ,\quad n\ges 1,\ s\in[t_\lambda,T],\ \dbP\mbox{-a.s.} 
			\end{array}
		\end{equation*}
		Thus, $|\tilde Y_s^{n,1,u^\lambda} - \tilde  Y_s^{n,0,u^\lambda} |\les C_\delta \D_s $, for  all $n\ges 1,\ s\in[t_\lambda,T],\ \dbP\mbox{-a.s.}$
\end{proof}

 \bl\label{CV-Le-ZK-1} \sl
Assume {\rm \textbf{(H$_1$)}-\textbf{(H$_3$)}}, {\rm \textbf{(H$_4$)}-(i), (ii)} and {\rm \textbf{(C)}}.\\
\rm{(i)} Setting
	\begin{equation}\label{sd}
	\sD_s^n:= \int_{s}^T (  h(  \varrho_0^\lambda(r) ,\tilde X_r^{0,u^\lambda} ) - \tilde  Y_r^{n,0,u^\lambda} )^-\mathrm d \tilde  A_r^{n,1,u^\lambda }
	+\int_{s}^T ( h(  \varrho_1^\lambda(r) ,\tilde X_r^{1,u^\lambda} )- \tilde Y_r^{n,1,u^\lambda}  )^-\mathrm d \tilde  A_r^{n,0,u^\lambda }, 
	 \end{equation}
for any $ p\ges 2$,  we have, 
	$ \lim\limits_{n\to\i} \dbE^{\mathcal F_s^\lambda}  [ (\sD_s^n)^\frac p2  ]=0 ,$ $s\in [t_\lambda,T]$, $\dbP$-a.s. \\
(ii) For any $p\ges 2$ and $s\in[t_\lambda,T]$, $\dbP$-a.s., 
 $$
 \dbE^{\cF_s^\lambda}\[\(\int_s^T(|\tilde Z_r^{n,1,u^\lambda} -\tilde  Z_r^{n,0,u^\lambda} |^2+\|\tilde V_r^{n,1,u^\lambda}(\cd ) -\tilde  V_r^{n,0,u^\lambda}(\cd ) \|_{\nu,2}^2)\mathrm dr\)^{\frac p2}
 +|\widehat A_T-\widehat A_s |^{ p} \]
 \les C_\delta  \D_s^{ p}  +C  \dbE^{\cF_s^\lambda}[(\sD_s^n)^{\frac   p2}  ] .
 $$

\el
\begin{proof} 
%
The first result can be directly proved by using Lemma  \ref{Le-h-Y-0}-(ii). Now we focus on the second one.
For all $n\ges 1$, $s\in[t_\lambda,T]$, setting $\widehat X_s := \tilde X_s^{ 1,u^\lambda} -\tilde  X_s^{ 0,u^\lambda} $ and
$\widehat \varphi_s := \tilde \varphi_s^{n,1,u^\lambda} -\tilde  \varphi_s^{n,0,u^\lambda} $ with $\varphi = Y, Z, V, A$.
From equation \eqref{tilde_BSDEP_in}, we have
\begin{equation}\label{equ-hat-Y}
		\left\{
		\begin{array}{ll}
			\ns\ds\!\!\!  \mathrm d \widehat Y_s   
			=  - \Big\{  \frac{1}{\dot \tau_1^\lambda}      f\big(  \varrho_1^\lambda(s), \tilde X_s^{1,u^\lambda}, \tilde Y_s^{n,1,u^\lambda},  \sqrt{\dot \tau_1^\lambda} \tilde Z_s^{n,1,u^\lambda},  \int _E  l(e) \tilde V_s^{n,1,u^\lambda} (e) \nu(\mathrm de),u_s^\lambda  \big) \\
			\ns\ds\!\!\! \hskip 1.425cm  -\frac{1}{\dot \tau_0^\lambda}        f\big(  \varrho_0^\lambda(s), \tilde X_s^{0,u^\lambda}, \tilde  Y_s^{n,0,u^\lambda},  \sqrt{\dot \tau_0^\lambda} \tilde Z_s^{n,0,u^\lambda},  \int _E  l(e) \tilde V_s^{n,0,u^\lambda} (e) \nu(\mathrm de),u_s^\lambda  \big) \\ 
			\ns\ds\!\!\! \hskip 1.425cm -   \int_E  \(\widehat  V_s  (e)  -\frac{1}{\dot \tau_1^\lambda} \tilde V_s^{n,1,u^\lambda} (e) +\frac{1}{\dot \tau_0^\lambda} \tilde V_s^{n,0,u^\lambda} (e)\)\nu(\mathrm de)   \Big\}\mathrm ds  +\mathrm d  \widehat  A_s \\
			\ns\ds\!\!\! \hskip 1.05cm   +   \widehat Z_s   \mathrm dB_s^\lambda  + \int_E   \widehat V_s  (e)   \tilde  N^\lambda(\mathrm ds,\mathrm de)  , \quad s \in [t_\lambda,T],\\
			\ns\ds\!\!\!  \widehat Y_T  =  \Phi(\tilde X_T^{1,u^\lambda})-\Phi(\tilde X_T^{0,u^\lambda}).
		\end{array}
		\right.
	\end{equation}


\emph{Step 1.} We prove the stated result holds with $p=2$. For this,
 applying It\^o's formula to $|\widehat Y_s |^2$, we have  
%
	\begin{equation}\label{est9} 
		\begin{array}{ll}
			\ns\ds\!\!\!        |\widehat Y_s|^2 
			+\int_s^T |\widehat Z_r|^2 \mathrm dr
			+\int_s^T  \int_E|\widehat V_r  (e)|^2   N^\lambda(\mathrm dr,\mathrm de) 
			\\
%
%
%
				\ns\ds\!\!\!   = |\Phi(\tilde X_T^{1,u^\lambda}) - \Phi(\tilde X_T^{0,u^\lambda})|^2 + 2\int_s^T \widehat Y_r   \big(I_r^1  - I_r^2\big)\mathrm dr 
			-2\int_s^T  \widehat Y_r  \mathrm d  \widehat  A_r - \int_s^T  2\widehat Y_r   \widehat Z_r   \mathrm dB_r^\lambda
			
			  \\
			\ns\ds \!\!\! \q
			-\int_s^T  \int_E   2\widehat Y_r \widehat V_r  (e)   \tilde  N^\lambda(\mathrm dr,\mathrm de),
			\\
		\end{array}
	\end{equation}
where
 $$
 \ba{ll}
\ns\ds I_r^1:= \frac{1}{\dot \tau_1^\lambda}      f\big(  \varrho_1^\lambda(r), \tilde X_r^{1,u^\lambda}, \tilde Y_r^{n,1,u^\lambda},  \sqrt{\dot \tau_1^\lambda} \tilde Z_r^{n,1,u^\lambda},  \int _E  l(e) \tilde V_r^{n,1,u^\lambda} (e) \nu(\mathrm de),u_r^\lambda  \big) \\
\ns\ds\!\!\! \hskip 1.1cm  -\frac{1}{\dot \tau_0^\lambda}        f\big(  \varrho_0^\lambda(r), \tilde X_r^{0,u^\lambda}, \tilde  Y_r^{n,0,u^\lambda},  \sqrt{\dot \tau_0^\lambda} \tilde Z_r^{n,0,u^\lambda},  \int _E  l(e) \tilde V_r^{n,0,u^\lambda} (e) \nu(\mathrm de),u_r^\lambda  \big)  ,\\
\ns\ds I_r^2:= \int_E      \(\widehat  V_r  (e)  -\frac{1}{\dot \tau_1^\lambda} \tilde V_r^{n,1,u^\lambda} (e) +\frac{1}{\dot \tau_0^\lambda} \tilde V_r^{n,0,u^\lambda} (e)\)\nu(\mathrm de)  .
\ea
$$

\no Then, according to Lemma \ref{CV-Le-Y-1}, we have 
 \begin{equation}\label{I1}
 	\ba{ll}
 	%
 	%
 	\ns\ds\!\!\!  |I_r^1|  \les       |t_1-t_0|     +        |\varrho_1^\lambda(r)-\varrho_0^\lambda(r)| + |\widehat X_r |+| \widehat Y_r |
 	+ \Big|\sqrt{\dot \tau_1^\lambda}-\sqrt{\dot \tau_0^\lambda} \Big|\cd| \tilde Z_r^{n,1,u^\lambda}|+ \Big| \sqrt{\dot \tau_0^\lambda} \Big|\cd|\widehat  Z_r | 
 	\\
 	\ns\ds\!\!\!\hskip1cm
 	+     \| l( \cd) \|_{\nu,2}    \cd     \|\widehat V_r  (\cd)\|_{\nu,2}      \\
 	\ns\ds\!\!\!    \les C       \(  \D_r  
 	+ |t_1-t_0|\cd | \tilde Z_r^{n,1,u^\lambda}|+  |\widehat  Z_r | +\|\widehat V_r  (\cd)\|_{\nu,2} \), \\
 	\ea
 \end{equation}
and 
 \begin{equation}\label{I2}	 
%
%
 |I_r^2|   \les  C    \int_E  \(  \Big|1 -\frac{1}{\dot \tau_0^\lambda} \Big| \cd |\widehat  V_r  (e)|+ |t_1-t_0| \cd|\tilde V_r^{n,1,u^\lambda} (e)| \)\nu(\mathrm de) . 
\end{equation}

 Moreover, 
$$\ba{ll} 
\ns\ds\!\!\!   -\int_s^T\widehat Y_r  \mathrm d  \widehat  A_r  
= -\int_s^T (\tilde Y_r^{n,1,u^\lambda}-\tilde  Y_r^{n,0,u^\lambda})  \mathrm d (\tilde A_r^{n,1,u^\lambda}-\tilde A_r^{n,0,u^\lambda})  \\
%
%
 \ns\ds\!\!\!\les    \int_s^T  \big(   h( \varrho_1^\lambda(r), \tilde X_r^{1,u^\lambda} )-\tilde Y_r^{n,1,u^\lambda}  \big)  ^-\mathrm d  \tilde A_r^{n,0,u^\lambda}  +\int_s^T (   h( \varrho_0^\lambda(r), \tilde X_r^{0,u^\lambda} ) -\tilde  Y_r^{n,0,u^\lambda}  )^-  \mathrm d \tilde A_r^{n,1,u^\lambda}  \\
\ns\ds\!\!\!   \q -\int_s^T \big(  h( \varrho_1^\lambda(r), \tilde X_r^{1,u^\lambda} )  - h( \varrho_0^\lambda(r), \tilde X_r^{0,u^\lambda} )   \big)  \mathrm d (\tilde A_r^{n,1,u^\lambda}-\tilde A_r^{n,0,u^\lambda})\\ 
\ns\ds\!\!\!    \les \sD_s^n  +C\D_T  |    \widehat A_T -\widehat A_s   |.
\ea$$

Based on the above estimates,  we return  to \eqref{est9} and obtain
	\begin{equation}\label{hat Y-2} 
	\begin{array}{ll}
		\ns\ds\!\!\!        |\widehat Y_s|^2 
		+\int_s^T |\widehat Z_r|^2 \mathrm dr
		+\int_s^T  \int_E|\widehat V_r  (e)|^2   N^\lambda(\mathrm dr,\mathrm de) 
		\\
		%
		%
		\ns\ds\!\!\!   \les C|\widehat X_T |^2 +    C_\delta \int_s^T       \D_r    \(  \D_r  
		+ |t_1-t_0|\cd | \tilde Z_r^{n,1,u^\lambda}|+  |\widehat  Z_r | +\|\widehat V_r  (\cd)\|_{\nu,2}\)\mathrm dr\\
		\ns\ds\!\!\!\q
		+C_\delta   \int_s^T      \int_E  \D_r 
		  \(  |1 -\frac{1}{\dot \tau_0^\lambda} | \cd |\widehat  V_r  (e)|+ |t_1-t_0| \cd|\tilde V_r^{n,1,u^\lambda} (e)| \)\nu(\mathrm de) 
		    \mathrm dr
		     +C\D_T  |    \widehat A_T -\widehat A_s   |+2  \sD_s^n  \\
		  \ns\ds\!\!\! \q
		- \int_s^T  2\widehat Y_r   \widehat Z_r   \mathrm dB_r^\lambda
		-\int_s^T  \int_E   2\widehat Y_r \widehat V_r  (e)   \tilde  N^\lambda(\mathrm dr,\mathrm de) \\
		\ns\ds\!\!\!   \les C_{\delta,\varepsilon}  \D_T^2   +    C_\delta |t_1-t_0|^2  \int_s^T         \big(    
		| \tilde Z_r^{n,1,u^\lambda}|^2 +  \|\tilde V_r^{n,1,u^\lambda} (\cd)\|_{\nu,2}^2    \big)\mathrm dr
		+\varepsilon  \int_s^T   \big( |\widehat  Z_r |^2 + \|\widehat V_r  (\cd)\|_{\nu,2}^2 \big)      \mathrm dr   \\
		\ns\ds\!\!\!\q    
		+C\D_T  |    \widehat A_T -\widehat A_s   |+2\sD_s^n  
		- \int_s^T  2\widehat Y_r   \widehat Z_r   \mathrm dB_r^\lambda
		-\int_s^T  \int_E   2\widehat Y_r \widehat V_r  (e)   \tilde  N^\lambda(\mathrm dr,\mathrm de).\\
	\end{array}
\end{equation}

\no Then, choosing $\varepsilon \in(0,1)$, we have
\begin{equation*}
	\begin{array}{ll} 
		\ns\ds\!\!\!        |\widehat Y_s|^2 
		+\int_s^T |\widehat Z_r|^2 \mathrm dr
		
		+     \int_s^T     \|\widehat V_r  (\cd)\|_{\nu,2}^2    \mathrm dr 
		\\
		\ns\ds\!\!\!   \les C_\delta  \D_T^2   +    C_\delta |t_1-t_0|^2  \int_s^T         \big(    
		| \tilde Z_r^{n,1,u^\lambda}|^2 +  \|\tilde V_r^{n,1,u^\lambda} (\cd)\|_{\nu,2}^2    \big)\mathrm dr
		+C\D_T  |    \widehat A_T -\widehat A_s   |+C\sD_s^n  
		  \\
		\ns\ds\!\!\!\q    
		- C\int_s^T   \widehat Y_r   \widehat Z_r   \mathrm dB_r^\lambda
		-C\int_s^T  \int_E    \widehat Y_r \widehat V_r  (e)   \tilde  N^\lambda(\mathrm dr,\mathrm de)
		-C\int_s^T  \int_E|\widehat V_r  (e)|^2   \tilde  N^\lambda(\mathrm dr,\mathrm de)  .\\
	\end{array}
\end{equation*}

\no Therefore, according to  Lemma \ref{Y-esti},  for any $\varepsilon >0$,
\begin{equation}\label{hat-Y-2-A}
	\begin{array}{ll} 
		\ns\ds\!\!\!        |\widehat Y_s|^2 
		+ \dbE^{\cF_s^\lambda}\[ \int_s^T\big(  |\widehat Z_r|^2  
		+    \|\widehat V_r  (\cd)\|_{\nu,2}^2  \big) \mathrm dr \]
		\les C_{ \d,\varepsilon} \D_s^2    + C  \dbE^{\cF_s^\lambda} [ \sD_s^n  ]
		 +\e  \dbE^{\cF_s^\lambda} [ |      \widehat A_T -\widehat A_s    |^2   ] , \q \dbP\mbox{-a.s.}
		%
	\end{array}
\end{equation}

Now we deal with $ \dbE^{\cF_s^\lambda} [ |   \widehat A_T -\widehat A_s  |^2  ] $, from equation  \eqref{equ-hat-Y}, \eqref{I1} and \eqref{I2},
 \begin{equation}\label{AT-As} 
		\begin{array}{ll}
		\ns\ds\!\!\!  \dbE^{\cF_s^\lambda}\[ \big|      \widehat A_T -\widehat A_s   \big|^2 \]\\
		\ns\ds\!\!\! \les C| \widehat Y_s|^2 + C  \dbE^{\cF_s^\lambda}\[ | \Phi(\tilde X_T^{1,u^\lambda}) - \Phi(\tilde X_T^{0,u^\lambda})|^2\] 
		+\dbE^{\cF_s^\lambda}\[ \int_s^T\big(  |\widehat Z_r|^2  
		+    \|\widehat V_r  (\cd)\|_{\nu,2}^2  \big) \mathrm dr \] 
		\\
		\ns\ds\!\!\!\q 
		+C\dbE^{\cF_s^\lambda}\[ \Big(   \int_s^T  \big(I_r^1 -  I_r^2  \big) \mathrm d r  \Big)  ^2\]   \\
\ns\ds\!\!\! \les C \D_s ^2  +C \dbE^{\cF_s^\lambda}\[ \int_s^T\big(  |\widehat Z_r|^2  
+    \|\widehat V_r  (\cd)\|_{\nu,2}^2  \big) \mathrm dr \],\q \dbP\mbox{-a.s.}   \\
		\end{array}  
	\end{equation}
Combined with \eqref{hat-Y-2-A}, choosing $\e$ small enough,  
 \begin{equation*} 
		\begin{array}{ll}
			\ns\ds\!\!\!  \dbE^{\cF_s^\lambda}\[ \big|      \widehat A_T -\widehat A_s   \big|^2 \] 
			\les C_\delta  \D_s^2    ,\q \dbP\mbox{-a.s.}
		\end{array}  
	\end{equation*}
Thus, 
\begin{equation*}\label{hat-Y-2-A-1}
	\begin{array}{ll}
			\ns\ds\!\!\! |\widehat Y_s|^2 
			+ \dbE^{\cF_s^\lambda}\[ \int_s^T\big(  |\widehat Z_r|^2  
			+    \|\widehat V_r  (\cd)\|_{\nu,2}^2  \big) \mathrm dr \]
			 \les   C_\delta  \D_s^2    + C  \dbE^{\cF_s^\lambda} [ \sD_s^n  ],\q \dbP\mbox{-a.s.} 
		\end{array}
	\end{equation*}

  \emph{Step 2}. Now we focus on the case $p>2 $. 
From \eqref{hat Y-2}, for $p>2$,
\begin{equation*} 
	\begin{array}{ll}
			\ns\ds\!\!\!    \dbE^{\cF_s^\lambda}\[    |\widehat Y_s|^{2 }+  \(\int_s^T|  \widehat Z_r |^2  \mathrm d r \)^{ \frac p2} +\(\int_s^T\int_E|\widehat V_r  (e)|^2   N^\lambda(\mathrm dr,\mathrm de)  \)^{  \frac p2}\]\\
			\ns\ds\!\!\!   \les  \dbE^{\cF_s^\lambda}\[   \(  |\widehat Y_s|^{2}+  \int_s^T|  \widehat Z_r |^2  \mathrm d r  +\int_s^T\int_E|\widehat V_r  (e)|^2   N^\lambda(\mathrm dr,\mathrm de)  \)^\frac p2 \]\\
\ns\ds\!\!\!  \les   C_\delta  \D_s^{ p}+  C_\delta  |t_1-t_0|^{ p}  \cd  \dbE^{\cF_s^\lambda}\[  \( \int_s^T       \big(    
| \tilde Z_r^{n,1,u^\lambda}|^2 +  \|\tilde V_r^{n,1,u^\lambda} (\cd)\|_{\nu,2}^2    \big)\mathrm dr \)^{\frac p2} \]
+C\dbE^{\cF_s^\lambda}\[  \D_T^{ \frac p2}   |    \widehat A_T -\widehat A_s   |^{ \frac p2} \]  \\
\ns\ds\!\!\! \q  +C  \dbE^{\cF_s^\lambda}[|\sD_s^n|^{  \frac p2}  ] 
+ C\dbE^{\cF_s^\lambda}\[ \Big( \int_s^T  |\widehat Y_r |^2 |  \widehat Z_r |^2  \mathrm d r   \Big)^{\frac p4}\] 
+C \dbE^{\cF_s^\lambda}\[ \Big( \int_s^T \int_E    |\widehat Y_r |^2 |  \widehat V_r  (e)|^2     N^\lambda(\mathrm dr,\mathrm de)  \Big)^{\frac  p4}\]\\
\ns\ds\!\!\! \q  +   C \epsilon ^\frac p2  \dbE^{\cF_s^\lambda}\[   \( \int_s^T   \big( |\widehat  Z_r |^2 + \|\widehat V_r  (\cd)\|_{\nu,2}^2 \big)      \mathrm dr\)^\frac p2    \] \\  
\ns\ds\!\!\!  \les   C_\delta  \D_s^{ p}
+C\dbE^{\cF_s^\lambda}\[  \D_T^{ \frac p2}   |    \widehat A_T -\widehat A_s   |^{ \frac p2} \]
+   C \epsilon ^\frac p2  \dbE^{\cF_s^\lambda}\[   \( \int_s^T   \big( |\widehat  Z_r |^2 + \|\widehat V_r  (\cd)\|_{\nu,2}^2 \big)      \mathrm dr\)^\frac p2    \]  \\
\ns\ds\!\!\! \q  +C  \dbE^{\cF_s^\lambda}[|\sD_s^n|^{  \frac p2}  ] 
+ C\dbE^{\cF_s^\lambda}\[  \D_T  ^\frac p2 \Big( \int_s^T    |  \widehat Z_r |^2  \mathrm d r   \Big)^{\frac p4}\] 
+C \dbE^{\cF_s^\lambda}\[ \D_T  ^\frac p2 \Big( \int_s^T \int_E      |  \widehat V_r  (e)|^2     N^\lambda(\mathrm dr,\mathrm de)  \Big)^{\frac  p4}\]\\
\ns\ds\!\!\!  \les   C_\delta  \D_s^{ p} 
+C\dbE^{\cF_s^\lambda}\[  \D_T^{ \frac p2}   |    \widehat A_T -\widehat A_s   |^{ \frac p2} \] +C  \dbE^{\cF_s^\lambda}[|\sD_s^n|^{  \frac p2}  ] 
+ C\epsilon ^\frac p2  \dbE^{\cF_s^\lambda}\[    \Big( \int_s^T    |  \widehat Z_r |^2  \mathrm d r   \Big)^{\frac p2}\]  \\
\ns\ds\!\!\! \q  
+C \epsilon ^\frac p2 \dbE^{\cF_s^\lambda}\[   \Big( \int_s^T \int_E      |  \widehat V_r  (e)|^2     N^\lambda(\mathrm dr,\mathrm de)  \Big)^{\frac  p2}\],\\
		\end{array}
	\end{equation*}
	where we have used Lemma 3.1 in  \cite{LW-2014-lp}, Lemma  \ref{Y-esti}, Lemma \ref{CV-Le-Y-1} and \eqref{A-esti}.

Thus, by choosing $\epsilon$ small enough, we get
\begin{equation*} 
	\begin{array}{ll}
		\ns\ds\!\!\!    \dbE^{\cF_s^\lambda}\[    |\widehat Y_s|^{2 }+  \(\int_s^T|  \widehat Z_r |^2  \mathrm d r \)^{ \frac p2} +\(\int_s^T\int_E|\widehat V_r  (e)|^2   N^\lambda(\mathrm dr,\mathrm de)  \)^{  \frac p2}\]\\
		\ns\ds\!\!\!  \les   C_\delta  \D_s^{ p} 
		+C\dbE^{\cF_s^\lambda}\[  \D_T^{ \frac p2}   |    \widehat A_T -\widehat A_s   |^{ \frac p2} \] +C  \dbE^{\cF_s^\lambda}[|\sD_s^n|^{  \frac p2}  ]
		   \\
		\ns\ds\!\!\!   \les   C_\delta  \D_s^{ p} 
		+C\varepsilon\dbE^{\cF_s^\lambda}\[    |    \widehat A_T -\widehat A_s   |^{  p } \] +C  \dbE^{\cF_s^\lambda}[|\sD_s^n|^{  \frac p2}  ].
	 \\
	\end{array}
\end{equation*}

\no Further, similar to the proof of \eqref{AT-As}, we get 
 \begin{equation*}\label{ } 
		\begin{array}{ll}
			\ns\ds\!\!\! \dbE^{\cF_s^\lambda}\[  \(\int_s^T|  \widehat Z_r |^2  \mathrm d s \)^{ \frac p 2} +\(\int_s^T\int_E|\widehat V_r  (e)|^2   N^\lambda(\mathrm ds,\mathrm de)  \)^{  \frac p 2}+    \big| \widehat A_T -\widehat A_s   \big|^{ p} \]
			\les C_\delta  \D_s^{ p}  +C  \dbE^{\cF_s^\lambda}[|\sD_s^n|^{   \frac p 2}  ]  .   
		\end{array}  
	\end{equation*}

\no  Finally, combined with Lemma 3.1 in \cite{LW-2014-lp}, the desired result is proved. 
\end{proof}

		 Recalling the notation  $ \tilde \cX_\cd^{ u^\lambda} = \lambda\tilde  X_\cd^{1,u^\lambda}  + (1-\lambda)\tilde X_\cd^{0,u^\lambda}$, we also put
		$$
		\big (\tilde\cY_\cd^{n,u^\lambda} ,\tilde\cZ_\cd^{n,u^\lambda},\tilde\cV_\cd^{n,u^\lambda}\!(\cd) \big)
		 = \big(\lambda \tilde Y_\cd^{n,1,u^\lambda} + (1-\lambda) \tilde  Y_\cd^{n,0,u^\lambda},   \lambda \tilde Z_\cd^{n,1,u^\lambda} + (1-\lambda) \tilde Z_\cd^{n,0,u^\lambda},   \lambda \tilde V_\cd^{n,1,u^\lambda}\!(\cd) + (1-\lambda) \tilde V_\cd^{n,0,u^\lambda}\!(\cd) \big).
		$$

	\begin{lemma}\label{comp-hat-tilde}\sl
		Suppose {\rm \textbf{(H$_1$)}-\textbf{(H$_5$)}} and {\rm \textbf{(C)}} hold, for all $ n\ges 1$, we have 
		$$
		\tilde\cY_s^{n,u^\lambda} \les \overline \cY_s^{n,u^\lambda}, \quad\ s\in[t_\lambda,T],\ \dbP\mbox{-a.s.},
		$$
		where	 
		\begin{equation*}\label{ }
			\left\{
			\begin{array}{ll}
				\ns\ds\!\!\!  \mathrm d \overline \cY_s^{n,u^\lambda}\!  = -   \Big[  f  \big (s,   X_s^{\lambda,u^\lambda}, \overline \cY_s^{n,u^\lambda}\! -\! \overline \D_s , \overline \cZ_s^{n,u^\lambda}, \int_E l( e)\overline\cV_s^{n,u^\lambda}(e)   \nu(\mathrm de),  u_s^\lambda \big )  
				\! - \! n \big(  h (s,  X_s^{\lambda,u^\lambda} )\! -\! \overline\cY_s^{n,u^\lambda}  +\overline \D_s \big)^- \\
				%
				%
				\ns\ds \!\!\!\hskip 1.9cm   +  C\overline \D_s  +  C_\delta \lambda (1-\lambda)\( 
				|\widehat Z_s |^2+\|\widehat  V_s(\cd)\|^2_{\n,2}    +  |t_0-t_1|^2 |\tilde Z_s^{n,0,u^\lambda}|^2   \)
				\Big] \mathrm ds\\
				\ns\ds\!\!\!\hskip 1.5cm +   \overline\cZ_s^{n,u^\lambda}   \mathrm dB_s^\lambda  + \int_E   \overline\cV_s^{n,u^\lambda} (e)    \tilde  N^\lambda(\mathrm ds,\mathrm de)  ,\quad s \in [t_\lambda,T],\\
				\ns\ds   \overline\cY_T^{n,u^\lambda}  = \Phi(  X_T^{\lambda,u^\lambda})+\overline \D_T,
			\end{array}
			\right.
		\end{equation*}
		with $\overline \D_s= C\tilde\D_s + C_\delta \lambda (1-\lambda) \D_s^2$, $\tilde\D_s:= \sup\limits_{r\in[t_\lambda,s]}|   \tilde \cX_r^{u^\lambda}-X_r^{\lambda,u^\lambda}|$, $\forall s \in [t_\lambda, T]$.
		%
	\end{lemma}
	\begin{proof}
		Using the  semiconcavity and Lipschitz  continuity  of $f$, 
%
		\begin{equation*}
			\begin{array}{ll}
				\ns\ds\!\!\!   \lambda           f \big(  \varrho_1^\lambda(s), \tilde X_s^{1,u^\lambda}, \tilde Y_s^{n,1,u^\lambda},  \sqrt{\dot \tau_1^\lambda} \tilde Z_s^{n,1,u^\lambda},  \int _E  l( e)  \tilde V_s^{n,1,u^\lambda} (e) \nu(\mathrm de),u_s^\lambda   \big) \\
				\ns\ds\!\!\! 	  +(1-\lambda) f \big(  \varrho_0^\lambda (s), \tilde X_s^{0,u^\lambda}, \tilde  Y_s^{n,0,u^\lambda},  \sqrt{\dot \tau_0^\lambda} \tilde Z_s^{n,0,u^\lambda},  \int _E l( e)  \tilde V_s^{n,0,u^\lambda} (e) \nu(\mathrm de), u_s^\lambda   \big)\\
				\ns\ds\!\!\!    \les  f \big (s, \tilde \cX_s^{u^\lambda}, \tilde \cY_s^{n,u^\lambda}, 
				\lambda \sqrt{\dot \tau_1^\lambda} \tilde Z_s^{n,1,u^\lambda} +(1-\lambda) \sqrt{\dot \tau_0^\lambda} \tilde Z_s^{n,0,u^\lambda}, 
				 \int _E l(e)\tilde\cV_s^{n,u^\lambda}  (e)   \nu(\mathrm de) , u_s^\lambda  \big )\\
				\ns\ds\!\!\!   \q + C_\delta \lambda (1-\lambda)\Big( 
				 |  \varrho_1^\lambda(s) \!- \!  \varrho_0^\lambda (s) |^2  
				+ |\widehat X_s |^2  +  |   \widehat Y_s  |^2 +  \Big| \sqrt{\dot \tau_1^\lambda} \tilde Z_s^{n,1,u^\lambda} \!-\!\sqrt{\dot \tau_0^\lambda}  \tilde Z_s^{n,0,u^\lambda} \Big|^2 +\Big|\int_E l(e)   \widehat V_s   (e)    \nu(\mathrm de)\Big|^2\Big) 	\\

				\ns\ds\!\!\!  	  \les   f \big(s, X_s^{\lambda,u^\lambda}, \tilde\cY_s^{n,u^\lambda}, \tilde\cZ_s^{n,u^\lambda}, \int_E l( e)\tilde\cV_s^{n,u^\lambda}(e)   \nu(\mathrm de),  u_s^\lambda  \big)
				+ C\Big|\lambda \sqrt{\dot \tau_1^\lambda} \tilde Z_s^{n,1,u^\lambda} +(1-\lambda) \sqrt{\dot \tau_0^\lambda} \tilde Z_s^{n,0,u^\lambda} -\tilde\cZ_s^{n,u^\lambda}\Big|  \\
				\ns\ds\!\!\!   \q + C_\delta \lambda (1-\lambda)\Big( 
				|  \varrho_1^\lambda(s) \!- \!  \varrho_0^\lambda (s) |^2 
				+ |\widehat X_s |^2  +  |   \widehat Y_s  |^2 +  \Big| \sqrt{\dot \tau_1^\lambda} \tilde Z_s^{n,1,u^\lambda} \!-\!\sqrt{\dot \tau_0^\lambda}  \tilde Z_s^{n,0,u^\lambda} \Big|^2 +\Big|\int_E l(e)   \widehat V_s   (e)    \nu(\mathrm de)\Big|^2\Big) 	\\
				\ns\ds\!\!\! \q
				+ C | \tilde \cX_s^{u^\lambda}-X_s^{\lambda,u^\lambda} | \\
				%
				%
				%

				\ns\ds\!\!\!  	 \les   f (s,   X_s^{\lambda,u^\lambda}, \tilde\cY_s^{n,u^\lambda}, \tilde\cZ_s^{n,u^\lambda}, \int_E l( e)\tilde\cV_s^{n,u^\lambda}(e)   \nu(\mathrm de),  u_s^\lambda  )+ C  \overline \D_s
				\\
				\ns\ds\!\!\! \q +  C_\delta \lambda (1-\lambda)\( 
			      |\widehat Z_s |^2+\|\widehat  V_s(\cd)\|^2_{\n,2}    +  |t_0-t_1|^2 |\tilde Z_s^{n,0,u^\lambda}|^2   \), %
			\end{array}
		\end{equation*}
			where $\big(\widehat X , \widehat Y , \widehat Z , \widehat V    \big)$ are the ones introduced in the proof of Lemma \ref{CV-Le-ZK-1}.
			
		 From the Lipschitz  continuity  of $f$, we obtain
		\begin{equation*}
			\begin{array}{ll}
				\ns\ds\!\!\!	  (\frac{1}{\dot \tau_1^\lambda}-1)\lambda    f \big(  \varrho_1^\lambda(s), \tilde X_s^{1,u^\lambda}, \tilde Y_s^{n,1,u^\lambda},  \sqrt{\dot \tau_1^\lambda} \tilde Z_s^{n,1,u^\lambda},  \int _E  l( e)  \tilde V_s^{n,1,u^\lambda} (e) \nu(\mathrm de),u_s^\lambda   \big)\\
				\ns\ds\!\!\! 	  +(\frac{1}{\dot \tau_0^\lambda}-1) (1-\lambda) f \big(  \varrho_0^\lambda (s), \tilde X_s^{0,u^\lambda}, \tilde  Y_s^{n,0,u^\lambda},  \sqrt{\dot \tau_0^\lambda} \tilde Z_s^{n,0,u^\lambda},  \int _E l( e)  \tilde V_s^{n,0,u^\lambda} (e) \nu(\mathrm de), u_s^\lambda   \big)\\	 
				\ns\ds\!\!\! \les C_\delta \lambda (1-\lambda) |t_0-t_1| \Big( | \varrho_1^\lambda(s) \!-\! \varrho_0^\lambda  (s) |
				+ | \widehat X_s  | +  | \widehat Y_s |
				+ |\sqrt{\dot \tau_1^\lambda} \tilde Z_s^{n,1,u^\lambda} \!-\! \sqrt{\dot \tau_0} \tilde Z_s^{n,0,u^\lambda} | 
				+   | \int _E     \widehat V_s    (e)  l(e)   \nu(\mathrm de)  |  \Big)
				\\
				\ns\ds\!\!\!	 	 \les  C_\delta \lambda (1-\lambda)\( 
				\D_s  ^2 +|\widehat Z_s |^2+\|\widehat  V_s(\cd)\|^2_{\n,2}    +  |t_0-t_1|^2 |\tilde Z_s^{n,0,u^\lambda}|^2   \),	\\

			\end{array}
		\end{equation*}

		Morover,
		\begin{equation*}
			\begin{array}{ll}
				\ns\ds\!\!\!	 -\int_E \Big( \lambda (1-\frac{1}{\dot \tau_1^\lambda})   \tilde V_s^{n,1,u^\lambda} (e) + (1-\lambda)(1-\frac{1}{\dot \tau_0^\lambda})   \tilde V_s^{n,0,u^\lambda} (e) \Big) \nu(\mathrm de)
				\les  C_\delta \lambda (1-\lambda)  \big( \Delta_s^2+\|\widehat V_s(\cd)\|_{\n,2}^2 \big).\\ 
				%
				%
			\end{array}
		\end{equation*}

		Thus, by adding the terms in the left sides of the above inequalities, we get
				\begin{equation}\label{coe-hat-tilde}
						\begin{array}{ll}
				\ns\ds\!\!\! 	 \frac{\lambda}{\dot \tau_1^\lambda}       f \big(  \varrho_1^\lambda(s), \tilde X_s^{1,u^\lambda}, \tilde Y_s^{n,1,u^\lambda},  \sqrt{\dot \tau_1^\lambda} \tilde Z_s^{n,1,u^\lambda},  \int _E  l( e)  \tilde V_s^{n,1,u^\lambda} (e) \nu(\mathrm de),u_s^\lambda   \big) 
								
								\\
				\ns\ds\!\!\!  + \frac{1-\lambda}{\dot \tau_0^\lambda}   f \big(  \varrho_0^\lambda (s), \tilde X_s^{0,u^\lambda}, \tilde  Y_s^{n,0,u^\lambda},  \sqrt{\dot \tau_0^\lambda} \tilde Z_s^{n,0,u^\lambda},  \int _E l( e)  \tilde V_s^{n,0,u^\lambda} (e) \nu(\mathrm de), u_s^\lambda   \big)\\
				\ns\ds\!\!\! 
				- \int_E  \Big( \lambda (1-\frac{1}{\dot \tau_1^\lambda})   \tilde V_s^{n,1,u^\lambda} (e)   + (1-\lambda)   (1-\frac{1}{\dot \tau_0^\lambda})   \tilde V_s^{n,0,u^\lambda} (e)   \Big) \nu(\mathrm de) \\
								%
								%
				\ns\ds\!\!\! \les  	    f (s,    X_s^{\lambda,u^\lambda}, \tilde\cY_s^{n,u^\lambda}, \tilde\cZ_s^{n,u^\lambda}, \int_E l( e)\tilde\cV_s^{n,u^\lambda}(e)   \nu(\mathrm de),  u_s^\lambda  )
				+  C\overline \D_s \\
				\ns\ds\!\!\! \q +  C_\delta \lambda (1-\lambda)\( 
				\D_s  ^2 +|\widehat Z_s |^2+\|\widehat  V_s(\cd)\|^2_{\n,2}    +  |t_0-t_1|^2 |\tilde Z_s^{n,0,u^\lambda}|^2   \)\\
				 \ns\ds\!\!\! \les  	    f (s, X_s^{\lambda,u^\lambda}, \tilde\cY_s^{n,u^\lambda}-\overline \D_s , \tilde\cZ_s^{n,u^\lambda}, \int_E l( e)\tilde\cV_s^{n,u^\lambda}(e)   \nu(\mathrm de),  u_s^\lambda  )  +  C\overline \D_s  \\
				 \ns\ds\!\!\! \q +  C_\delta \lambda (1-\lambda)\( 
				  |\widehat Z_s |^2+\|\widehat  V_s(\cd)\|^2_{\n,2}    +  |t_0-t_1|^2 |\tilde Z_s^{n,0,u^\lambda}|^2   \). 
							\end{array}
					\end{equation}

		Similarly, based on  the semiconcavity and Lipschitz  continuity of $h$ and $\Phi$, we obtain
		\begin{equation}\label{pen-hat-tilde}
			\begin{array}{ll}
				\ns\ds\!\!\!  -\frac{\lambda}{\dot \tau_1^\lambda} n \big(  h(\varrho_1^\lambda(s),\tilde X_s^{1,u^\lambda}) -\tilde Y_s^{n,1,u^\lambda}  \big)^-
				-\frac{1-\lambda}{\dot \tau_0^\lambda} n \big(h(\varrho_0^\lambda (s),\tilde X_s^{0,u^\lambda})- \tilde  Y_s^{n,0,u^\lambda} \big)^-	\\
				\ns\ds\!\!\!  \les -n \Big(\frac{\lambda}{\dot \tau_1^\lambda} \big(  h(\varrho_1^\lambda(s),\tilde X_s^{1,u^\lambda}) -\tilde Y_s^{n,1,u^\lambda}   \big)   
				+\frac{1-\lambda}{\dot \tau_0^\lambda}  \big( h(\varrho_0^\lambda (s),\tilde X_s^{0,u^\lambda})- \tilde  Y_s^{n,0,u^\lambda} \big)	\Big)^-	\\
				\ns\ds\!\!\!  =-n \Big( -  \tilde\cY_s^{n,u^\lambda} 
				-\lambda(1-\lambda) \frac{t_0-t_1}{T-t_\lambda}  \widehat Y_s  
				+\lambda(1-\lambda) \frac{t_0-t_1}{T-t_\lambda} \big( h(\varrho_1^\lambda(s),\tilde X_s^{1,u^\lambda})- h(\varrho_0^\lambda (s),\tilde X_s^{0,u^\lambda})   \big)
				\\
				\ns\ds\!\!\! \q  +\lambda h(\varrho_1^\lambda(s),\tilde X_s^{1,u^\lambda})+ (1-\lambda ) h(\varrho_0^\lambda (s),\tilde X_s^{0,u^\lambda})
				\Big)^-	\\
				%
				\ns\ds\!\!\! 	\les -n \big(  h (s,  X_s^{\lambda, u^\lambda} )-\tilde\cY_s^{n,u^\lambda}  +\overline \D_s \big)^-,\\
			\end{array}
		\end{equation}
	and 
		%
		\begin{equation}\label{ter-hat-tilde}
			\begin{array}{ll}
				\ns\ds\!\!\! 	\lambda \Phi(\tilde X_T^{1,u^\lambda}) + (1-\lambda) \Phi(\tilde X_T^{0,u^\lambda})
				\les \Phi( \tilde \cX_T^{u^\lambda})+C_\delta \lambda (1-\lambda) | \tilde X_T^{1,u^\lambda} - \tilde X_T^{0,u^\lambda}|^2
				\les \Phi(  X_T^{\lambda,u^\lambda}) +C\overline \D_T.\\
			\end{array}
		\end{equation}
		\par Finally, based on the above inequalities \eqref{coe-hat-tilde}-\eqref{ter-hat-tilde}, we can  apply the comparison theorem to get the desired result.
		\end{proof}

		\ms
		The last auxiliary process is introduced as follows. 
		$$
		 \wt \sY^{n,u^\lambda} _s:= \overline \cY_s^{n,u^\lambda}-\overline \D_s,\quad s\in [t_\lambda,T],\ \dbP\mbox{-a.s.}
		$$
		Then, 
		\begin{equation}\label{bar-BSDEP}
			\left\{
			\begin{array}{ll}
				\ns\ds\!\!\!  \mathrm d  \widetilde \sY^{n,u^\lambda} _s\!  = -   \Big[  f  \big (s,  X_s^{\lambda,u^\lambda},  \widetilde \sY^{n,u^\lambda}  _s , \overline \cZ_s^{n,u^\lambda}, \int_E l( e)\overline\cV_s^{n,u^\lambda}(e)   \nu(\mathrm de),  u_s^\lambda \big )  
				\! - \! n \big(  h (s, X_s^{\lambda,u^\lambda})\! -\!  \widetilde\sY^{n,u^\lambda} _s   \big)^- \\
				%
				%
				\ns\ds \!\!\!\hskip 1.9cm   +  C\overline \D_s  +  C_\delta \lambda (1-\lambda)\( 
				|\widehat Z_s |^2+\|\widehat  V_s(\cd)\|^2_{\n,2}    +  |t_0-t_1|^2 |\tilde Z_s^{n,0,u^\lambda}|^2   \)
				\Big] \mathrm ds
				-\mathrm d\overline \D_s\\
				\ns\ds\!\!\!\hskip 1.5cm +   \overline\cZ_s^{n,u^\lambda}   \mathrm dB_s^\lambda  + \int_E   \overline\cV_s^{n,u^\lambda} (e)    \tilde  N^\lambda(\mathrm ds,\mathrm de)  ,\quad s \in [t_\lambda,T],\\
				\ns\ds   \overline\cY_T^{n,u^\lambda}  = \Phi( X_T^{\lambda,u^\lambda}).
			\end{array}
			\right.
		\end{equation}
		%
		Note that the penalization term of the above BSDE \eqref{bar-BSDEP} is the same as BSDE \eqref{t-lambda_BSDEP}, and they are driven by the same Brownian motion and Poisson random measure in $[t_\lambda,T]$. Similar to the proof of Lemmas \ref{esti-hatY1n} and \ref{esti-Y0n-Y2n}, we give an estimate between them as follows.

		\begin{lemma}\label{esti-barY-lambdaY}\sl
			Assume the conditions {\rm \textbf{(H$_1$)}-\textbf{(H$_4$)}} and {\rm \textbf{(C)}}  hold true. Then, there exists some constant  $C_\delta\ges 0$ such that, for all $n\ges 1$ and $s\in[t_\lambda,T]$,  $\dbP$-a.s.,
			$$
			\begin{array}{ll}
				\ns\ds\!\!\!  \mathbb E^{\mathcal F_s^\lambda} \Big[ \sup_{r \in[s,T]} |  \wt \sY^{n,u^\lambda} _r - Y_r^{n,\lambda,u^\lambda}  |^2 + \int_s^T   \Big( |   \overline\cZ_r^{n,u^\lambda} -   Z_r^{n,\lambda,u^\lambda}  |^2  
				+ \|  \overline\cV_r^{n,u^\lambda} (\cd)-   V_r^{n,\lambda,u^\lambda} (\cd) \|_{\nu,2}^2\Big)    \mathrm dr  \Big] \\
				\ns\ds\!\!\! 
				\les 
				C \overline \D_s^2
				+    \dbE^{\cF_ s^\lambda} \big[  |\sD_s^n| ^2 \big] ,  
			\end{array}
			$$
			where $\sD^n$ is defined in \eqref{sd}.
		\end{lemma}

\begin{proof}
	For all $s\in[t_\lambda,T]$, denote $ \big(\sY^n_s ,\sZ^n_s ,\sV^n_s \big):=\big( \wt \sY^{n,u^\lambda} _s- Y_s^{n,\lambda,u^\lambda} ,\overline \cZ_s^{n,u^\lambda} -  Z_s^{n,\lambda,u^\lambda} , \overline\cV_s^{n,u^\lambda}  -  V_s^{n,\lambda,u^\lambda}   \big)$. Then,
	\begin{equation*} 
	\left\{
	\begin{array}{ll}
		\ns\ds\!\!\!  \mathrm  d\sY^n_s =   -   \Big[  f  \big (s, X_s^{\lambda,u^\lambda},  \widetilde \sY^{n,u^\lambda} _s  , \overline \cZ_s^{n,u^\lambda}, \int_E l( e)\overline\cV_s^{n,u^\lambda}(e)   \nu(\mathrm de),  u_s^\lambda \big )  
		\! - \! n \big(  h (s, \cX_s^{ u^\lambda} )\! -\! \wt \sY^{n,u^\lambda}_s   \big)^-
		+  C\overline \D_s+D^n_s \\
		\ns\ds\!\!\! \hskip1.6cm  - f\big(  s, X_s^{\lambda,u^\lambda}, Y_s^{n,\lambda,u^\lambda},  Z_s^{n,\lambda,u^\lambda},  \int _E  l( e) V_s^{n,\lambda,u^\lambda} (e)   \nu(\mathrm de),u_s^\lambda  \big)  
		+ n    \big(  h( s ,X_s^{\lambda,u^\lambda})\!-\!Y_s^{n,\lambda,u^\lambda}  \big)^ - 
		
		\Big] \mathrm ds \\
		%
		%
		\ns\ds\!\!\!\hskip 1.3cm -\mathrm d\overline \D_s +   \sZ_s^{n}   \mathrm dB_s^\lambda  + \int_E    \sV_s  (e)    \tilde  N^\lambda(\mathrm ds,\mathrm de) ,\quad s \in [t_\lambda,T], \\
		\ns\ds\!\!\! \sY_T^n=0,
	\end{array}
	\right.
\end{equation*} 
where $D^n_s:=  C_\delta \lambda (1-\lambda)\( 
		|\widehat Z_s |^2+\|\widehat  V_s(\cd)\|^2_{\n,2}    +  |t_0-t_1|^2 |\tilde Z_s^{n,0,u^\lambda}|^2   \),\ s\in[t_\lambda,T] $.

 For some constant $\a>0$, we apply  It\^o's formula to $e^{\a s}|\sY^n_s|^2$, 
	\begin{equation*} 
		\begin{array}{ll}
			\ns\ds\!\!\! e^{\a  s}|\sY^n_ s|^2 + \int_s^T \a e^{\a r}|\sY^n_r|^2 \mathrm dr  
			+ \int_s^T   e^{\a r}  |\sZ^n_r|^2\mathrm dr
			+  \int_s^T \int_E e^{\a r} |\sV^n_r(e)|^2 N^\lambda(\mathrm ds,\mathrm de) \\
			\ns\ds\!\!\! =  
		 	 2 \int_s^T e^{\a r} \sY^n_r  \[ 
            I_r^3 + n    \big(  h( r ,X_r^{\lambda,u^\lambda})-Y_r^{n,\lambda,u^\lambda}  \big)^ - 
			\! - \! n \big(  h (r, \cX_r^{ u^\lambda} )\! -\!\wt \sY^{n,u^\lambda}_r   \big)^- 
			+  C\overline \D_r+D^n_r \]\mathrm dr\\
			%
			%
			%
			\ns\ds\!\!\!  \q  +2 \int_s^T e^{\a r} \sY^n_r\mathrm d\overline \D_r
			- 2 \int_s^T e^{\a r} \sY^n_r \sZ^n_r \mathrm dB_r^\lambda 
			- 2 \int_s^T   \int_E  e^{\a r} \sY^n_r  \sV^n_r (e) \tilde N^\lambda(\mathrm dr,\mathrm de), 
		\end{array}
	\end{equation*} 
	%
%
%
with $$\ds I_r^3\! :=\! f\big(   r\!,  X_r^{\lambda,u^\lambda}\!,      \wt \sY^{n,u^\lambda} _r\!,    \overline\cZ_r^{n,u^\lambda}\!, \!
\int _E \! l(  e )\overline\cV_r^{n,u^\lambda} (e) \nu(\mathrm de),\! u_r^\lambda   \big) 
-  f\big(   r \!, X_r^{\lambda,u^\lambda}\!, Y_r^{n,\lambda,u^\lambda}\!,  Z_r^{n,\lambda,u^\lambda}\!, \!  \int _E \!  l( e) 		V_r^{n,\lambda,u^\lambda} (e)   \nu(\mathrm de),\! u_r^\lambda   \big).    $$
\no Then, based on the continuity of $f$,
	\begin{equation*} 
		\begin{array}{ll}
			\ns\ds\!\!\!  e^{\a s}|\sY^n_s|^2+ 
			\dbE^{\cF_ s^\lambda}\[ \int_ s^T\(\a e^{\a r}|\sY^n_r|^2 +e^{\a r} |\sZ^n_r|^2 + \int_E e^{\a r} |\sV^n_r(e)|^2 \n(\mathrm de) \)\mathrm dr  \]\\
\ns\ds\!\!\!   \les  C \dbE^{\cF_ s^\lambda}\[\int_ s^T   e^{\a r} \sY^n_r  \(   
| \sY^n_r|+| \sZ^n_r|+ \Big|\int_E  l(e)  \sV^n_r(e)\n(\mathrm de) \Big|    \)\mathrm dr  \]
+  2 \dbE^{\cF_ s^\lambda}\[ \int_s^T e^{\a r} \sY^n_r  \big( C\overline \D_ r +D^n_r\big) \mathrm dr\] \\
\ns\ds\!\!\! \q
	+2\dbE^{\cF_ s^\lambda}\[ \int_s^T e^{\a r} \sY^n_r\mathrm d\overline \D_r\]\\
\ns\ds\!\!\!   \les  
C\dbE^{\cF_ s^\lambda} \[ \int_ s^T   e^{\a r}|\sY^n_r|^2  \mathrm dr  \]
+C\dbE^{\cF_s^\lambda}\big[| \overline \D_T|^2   \big]
+C\dbE^{\cF_s^\lambda}\big[  |\sD_s^n| ^2 \big]
+2\dbE^{\cF_ s^\lambda}\[ \int_s^T e^{\a r} \sY^n_r\mathrm d\overline \D_r\]
\\
\ns\ds\!\!\!\q 
+ \frac12 \dbE^{\cF_ s^\lambda}\[ \int_ s^T\( e^{\a r} |\sZ^n_r|^2 + \int_E e^{\a r} |\sV^n_r(e)|^2 \n(\mathrm de) \)\mathrm dr  \],
\\
%
		\end{array}
	\end{equation*} 
where we haved used Lemmas  \ref{Y-esti} and  \ref{CV-Le-ZK-1}.
	%

  Taking $\a$   large enough, we get
\begin{equation*} 
	\begin{array}{ll}
		\ns\ds\!\!\!  |\sY^n_s|^2+ 
		\dbE^{\cF_ s^\lambda}\[ \int_ s^T\big(  |\sZ^n_r|^2 +     \|\sV^n_r(\cd)\|^2_{\nu,2}   \big)\mathrm dr  \]
		\les   
		 C\dbE^{\cF_ s^\lambda}\[ \int_s^T  \sY^n_r\mathrm d\overline \D_r\]
		+C\dbE^{\cF_s^\lambda}\big[| \overline \D_T|^2   \big]
		+C\dbE^{\cF_s^\lambda}\big[|  \sD^n_s|^2   \big].\\
		%
		%
	\end{array}
\end{equation*}

Following thr techniques used in \eqref{est7} and \eqref{estimate-M}, we can tackle with $\ds\dbE^{\cF_ s^\lambda}\[ \int_s^T  \sY^n_r\mathrm d\overline \D_r\]$ similarty to get
\begin{equation*} 
	\begin{array}{ll}
			%
			%
			\ns\ds\!\!\!   \dbE^{\cF_ s^\lambda}\[ \sup_{s \in[t,T]}  |\sY^n_s|^2  \]
			\les   C \overline \D_s^2
			+C\dbE^{\cF_s^\lambda} \big[|  \sD^n_s|^2    \big], \\
		\end{array}
\end{equation*} 
and 
\begin{equation*} 
	\begin{array}{ll}
			\ns\ds\!\!\!   
			\dbE^{\cF_ s^\lambda}\[ \int_ s^T\big(  |\sZ^n_r|^2 +     \|\sV^n_r(\cd)\|^2_{\nu,2}   \big)\mathrm dr  \]
			\les   C \overline \D_s^2
			+    \dbE^{\cF_ s^\lambda} \big[  |\sD_s^n|^2 \big].\\
			%
			%
			%
		\end{array}
\end{equation*} 

\end{proof}

 Based on the above preparations, we give the proof of \eqref{est11}. 

 	For any $(t_0,x_0),\ (t_1,x_1) \in [0, T-\delta] \times \mathbb R$, and $\lambda\in[0,1]$,
		according to the Lemma \ref{comp-hat-tilde}, we have
		\begin{equation*}
			 \lambda   \tilde Y_{t_\lambda}^{n,1,u^\lambda}  + (1-\lambda) \tilde Y_{t_\lambda}^{n,0,u^\lambda} - Y_{t_\lambda}^{n,\lambda,u^\lambda}
			=  \tilde \cY_{t_\lambda}^{n,u^\lambda} - Y_{t_\lambda}^{n,\lambda,u^\lambda}
			\les  \overline \cY_{t_\lambda}^{n,u^\lambda}  - Y_{t_\lambda}^{n,\lambda,u^\lambda}.  
		\end{equation*}

	\no	Furthermore, from the definition of $ \overline \cY ^{n,u^\lambda} _\cd $ and   Lemma \ref{esti-barY-lambdaY},
		$$
		 \overline \cY_{t_\lambda}^{n,u^\lambda}  - Y_{t_\lambda}^{n,\lambda,u^\lambda }
		 \les     \wt \sY ^{n,u^\lambda} _{t_\lambda} +\overline \D_{t_\lambda}   - Y_{t_\lambda}^{n,\lambda,u^\lambda }
		 \les 
		 C \overline \D_{t_\lambda} 
		 +   \( \dbE \big[  |\sD_{t_\lambda}^n| ^2 \big]\) ^\frac12,
		$$ 
		that is,
		\begin{equation}\label{est10}
			\lambda   \tilde Y_{t_\lambda}^{n,1,u^\lambda}  + (1-\lambda) \tilde Y_{t_\lambda}^{n,0,u^\lambda} - Y_{t_\lambda}^{n,\lambda,u^\lambda}  \les C_\d \lambda (1-\lambda)\big( |t_1-t_0|^2 + |x_1-x_0|^2  \big)
			+   \( \dbE \big[  |\sD_{t_\lambda}^n| ^2 \big] \) ^\frac12.
		\end{equation}
		Then the desired \eqref{est11} is proved.

		According to  \eqref{est1111} and \eqref{est2222}, Proposition \ref{Pro-convex-W-n} is proved. For Theorem \ref{semicon}, we only need to take the limit $n\to \i$ in Proposition \ref{Pro-convex-W-n}.

 \subsection{The Lipschitz continuity of $W(\cd,\cd)$ in $(t,x)$}
  
 In this part, we study the Lipschitz continuity of $W(\cd,\cd)$ with respect to $(t,x)$. For this, the transformation introduced in the previous part will be used again.

 Taking $\lambda=0,\ i=1$ in Subsection \ref{semi-concavity}, we get a transformation $\t_1^0(s)$ from \eqref{trans}, which is the only one used in the following. So we denote $\tau(\cd):= \tau_1^\lambda(\cd) $. Obviously,  $\tau$ maps $ [t_1,T]$ to $[t_0,T]$, and   $\dot{\tau} = \displaystyle \frac{\mathrm d}{\mathrm  ds}  \tau(s)  =  \frac{T-t_0}{T-t_1}$.
 The rest settings are similar to Subsection \ref{semi-concavity} (keeping in mind that $\lambda=0$). Then, the meanings of $ \dbB_s^1,\ \dbN_s^1,\ \tilde\dbN_s^1,\ g_\t,\ \mathbf F^1,\ \sF_s^1,\ \cU_{t_0,T},\ \cU_{t_1,T} $ are recognized.

 In this framework, some  results can degenerate into the following ones. The first one is from Lemma \ref{tau_pro}. We also refer to  \cite{BHL-2011, J-2013}.

	\begin{lemma}\sl
		 There exists the constant $C_{\delta}>0$ only depending on $\delta$, such that
		 $$
		 |\varrho(s) -s| + |\frac{1}{\dot \tau }-1| + |\frac{1}{\sqrt{\dot \tau  }}-1|  \les  C_\delta |t_0-t_1|,\quad s\in[t_0,T].
		 $$
	\end{lemma}

According to Lemma \ref{x-cx}, we have 

	\begin{lemma}\sl
		  Suppose {\rm \textbf{(H$_1$)}}-{\rm \textbf{(H$_3$)}} and {\rm \textbf{(H$_4$)}-(i),(ii)} hold,  
 for all $    p\ges 1$ and $s\in[t_0,T]$, we have, $\dbP$-a.s.,
		$$
		\mathbb E ^{\cF_s^0}\Big[ \underset{ \t\in [s ,T] }{ \sup }  |X_\t^{0,u^0} - \tilde X_\t^{1,u^0}|^p  \Big] 
		\leqslant  C_{T,p,\delta} \big( |t_0-t_1|^p + |X_s^{0,u^0}-\tilde X_s^{1,u^0}|^p \big), \quad \dbP\mbox{-a.s.}
		$$
	\end{lemma}

	From Lemma \ref{CV-Le-Y-1}, we have

		 \bl\label{LP-Le-Y-1-Lip}\sl
		  Under the conditions {\rm \textbf{(H$_1$)}-\textbf{(H$_3$)}}, {\rm \textbf{(H$_4$)}-(i), (ii)} and {\rm\textbf{(C)}},   
		  $$
		 |\tilde Y_s^{n,1,u^0} -   Y_s^{n,0,u^0} |\les C_\delta \D_s,\q \forall s\in[t_0,T],\  \dbP \mbox{-a.s.}$$
 
\el

	\begin{proposition}\label{join-lip-n}\sl
		Assume that {\rm \textbf{(H$_1$)}-\textbf{(H$_3$)}} and {\rm \textbf{(H$_4$)}-(i), (ii)}   hold. Then, for all $\d>0$, there exists a constant $C_{T,p,\delta}>0$ only depending on  $\delta$, the bounds and Lipschitz constants of $\sigma$, $b$ and $\gamma$, 
  such that for any 
		$t_0, t_1 \in [0, T-\delta]  $, $x_0,x_1\in\dbR^n$,
		$$
		|W^n(t_0,x_0  )-W^n(t_1,x_1  )| \les C_{ \delta} |t_0-t_1|+ |x_0-x_1| .
		$$
	\end{proposition}
	%


\begin{proof} For all $n\ges 1$ and $(t_0,x_0),(t_1,x_1)\in[0,T-\d]\times\dbR^n$, we recall 
$$\ba{ll}
\ns\ds W^n(t_0,x_0 )
=  \underset{u^0(\cd)\in \mathcal{U}_{t_0, T}^{0}}{\operatorname{essinf}}J_{n}(t_0,x_0;u^0(\cd))
= \underset{u^0(\cd)\in \mathcal{U}_{t_0, T}  ^{0}}  {\operatorname{essinf}}  Y_{t_0}^{n,0,u^0 },\\  
\ns\ds	 W^n(t_1,x_1 )
=  \underset{u^1(\cd)\in \mathcal{U}_{t_1, T}^{1} }{\operatorname{essinf}}J_{n}(t_1,x_1;u^1(\cd))
 =\underset{u^0(\cd)\in \mathcal{U}^{0} _{t_0, T}}{\operatorname{essinf}}\wt Y_{t_0}^{n,1,u^0}.\ea $$

\no Then, for any $\e>0$, there exists some $u^{1,\e}(\cd)   \in\mathcal{U}_{t_1, T}^{1} $ such that
$$  W^n(t_1,x_1 )> Y_{t_1}^{n,1,u_s^{1,\e}} -\e  =\wt Y_{t_0}^{n,1,u_s^{0,\e}}-\e  ,$$
where    $u^{0,\e}_ s =u^{1,\e}_{\varrho(s)}   \in \mathcal{U}_{t_0, T}^{0} $.

Therefore, using Lemma \ref{LP-Le-Y-1-Lip}, 
$$  W^n(t_0,x_0 )-W^n(t_1,x_1) \les Y_{t_0}^{n,0,u^{0,\e}}-\wt Y_{t_0}^{n,1,u^{0,\e}}+\e\les  C_\delta \D_{t_0}  +\e=C_\d(|t_0-t_1|+|x_0-x_1|)+\e.$$
 Due to the symmetry, for any $\e>0$,
$$  |W^n(t_0,x_0 )-W^n(t_1,x_1 ) |\les  C_\d(|t_0-t_1|+|x_0-x_1|)+\e.$$
By the arbitrariness of $\e$, the desired result holds true.
		\end{proof}

Letting  $n\rightarrow \infty $  in Proposition \ref{join-lip-n} and using Lemma \ref{Le-Vn}, we get the following result.  
	
 \begin{theorem}\label{join-lip}\sl
		Assume {\rm \textbf{(H$_1$)}-\textbf{(H$_3$)}}, {\rm \textbf{(H$_4$)}-(i), (ii)} and {\rm \textbf{(C)}}  hold true. Then, for all $\d>0$, there exists a constant $C_\d>0$, 
  such that for any 
		$t_0, t_1 \in [0, T-\delta]  $, $x_0,x_1\in\dbR^n$,
		$$
		|W (t_0,x_0 )-W (t_1,x_1 )| \les C_{ \delta} \big(  |t_0-t_1|+ |x_0-x_1| \big) .
		$$
That is to say,  the value function $W(\cdot,\cd)$ is  Lipschitz continuous  on 
		$[0, T-\delta]\times \dbR^n, $ for  all $\delta>0$.
	\end{theorem}

\section{Stochastic Verification Theorems}\label{SVT}
		In this section, we focus on the research of stochastic verification theorems of Problem (C)$_{t,x}$. The study will be carried out in two cases: classical solutions and viscosity solutions.
	\subsection{The classical solution case}
	We begin with the case when   PIDE  \eqref{HJB} admits the classical solution.
	In our framework, we try to construct an optimal feedback control of   Problem (C)$_{t,x}$  from the classical solution of PIDE   \eqref{HJB}.
	To begin with, we introduce the following definition of admissible feedback control laws.
	\begin{definition}\sl
		  Let $t\in[0,T]$. A measurable mapping $\mathbbm{u}:[t,T]\times\dbR^n\to U$ is said to be an admissible feedback control law, if for all $  x \in \dbR^n$, the following equation 
		 	\begin{equation}\label{SDEP-u}
		 			\left\{
		 			\begin{array}{ll}
		 					\ns\ds\!\!\!  \mathrm dX_s^{t,x;\mathbbm{u}} \! = \! b\big( s, X_s^{t,x;\mathbbm{u}}, \mathbbm{u}(s,
		 					X^{t,x;\mathbbm{u}}_s) \big)\mathrm ds \!+\! \sigma \big(s, X_s^{t,x;\mathbbm{u}}, \mathbbm{u}(s,
		 					X^{t,x;\mathbbm{u}}_s) \big)\mathrm dB_s\\
		 									\ns\ds\!\!\!  \hskip 1.55cm     
		 					\!+\!\int_{E} \! \gamma \big( s,  X_{s-}^{t,x;\mathbbm{u}}, \mathbbm{u}(s,
		 					X^{t,x;\mathbbm{u}}_{s-}) ,e  \big) \tilde N (\mathrm ds,\mathrm de),\\
		 					\ns\ds\!\!\!  X_t^{t,x;\mathbbm{u}}=x, \quad  (t,x) \in [0,T] \times \mathbb{R}^n,
		 				\end{array}
		 			\right.
		 		\end{equation}
		 	%
		 	%
		 admit   the unique strong  solution  $ X^{t,x;\mathbbm{u}}_\cd \in\cS_\dbF^2(t,T;\dbR^n) $. 
		 The set of all such admissible feedback control laws on $[t,T] $ is denoted by $\sU_{t,T}$.
	\end{definition}

	 For convenience, 
	  we introduce the  mapping    $\psi:[0,T]\times \dbR^n\times\dbR\times\dbR^n\times\dbS^n\to U$  such that
	$$
	\ba{ll}
	\ns\ds \psi(r,x,y,p,P)\in\argmin\dbH(r,x,y,p,P,\cd)\equiv\Big\{\bar u\in U \mid \dbH(r,x,y,p,P,\bar u)=\min_{u\in U}\dbH(r,x,y,p,P,u)\Big\} .
	\ea
	$$

	Now we present the  first main result of this subsection.
	\begin{theorem}\sl
		Assume  {\rm \bf(H$_1$)}, {\rm \bf(H$_2$)} and {\rm \bf(C)}. Let $ \dbW ( \cdot , \cdot ) \in {C^{1,2}}([0,T] \times {\mathbb R^n})$ be the classical solution of PIDE \eqref{HJB}.
		Then
		
		{\rm(i)}
		for any $(t,x) \in [0,T] \times {\mathbb R^n}$ and $u(\cd) \in \cU_{t,T}$, we have
		$ \dbW(t,x) \les J(t,x;u(\cd) );$
		
		{\rm (ii)} for any $(t,x)\in[0,T]\times\mathbb R^n$, defining $\bar{\mathbbm{u}}:[t,T]\times\dbR^n\to U  $ as
		\begin{equation}\label{OC-C}
			\bar{\mathbbm{u}}(s,y)=\psi \big( s,y,( W ,W_{x} ,W_{xx} )(s,y) \big),\q (s,y)\in[t,T]\times\dbR^n,
		\end{equation} 
		%
	%
	if $\bar{\mathbbm{u}} (\cd,\cd)\in\sU_{t,T}$, then $\bar{\mathbbm{u}} \big(\cd,\mathbb  X_\cd \big)$ is an optimal control of   Problem (C)$_{t,x}$, where $\mathbb  X_\cd=X_\cd^{t,x;\bar{\mathbbm{u}}} $ satisfies
	\eqref{SDEP-u} with $\bar {\mathbbm{u}}(\cd,\cd).$
	Moreover,  $\dbW(\cd,\cd)$ is in fact the value function $W(\cd,\cd)$, i.e.,
	$$  \mathbb{W } (t,x)=J\big(t,x;\bar{\mathbbm{u}}(\cd,\mathbb  X_\cd )\big)=W(t,x),\q (t,x)\in[0,T]\times\dbR^n.$$
	\end{theorem}
	\begin{proof}
	{\rm (i) }
	For any $(t,x) \in [0,T] \times {\mathbb R^n},$ $u(\cd)\in \cU_{t,T}$, denote $X_\cd= X_\cd^{t,x;u}$. Applying It\^o's formula to $\dbW(\cd,X_\cd)$, we have, for any $s\in [t,T]$, $\dbP$-a.s.,
	\begin{equation}\label{BSDEP-W}
		\begin{array}{ll}
			\ns\ds\!\!\! \dbW \big(s,X_s\big)= \Phi  (X_T ) 
			- \int_s^T \[  \frac{\partial }{{\partial r}} \dbW \big(r,X_r\big) + \mathbb H\big(r,X_r,\big(\dbW,\dbW_x,\dbW_{xx}\big)(r,X_r),u_r\big)\] \mathrm dr  \\ 
			\ns\ds\!\!\!  +  \int_s^T \! \!\[f\big(r,X_r,\dbW (r,X_r ), \dbW_x (r,X_r ).  \sigma (r,X_r,u_r),\cC^u \dbW(r,X_r), u_r \big) \] \mathrm dr\\
			%
			%
			\ns\ds\!\!\! - \int_s^T  \dbW_x (r,X_r  ). \sigma\big (r,X_r,u_r\big)  \mathrm dB_r 
			-\int_s^T\int_E \big( \dbW (r,X_{r-}+\gamma(r,X_r,u_r,e)) - \dbW (r,X_{r-})\big) \tilde N(\mathrm dr,\mathrm de).\\
		\end{array}
	\end{equation}

	On the other hand, for any  $ (t,x) \in [0,T] \times \mathbb R^n$, $n \in \mathbf \dbN $, we consider the following  penalized BSDEs,
	\begin{equation}\label{Yn}\ba{ll}
		\ns\ds  {}^nY_s = \Phi \big(X_T\big) + \int_s^T f\big(r, X_r,{}^nY_r,{}^nZ_r,
		\int_E l(e)  {}^n V_r(e) \nu(\mathrm de), u_r \big)\mathrm dr - n\int_s^T \big(   h(r,X_r )-{}^nY _r\big)^ - \mathrm dr\\
		\ns\ds \hskip1.05cm - \int_s^T \! {{}^nZ _r}  \mathrm dB_r
		-\int_s^T \!\! \int_E {}^n V _r(e) \tilde N(\mathrm dr,\mathrm de),
		\hskip0.2cm  s\in[t,T].
		\ea\end{equation}
	By the fact that $\dbW(\cd,\cd)$ being the classical solution  of \eqref{HJB}, we get the following two cases,
	
	{\rm Case (a).} at any point $(t,x)\in [0,T]\times\dbR^n$ such that $(\dbW -h)(t,x)= 0$,
	$$- \frac{\partial}{\partial t}\dbW(t,x) -\inf_{u\in U}  \mathbb H \big(
	t,x, ( \dbW  ,\dbW_x  , \dbW_{xx} ) (t,x),u \big) \les 0;$$
	
	{\rm Case (b).} at any point $(t,x)\in [0,T]\times\dbR^n$ such that $(\dbW -h)(t,x)< 0$,
	$$- \frac{\partial}{\partial t}\dbW(t,x) -\inf_{u\in U}  \mathbb H \big(
	t,x, ( \dbW  ,\dbW_x  , \dbW_{xx} ) (t,x),u \big) = 0.$$
	No matter (a)  or  (b),   for any $r\in[t,T]$ and $u(\cd)\in \cU_{t,T}$, $n \in \mathbf \dbN$, we have
	 \begin{equation}\label{compar-f}\ba{ll}
		\ns\ds\!\!\!
		f\big(r,X_r,\dbW\big(r,X_r\big),\dbW_x\big(r,X_r\big). \sigma (r,X_r,u_r), \cC^u \dbW (r,X_r)   
		  ,u_r\big)\\
		\ns\ds\!\!\!   
		  - \frac{\partial }{{\partial r}}\dbW\big(r,X_r\big)
		  -\mathbb H\big(r,X_r, (\dbW ,\dbW_x ,\dbW_{xx}) (r,X_r ),u_r\big)\\
		\ns\ds\!\!\!
		\les f\big(r,X_r,\dbW\big(r,X_r\big),\dbW_x\big(r,X_r\big). \sigma (r,X_r,u_r), \cC^u \dbW (r,X_r)   
		,u_r\big) 
		\\
		\ns\ds\!\!\!\q
		 - \frac{\partial }{{\partial r}}\dbW\big(r,X_r\big)-   \inf_{u\in U}\mathbb H\big(r,X_r, (\dbW ,\dbW_x ,\dbW_{xx}) (r,X_r ),u_r\big)\\
		%
		%
		%
		%
		\ns\ds\!\!\! \les f\big(r,X_r,\dbW\big(r,X_r\big),\dbW_x\big(r,X_r\big) .\sigma (r,X_r,u_r), \cC^u \dbW (r,X_r)   
		,u_r\big)
		-n\big(h (r,X_r)- \dbW (r,X_r )   \big)^ -. \\
		%
		\ea\end{equation} 
	Therefore, by using the comparison theorem to \eqref{BSDEP-W} and \eqref{Yn}, for  all $n \in \dbN$,  we get
	\begin{equation}\label{W-Yn} 
		\dbW \big(s,X _s\big) \les {}^n{Y_s },\quad  s \in [t,T],\ \dbP\mbox{-a.s.}
		\end{equation}

	Notice that Lemma \ref{well}, letting $n\to \i$ in \eqref{W-Yn},  
	we get 
	%
	$$\dbW  (s,X  _s ) \les Y _s^{t,x,u},\q s \in [t,T], \  \dbP\mbox{-a.s.}$$
	  Here $Y_\cd^{t,x,u}$ is the first component of the solution of RBSDE with jumps \eqref{RBSDEP}.

	Especially,  when $s=t$,
	\begin{equation}\label{Step1}
		\dbW (t,x) \les {Y_t^{t,x,u} } = J (t,x;u(\cd) ),\quad\mathrm{for\ any}\ u(\cd) \in \cU_{t,T}.\end{equation}

	{\rm(ii) } Let $\mathbb X $  
	 be the  solutions  of SDE with jumps \eqref{SDEP-u} with $\mathbbm{u}(\cd,\cd)$ replaced by $\bar{\mathbbm{u}}(\cd,\cd)\in\sU_{t,T}$ introduced in \eqref{OC-C}.
	And the following RBSDE with jumps
		\begin{equation}\label{RBSDEP-u}
				\left\{
				 \begin{array}{ll}
					 	\ns\ds \!\!\!\! 		{\rm(\romannumeral1)} \ 
					 	(\dbY ,\dbZ ,\dbV,\dbA ) \in \sS_\mathbb F ^2[t,T];\\
					 					%
					 	\ns\ds\!\!\!\! 			{\rm(\romannumeral2)} \     \dbY_s   =   \Phi   ( \dbX_T  )  
					 	+   \int_s^T   f \big(  r, \dbX_r , \dbY_r , \dbZ_r , \int _E l(e) \dbV_r (e)  \nu(\mathrm de), \bar {\mathbbm{u}}(r,	\dbX _r \big) \mathrm dr  -  ( \dbA_T   - \dbA _s  ) 
					 	-   \int_s^T \dbZ_r  \mathrm dB_r
					 					\\
					 	\ns\ds\!\!\!\! \hskip 1.5cm- \int_s^T   \int_E \dbV_r(e) \tilde N(\mathrm dr,\mathrm de) ,\quad  s\in [t,T];\\
					 	\ns\ds\!\!\!\! 			{\rm(\romannumeral3)} \   \dbY_s   \leqslant   h  ( s,\dbX_s ), \mbox{ a.e. } s \in  [t,T];\\
					 	\ns\ds\!\!\!\! 			{\rm(\romannumeral4)} \   \int_t^T \big(  h (s, \dbX_s  )-\dbY_s \big) \mathrm d\dbA_s =0 ,
					 				\end{array}
				 			\right.
				 		\end{equation} 
		admits  the unique $\dbF$-adapted  solution  $ (\dbY ,\dbZ ,\dbV,\dbA )$.
	Furthermore, based on Lemma \ref{well-C}, \eqref{RBSDEP-u} also has the penalized equation as follows.
	\begin{equation*} 
		\ba{ll}
		\ns\ds {}^n	\dbY_s = \Phi   ( \dbX_T  )  + \int_s^T   f \big(  r, \dbX_r , {}^n \dbY_r , {}^n\dbZ_r , \int _E l(e){}^n \dbV_r (e)  \nu(\mathrm de), \bar {\mathbbm{u}}(r,	\dbX _r) \big)  \mathrm dr
		-\int_s^T   n\big(   h  ( r,\dbX_r ) -{}^n \dbY_r \big)^- 
		\mathrm dr  \\
		\ns\ds\hskip1.1  cm  - \int_s^T  {}^n   \dbZ_r    \mathrm dB_r - \int_s^T   \int_E    {}^n \dbV_r  (e)   \tilde N(\mathrm dr,\mathrm de) ,\quad s \in [t,T],\ n\in\dbN.
		\ea\end{equation*} 
	%

	%
	In  Case (a), for  $(t,x)\in[0,T]\times\dbR^n$ such that    $\dbW(t,x)= h(t,x)$, combined with the obstacle condition in RBSDE with jumps \eqref{RBSDEP-u},
	we get
	$$\dbW(t,x)= h(t,x)\ges \mathbb Y_t=J\big(t,x;\bar{\mathbbm{u}}(\cd,\mathbb X_\cd)\big),\q (t,x)\in[0,T]\times\dbR^n.$$
	%
	
	In  Case (b), applying It\^o's formula to $\dbW(s,\mathbb X_s )$ on $[t,T]$, we have
  $$\ba{ll}
	\ns\ds\!\!\!  \dbW\big(s,\mathbb X_s\big)   = \Phi\big ({{\mathbb X_T}}\big) 
	-  \int_s^T \(  \frac{\partial }{{\partial r}}\dbW\big(r,\mathbb X_r\big)+ \mathbb H\big(r,\mathbb X_r,
 (\dbW ,	\dbW_x ,	\dbW_{xx} ) (r,\mathbb X_r\big)
	,\bar{\mathbbm{u}}(r,\mathbb X_r)
	\big) \) \mathrm dr \\
	\ns\ds\!\!\! + \int_s^T \!\!  f\big(r,\mathbb X_r, \dbW\big(r,\mathbb X_r\big),
	\dbW_x\big(r,\mathbb X_r\big) . \sigma\big (r,\mathbb X_r,\bar{\mathbbm{u}}(r,\mathbb X_r)\big),
 	\cC^{ \bar{ \mathbbm u}} \dbW(r,\mathbb X_r),
	\bar{\mathbbm{u}}(r,\mathbb X_r)\big) \mathrm dr\\
	%
	%
	\ns\ds\!\!\! - \int_s^T \dbW_x\big(r,\mathbb X_r\big).   \sigma\big (r,\mathbb X_r, \bar{\mathbbm{u}}(r,\mathbb X_r)\big) \mathrm dB_r
	 -\int_s^T\int_E \big( \dbW\big(r,\mathbb X_{r-} +\gamma(r,\mathbb X_{r-},u_r,e)\big) -  \dbW(r,\mathbb X_{r-})\big) \tilde N(\mathrm dr,\mathrm de).\\
	\ea$$
	 \normalsize
	Note that \eqref{OC-C}  and  Case (b) make ``$\les$" in \eqref{compar-f}  become ``$=$". Following the procedures in (i) and the uniqueness of the solution of BSDE with jumps, for all $n\in\mathbb N $, we get
	$$ \dbW\big(s,\mathbb X_s\big) = {}^n \mathbb Y_s,\q s\in[t,T],\hskip0.2cm \dbP\mbox{-a.s.}$$
	Similarly to (i), 
	letting $n\to\i$  and $s=t$, we have
	$ \dbW\big(t,x\big) = \mathbb Y_t $.

	Finally, combined with \eqref{Step1},  for any $(t,x)\in[0,T]\times\dbR^n$, we get
	$$ \dbW(t,x) =J\big(t,x;\bar{\mathbbm{u}}(\cd,\mathbb X_\cd) \big)=\essinf_{u(\cd)\in\cU_{t,T}}J\big(t,x;u(\cd)\big)=W(t,x).$$
	That is,  the classical solution $ \dbW(\cd,\cd)$ of PIDE  \eqref{HJB} is indeed the value function $W(\cd,\cd)$ of Problem (C)$_{t,x}$,
	and
	$\bar{\mathbbm{u}}(\cd,\mathbb X_\cd)$ is the optimal control of Problem (C)$_{t,x}$.
	\end{proof}

	\begin{remark}\sl
	  We make some explanations on  the collection $\sU_{t,T}$ of the admissible feedback control laws.  Let us first give a relatively direct condition
	 \begin{description}
		 	\item[{\bf(H$_7$)}] 
		 	  {\rm (i)} For every  $ (r,x,y,z,v) \in [0,T]\times\dbR^n\times \dbR \times \dbR^d \times \mathbb R $, $b(r,x,\cd)$, $\sigma(r,x,\cd)$ are Lipschitz continuous in $u\in U $. And, for any $u_1, u_2 \in U$,
		 	 $  
		 	 |\gamma (r,x,u_1,e)-\gamma (r,x,u_2,e)| \les \ell(e)|u_1-u_2|.\\
		 	 $ 
		 	  {\rm (ii)} $ {\mathbbm{u}}(\cd,\cd)\in\cL$, where
		 	 $\mathcal{L}$ is the class of measurable mappings $\mathbbm{u}:[0,T]\times\mathbb R^n \rightarrow U$ satisfying the two properties: {\rm (a)}  for every fixed  $ x  \in \dbR^n$, $u(\cd,x)$  is continuous in   $r\in[0,T]$; {\rm (b)} for every $ r\in[0,T],$ $  \mathbbm{u} (r, \cd )$ is Lipschitz continuous in $x\in\dbR^n$. 
		 \end{description} 
 If  {\rm \bf (H$_1$)}, {\rm \bf (H$_2$)}  and {\rm \bf (H$_7$)} hold   true, it is easy to check that $ \mathbbm{u}  (\cd,\cd)\in\sU_{t,T}$.   These are essentially the Lipschitz conditions on the coefficients, which may be hard for us to seek  such  
 $ {\mathbbm{u}}(\cd,\cd)\in\cL$.
However, as  the development of the theory of SDEs with discontinuous coefficients (such as the singular or the irregular coefficients),     there are many researches  devoted to seeking a strong solution of SDE with jumps under weak
assumptions, such as \cite{QZ-2008, XZ-2020, PSX-2021}.
Therefore, 
with the help of these studies,    some appropriately mild conditions may be posed to    ensure $\sU_{t,T}$ to be  non-empty.
 We just keep in mind that the constructed  $ {\mathbbm{u}}(\cd,\cd)$ must  lie in $\sU_{t,T}$ when  
  dealing with the specific (or computable) control
problems.
  
	\end{remark}

	\subsection{The viscosity solution case}

	  We now look at the case when the solution of PIDE is not necessarily smooth, i.e., viscosity solution. 
	  In order to give a clear   statement  of the main result, we rewrite \eqref{HJB} as follows.
	  	\begin{equation*} 
	  	\left\{
	  	\ba{ll}
	  	\ns\ds\!\!\!\! \max\Big\{ \! W(t,x)-h(t,x),- \frac{\partial}{\partial t}W(t,x) \!-\!\inf_{u\in U}  \mathcal H \big(t,x, (W ,W_x ,W_{xx})(t,x),\cB^{u} W(t, x),\cC^{u} W(t, x),u\big) \Big\}=0,\\
	  	\ns\ds\!\!\! \hskip12.2cm 
	  	(t,x)\in[0,T]\times\dbR^n,\\
	  	\ns\ds\!\!\!\! W(T,x) = \Phi(x),\quad x\in \mathbb R^n.
	  	\ea
	  	\right.
	  \end{equation*}
	 First, it is necessary to  introduce the  definition of second-order right parabolic superdifferentials (refer  to \cite{ZYL}).

	\begin{definition}  Let $(t,x)\in [0,T]\times \dbR^n$ and $w\in C([0,T]\times\dbR^n ;\dbR)$, the \emph{second-order parabolic superdifferential} of $w$ at  $(t, x)$ is defined as
		\begin{equation*}\label{super}\ba{ll}
			\ns\ds D^{1,2,+}_{t+,x}w(t,x):=  \Big\{(q,p,P)\in \dbR\times \dbR^n\times \dbS^n\Big|\limsup\limits_{s\to t^+, y\to x}\frac 1{|s- t|+|y- x|^2}\big[w(s,y)-w(t,x)\\
			\ns\ds\hskip 6.65cm -q(s-t)-\lan p,y-x\ran-\frac 12(y-x)^\top P(y-x) \big]\les 0\Big\}.
			\ea\end{equation*}
		%
		%
		%
	\end{definition}
	
The following is about the ‌characterization of $D^{1,2,+}_{t+,x}w(t,x)$.
	

	\begin{lemma}\label{L11}\sl Let  $w\in C([0,T]\times \mathbb{R}^n)$ and $(t_0,x_0)\in [0,T)\times \dbR^n$ be given. Then,
		$(q,p,P)\in D_{t+,x}^{1,2,+}w(t_0,x_0)$ if and only if there exists a function $\f\in  C^{1,2}([0,T]\times \mathbb{R}^n)$ such that, for any $ (t,x)\in [t_0,T]\times \dbR^n$, $(t,x)\neq (t_0,x_0)$,
		$\f(t,x)>w(t,x),$  and $$\big(\f(t_0,x_0),\f_t(t_0,x_0),\f_x(t_0,x_0),\f_{xx}(t_0,x_0)\big)=\big(w (t_0,x_0),q,p,P\big).$$
		%
		%
		Moreover, if for some $k\ges 1$, $(t,x)\in [0,T]\times \dbR^n$,
		\begin{equation}\label{poly}
			|w(t,x)|\les C(1+|x|^k),
		\end{equation}
		then we can choose $\f$ such that $\f$, $\f_t$, $\f_x$, $\f_{xx}$  also satisfy \eqref{poly} with different constants $C$.
	\end{lemma}
	%
	%
	Next, similar to \cite{LLW-2025},
	we gather the functions $\varphi$ in Lemma \ref{L11} as follows.
	$$
	\begin{array}{ll}
		\ns\ds\!\!\!
		 \mathfrak{I} (q,p,P;w)(t,x) = \Big\{  \f\in  C^{1,2}([0,T]\times \mathbb{R}^n) \mid   w-\f  \mbox{ attains a  strict maximum over } [0,T]\times \mathbb{R}^n, \\
		 \ns\ds\!\!\! \hskip 6.85cm
		 \mbox{and } \big(\f(t,x),\f_t(t,x),\f_x(t,x),\f_{xx}(t,x)\big)=\big(w(t,x),q,p,P\big) \Big\}.
	\end{array}
	$$

	%
	%
	%
	%
%
	
	The following result can be found in \cite{YZ-1999, GSZ1}.
	
	\begin{lemma}\label{le22}\sl
	Given $g\in C([0,T])$. Let's extend $g$ to $(-\i,+\i)$ by setting 
	$ g(t)=\left\{ \ba{ll} \ns\ds\!\!\! g(0),\ t<0\\
	\ns\ds\!\!\! g(t),\   t\in[0,T].\\
	\ns\ds\!\!\! g(T),\ t>T\ea\right.$	Suppose that for all $\d\in(0,T),$  there is a function  $\rho(\cd)\in L^1(0,T-\d;\mathbb{R})$ and some $h_2>0$, such that
	\begin{equation*}\frac{g(t+h)-g(t)}{h}\les \rho(t),\ \mbox{a.e.}, \ t\in [0,T-\d),\ h\les h_2,\end{equation*}
	then,
	$$g(\b)-g(\a)\les \int_\a^\b \limsup\limits_{h\to 0^+}\frac{g(t+h)-g(t)}{h}dt,\qq 0\les \a <\b\les T-\d. $$
	\end{lemma}
	%
	 Now, we give the stochastic verification theorem within the framework of viscosity solutions.
	
	\begin{theorem}\label{SVT-VS}\sl Assume that the conditions {\rm \textbf{(H$_1$)-(H$_5$)}} and {\rm \bf(C)}  hold true.
		   Let $W(\cd,\cd)\in C_p([0,T]\times\dbR^n)$ be the viscosity solution of   PIDE     \eqref{HJB}. For  $(t,x)\in [0,T]\times \dbR^n$, let $(\bar u(\cd)  , X_\cd^{t,x,\bar u}  )$ be the admissible pair and  $(  Y_\cd^{t,x,\bar u}  ,  Z_\cd^{t,x,\bar u}  ,  V_\cd^{t,x,\bar u},  A_\cd^{t,x,\bar u}  )$ solve RBSDE with jumps \eqref{RBSDEP}  under the control process $\bar u(\cd) \in\cU_{t,T}$. Assume that  there exists  a triple of
		$
		\big(\bar q,\bar p, \bar P\big)\in \cM^2_\dbF(t,T;\dbR)\times \cM^2_\dbF(t,T;\dbR^n)\times \cM^2_\dbF(t,T;\dbS^n)
		$ such that
		$$\left\{\ba{ll}
		\ns\ds\!\!\! {\rm(i)}\   \big(\bar q(s),\bar p(s) , \bar P(s) \big)\in D^{1,2,+}_{t+,x}W\big(s,  X _s ^{t,x,\bar u}\big) , \ {\rm a.e.}  \ s\in [t,T],\  \dbP\mbox{-a.s.};\\
		\ns\ds\!\!\! {\rm(ii)}\     \bar p(s).  \si\big(s,  X _s^{t,x,\bar u} ,\bar{u}_s\big)=  Z _s^{t,x,\bar u} ,\ {\rm a.e.} \ s\in [t,T],\   \dbP\mbox{-a.s.};\\
		\ns\ds\!\!\!{  \rm(iii)}\ W \big(s, X_{s-}^{t,x,\bar u} +\gamma(s, X_{s-}^{t,x,\bar u},\bar u_s, e)\big)-W(s, X_{s-}^{t,x,\bar u})
		 =   V_s^{t,x,\bar u}(e),\ e\in E,\ {\rm a.e.},\ s\in[t,T],\ \dbP\mbox{-a.s.}; \\
		\ns\ds\!\!\! {\rm(iv)} \ \mbox{For some } \f(\cd,\cd) \in    \mathfrak{I} (\bar q,\bar p, \bar P; W)\big(s,  X _s ^{t,x,\bar u}  \big),  
		\\
		 \ns\ds\!\!\!  \hskip 0.74cm \dbE   \[   \int_t^T \!\!\! \Big( \bar q(s) \!+  \! \mathcal H \big(s,  X_s^{t,x,\bar u} ,  Y_s^{t,x,\bar u} ,\bar p(s) , \bar P(s), \cB^{\bar u } \varphi(s,  X_s^{t,x,\bar u} ),\cC^{\bar u } \varphi(s,  X_s^{t,x,\bar u} ),\bar u_s\big)\Big) \mathrm ds   \]  \! \les\!  0.
		\ea\right.$$
		Then, $\bar u (\cd) $ is an optimal control of Problem (C)$_{t,x}$.
	\end{theorem}
		\begin{proof}
		Firstly, from the uniqueness of the viscosity solution of  \eqref{HJB} (Theorem \ref{1geq2}), we know, for any $(t,x)\in[0,T]\times\dbR^n$ and $u(\cd)\in\cU_{t,T}$
		\begin{equation}\label{ee1}
			W(t,x) \les J (t,x;u(\cd)  ).
		\end{equation}
		%
		
		%
		
		\par For any fixed  $t_0\in[t,T]$ satisfied (i) and given $\omega_0=(\o'_0,p'_0)\in\O$, we introduce a new  probability space $\big(\O,\cF,\dbP(\cd\mid\cF_{t_0}^t)(\o_0)\big)$ equipped with a new filtration $\{\cF_s^{t_0}\}_{t_0\les s\les T}$.  
		Here,  for any $\varsigma ,  \varsigma' \in [t,T]$ satisfied $\varsigma \les \varsigma' $,  define $\mathcal F _{\varsigma'}^{\varsigma}$ as $\mathcal{F}_{\varsigma'}^{B,\varsigma}\otimes\mathcal{F}_{\varsigma'}^{\varsigma}$ augmented by all the $\mathbb{P}$-null sets in $\mathcal{F}$, where $$\mathcal{F}_{\varsigma'}^{B, \varsigma} :=  \sigma\{B_r: \varsigma \les r\les \varsigma' \} \vee\mathcal{N}_{\mathbb{P}_1},\q
		\displaystyle \mathcal F_{\varsigma' }^{\mu,\varsigma }  :=   ( \bigcap_{s>\varsigma' }   \sigma\{ N( (\tau,r]\times\Delta),\varsigma  \les\tau\leqslant r\les s, \Delta\in \mathcal{B}(E)\}\Big)\vee\mathcal{N}_{\mathbb{P}_{2}}.$$

		 \no In this new probability space, according to the definition of regular conditional probability, for fixed $\o_0$, the regular conditional probability $\dbP(\cd\mid\cF_{t_0}^t)(\o_0)$  is probability measure.
		The expectation $\dbE_{t,t_0}$ defined on the new space is related to the expectation $\dbE$ of the initial probability space as follows,
		$$
		\dbE_{t,t_0}[\cd]=\dbE\big[\cd\mid\cF_{t_0}^t (\o_0) \big](\o_0).  
		$$

		\no Furthermore, it is easy to check that
		$\{B_s\}_{s\ges t_{0}}$ is still a standard Brownian motion with $B_{t_0}=B_{t_0}(\omega'_{0})$ almost surely, and $\mu$ is a Poisson random measure restricted to $[t_0,T]$ with $N((t,t_0],\Delta)=N(p'_0,(t,t_0]\times\Delta)$, $\Delta\in\mathcal{B}(E).$
		And, the control process $\bar {u}(\cd)$ is adapted to the new filtration.
		For $\o_0$, the process $  X ^{t,x,\bar u} _\cd $ is a solution of \eqref{state} on $[t_0,T]$  in $\big(\O,\cF,\dbP(\cd\mid\cF^{t}_{t_0})(\o_0)\big)$ with the  initial condition $  X^{t,x,\bar u}  _{t_0}=  X^{t,x,\bar u}  _{t_0}(\o_0)$.

		From now on, we keep in mind that the above $(t_0,\o_0)$ is fixed. From  Lemma \ref{L11} and $\big(\bar q(t_0),\bar p(t_0) , \bar P(t_0) \big)\in D^{1,2,+}_{t+,x}W\big(t_0,  X _{t_0} ^{t,x,\bar u}\big)$, we know there exists a function $ \bar \f ( t_0, \o_0,\cd,\cd) \in    \mathfrak{I} (\bar q,\bar p, \bar P; W)\big(t_0,  X^{t,x,\bar u}  _{t_0} \big)$. 
		Furthermore, the linear growth of $W(\cd,\cd)$ in \eqref{val-lip}  implies us
		$ \bar \f,$ $ \bar \f_t,$ $  \bar\f_x,$ $ \bar\f_{xx}$ are also linear growth in $x$, i.e.,
		\begin{equation}\label{phi-lg}
			| \bar\f( t_0, \o_0,t,x)|+| \bar\f_t( t_0, \o_0,t,x)|+| \bar\f_x( t_0, \o_0,t,x)|+|\bar \f_{xx}( t_0, \o_0,t,x)|\les C(1+|x|),\q (t,x)\in[0,T]\times\dbR^n.
		\end{equation}
		Obviously,  $\bar \varphi$ is deterministic  on the space $\big(\O,\cF,\dbP(\cd\mid\cF_{t_0}^t)(\o_0)\big)$.
%
%
%
%
		Then, for any $h>0$, applying It\^o's formula to $\bar\f\big(\cd,   X^{t,x,\bar u} _\cd \big)$ on $[t_0,t_0+h]$, we  obtain
		$$
		\ba{ll}
		\ns\ds\!\!\!    \bar \f\big(t_0+h, X ^{t,x,\bar u}_{t_0+h} \big)-\bar\f\big(t_0,  X^{t,x,\bar u} _{t_0}\big)\\
		\ns\ds\!\!\!    =\int_{t_0}^{t_0+h}\( \bar\f_t\big(r,  X_r^{t,x,\bar u} \big)+  \cL^{\bar u}\bar\f(r,  X_r^{t,x,\bar u})
		 +\mathcal B^{\bar u}\bar \f (r,X_r^{t,x,\bar u})
		\)\mathrm dr
		 +\int_{t_0}^{t_0+h}   \bar\f_x\big(r,X_r^{t,x,\bar u} \big).  \si\big(r,X_r^{t,x,\bar u} , \bar{u}_r\big)  \mathrm dB_r \\
		%
		%
		\ns\ds\!\!\!\hskip 0.425cm
		+\int_{t_0}^{t_0+h}\int_E   \(\bar\varphi \big( r, X_{r-}^{t,x,\bar u}
		+\gamma (r,X_{r-}^{t,x,\bar u},\bar u _r,e ) \big) 
		- \bar\varphi (r,X_{r-}^{t,x,\bar u} ) \) \tilde N(\mathrm dr,\mathrm de).
		\ea
		$$
		%
		%
	\par	Taking the expectation $\dbE_{t,t_0}[\cd]$, we get
		\begin{equation}\label{est123}
		\ba{ll}
		\ns\ds\!\!\!  \dbE_{t,t_0}\Big[W\big(t_0+h,  X^{t,x,\bar u}  _{t_0+h}\big)-W\big(t_0,   X ^{t,x,\bar u} _{t_0}\big) \Big]
		\les  \dbE_{t,t_0}\Big[\bar\f\big(t_0+h,    X^{t,x,\bar u}  _{t_0+h}\big)-\bar\f\big(t_0,   X^{t,x,\bar u}  _{t_0+h}\big) \Big]\\
		%
		%
		\ns\ds\!\!\!  =  \dbE_{t,t_0}\Big[\int_{t_0}^{t_0+h}\( \bar\f_t\big(r,  X_r^{t,x,\bar u} \big)+  \cL^{\bar u}\bar\f(r,  X_r^{t,x,\bar u})
		+\mathcal B^{\bar u}\bar \f (r,X_r^{t,x,\bar u})
		\)\mathrm dr \Big]. \\
		%
		%
		\ea
	\end{equation}

		Next, we will estimate  $\displaystyle \limsup\limits_{h\to 0^+} \frac 1h \dbE_{t,t_0}\Big[W\big(t_0+h,\bar X _{t_0+h}\big)-W\big(t_0,\bar X _{t_0}\big) \Big]$. Before that, we make the following preparations. 
		 For all $h\in(0,T-t_0)$, according to \eqref{phi-lg} and Lemma \ref{Le-SDE},
		
		$$
		\begin{array}{ll}
			\ns\ds\!\!\! \int_{t_0}^{t_0+h}   | \mathbb E _{t,t_ 0 } [ \cL^{\bar u} \bar \varphi (r,  X_r^{t,x,\bar u}) ]| \mathrm dr\\
			\ns\ds\!\!\! \les \mathbb E _{t,t_ 0 }  \[  \int_{t_0}^{t_0+h}  \big| \bar\f_x\big(r, X_r^{t,x,\bar u} \big) . b\big(r,  X_r^{t,x,\bar u} , \bar{u}_r\big)   + \frac12 \tr\big(\sigma\sigma^\top  (r,  X_r^{t,x,\bar u} , \bar{u}_r ) \bar\f_{xx} (r,  X_r^{t,x,\bar u}  ) \Big| \mathrm dr  \] \\
			 %
			 \ns\ds\!\!\!	   \les  C  \mathbb E _{t,t_ 0 } \[ \int_{t_0}^{t_0+h}   \big(1+ |  X_r^{t,x,\bar u}|  + |  X_r^{t,x,\bar u}|^2  \big)           \mathrm dr \]  < +\i, 
		\end{array}
		$$
		   and  
		$$
		\begin{array}{ll}
			%
			\ns\ds\!\!\!	 \int_{t_0}^{t_0+h}        \Big| \mathbb E _{t,t_ 0 }  [     \cB^{\bar u} \bar \f (r, X_r^{t,x,\bar u})       ]   \Big| \mathrm dr
			\\
			\ns\ds\!\!\!	  \les   \mathbb E _{t,t_ 0 } \[ \!\int_{t_0}^{t_0+h}\!  \Big |
			 \int_E \!\int_0^1 \! \big \langle  (1 \! -\! \theta)   \bar\varphi_{xx} \big(r, X_r^{t,x,\bar u}\! +\! \theta\gamma(r, X_r^{t,x,\bar u},\bar u_r,e  ) \big ) \gamma(r, X_r^{t,x,\bar u},\bar u_r,e  ), \gamma(r, X_r^{t,x,\bar u},\bar u_r,e  )     \big   \rangle  \nu(\mathrm de)
			      \Big| \mathrm dr \! \]	\\
			\ns\ds\!\!\!   \les C \mathbb E _{t,t_ 0 } \[\int_{t_0}^{t_0+h}   \Big| 
			\int_E        \big(1+ | X_r^{t,x,\bar u}|  + |\gamma(r, X_r^{t,x,\bar u},\bar u_r,e  ) |  \big) |\gamma(r, X_r^{t,x,\bar u},\bar u_r,e  )|^2   \nu(\mathrm de)
			   \Big| \mathrm dr\]	\\
			%
			%
				\ns\ds\!\!\!   \les C \int_{t_0}^{t_0+h}   \Big| \mathbb E _{t,t_ 0 } \[ 
				\int_E       \ell^2(e)   (1+\ell(e)+ | X_r^{t,x,\bar u}|     )     \nu(\mathrm de)
				\]   \Big| \mathrm dr	<+ \i. \\
		\end{array}
		$$
		Then, from the Lebesgue differentiation theorem, we get
%
%
%
		%
			\begin{equation}\label{ineq-111}
			\ba{ll}
			\ns\ds\!\!\! \limsup\limits_{h\to 0^+} \frac 1h\dbE_{t,t_0}\Big[W\big(t_0+h,  X ^{t,x,\bar u}_{t_0+h}\big)-W\big(t_0,  X^{t,x,\bar u} _{t_0}\big) \Big]\\
			\ns\ds\!\!\! \les\limsup\limits_{h\to0^+}  \frac 1h  \dbE_{t,t_0}\Big[ \!  \int_{t_0}^{t_0+h}\! \(\bar\f_t\big(r, X_r^{t,x,\bar u} \big)+\cL^{\bar u}\bar\f (r, X_r^{t,x,\bar u}  )
			+\cB^{\bar u} \bar \f (r, X_r^{t,x,\bar u}) \) \mathrm dr \! \Big]\\
			%
			%
			\ns\ds\!\!\! =  
			\bar \f_t\big(t_0,  X^{t,x,\bar u}_{t_0}\big)+\cL^{\bar u}\bar\f (t_0, X_{t_0}^{t,x,\bar u}  )
			+\cB^{\bar u} \bar \f (t_0, X_{t_0}^{t,x,\bar u})   \\
			\ns\ds\!\!\! = \bar q(t_0) +  \bar p(t_0).  b\big(t_0, X_{t_0}^{t,x,\bar u}, \bar u_{t_0} \big) 
			+\frac12 \tr\big( \sigma\sigma^\top  (t_0,  X^{t,x,\bar u}_{t_0} , \bar{u}_{t_0} ) \bar P(t_0) \big) +\cB^{\bar u} \bar \f (t_0,  X^{t,x,\bar u}_{t_0}).\\
			%
			%
			\ea
		\end{equation}

		Next, we claim that, for all $ \d\in(0,T)$ and the previous   $t_0$ lying in $[t,T-\d)$,
		for any  $h>0 $ with  $t_0+h\les T-\d$,
		\begin{equation}\label{ineq-000}\ba{ll}
		\ns\ds {\rm(a)  }\   \frac 1h\dbE_{t,t_0}\[  W\big(t_0+h,  X^{t,x,\bar u} _{t_0+h}\big)-W\big(t_0,  X^{t,x,\bar u} _{t_0} \big)   \]  \les  C h(1+ |  X^{t,x,\bar u}_{t_0}|  ),  
		\q \dbP\mbox{-a.s.};\\
		\ns\ds {\rm(b)  }\   \frac 1h \dbE\[   W\big(t_0+h,  X^{t,x,\bar u} _{t_0+h}\big)-W\big(t_0,  X ^{t,x,\bar u}_{t_0} \big)  \]    \les C(1 +|x| ) . 
		\ea\end{equation} 
		In fact, from Theorems \ref{join-lip} and \ref{semicon}
		 as well as $\big(\bar q (t_0 ) ,\bar p(t_0) , \bar P(t_0) \big)\in D^{1,2,+}_{t+,x}W(t_0,X^{t,x; \bar u }_{t_0})$, we know, for all $\d\in(0,T)$,  for any $h\in (0,T-t_0-\d]$,
		\begin{equation*}\label{} \ds W\big(t_0+h,  X^{t,x,\bar u} _{t_0+h}\big)-W\big(t_0,  X ^{t,x,\bar u} _{t_0} \big)\les {\rm  I }+\rm{II}+\rm{ III},\end{equation*} 
		with
		$$\left\{\ba{ll}
		\ns\ds\!\!\! {\rm  I }:= C\big(1+| X^{t,x,\bar u} _{t_0+h}| \big)h,  \\
		\ns\ds \!\!\! {\rm II}:= \big\langle \bar p(t_0) ,    X^{t,x,\bar u} _{t_0+h}-  X ^{t,x,\bar u}_{t_0} \big\rangle ,\\
		\ns\ds\!\!\!  {\rm  III} := C |  X^{t,x,\bar u} _{t_0+h}-   X^{t,x,\bar u} _{t_0} |^2.\\
		\ea\right.$$

	Using  \eqref{phi-lg}, we have   $|\bar q(t_0) |+|\bar p(t_0) |+|\bar P(t_0) |\les C(1+|  X _{t_0}^{t,x,\bar u} | ).$
		Furthermore, combine with Lemma \ref{Le-SDE}, we get
		$$
		\ba{ll}
		\ns\ds \dbE _{t,t_0}\big[{\rm I }  \big]\les C  h+ C h\Big(\dbE _{t,t_0}  \Big[\sup\limits_{ r\in [t_0,t_0+h]}|  X_r^{t,x,\bar u}|^2 \Big]\Big)^{\frac{1}{2}}
		 \les  C h(1+ |  X^{t,x,\bar u}_{t_0}|  )  , \\
		%
		\ns\ds  \dbE _{t,t_0}\big[{\rm II} \big]
		%
		\les  \dbE_{t,t_0} \[ \big\langle \bar p(t_0),\int_{t_0}^{t_0+h}  \! b\big(r, X_r^{t,x,\bar u} ,\bar{u}_r\big)\mathrm dr\big\rangle  \]
		\les  \(\dbE_{t,t_0}\big[|\bar p(t_0)|^2  \big]\)^\frac12 \! \(\dbE_{t,t_0}\[\big(\int_{t_0}^{t_0+h}\! b\big(r, X_r^{t,x,\bar u} ,\bar{u}_r\big)\mathrm dr\big)^2   \] \)^\frac12\\
		%
		%
		\ns\ds\hskip1.3cm 
			\les  C h(1+ |  X^{t,x,\bar u}_{t_0}|  ) , 
		\ea$$
		and
		$$
		\ba{ll}
		\ns\ds  \dbE _{t,t_0}\big[{\rm III} \big]= C  \dbE _{t,t_0}\big[|  X^{t,x,\bar u}_{t_0+h}-  X^{t,x,\bar u}_{t_0} |^2  \big]  \\
		\ns\ds\les C \dbE _{t,t_0}\[ \(\int_{t_0}^{t_0+h}b\big(r, X_r^{t,x,\bar u} ,\bar{u}_r\big)\mathrm dr \)^2\]
		+ C \dbE _{t,t_0} \[ \(\int_{t_0}^{t_0+h}\si\big(r, X_r^{t,x,\bar u} ,\bar{u}_r\big)\mathrm dB_r   \)^2\]\hskip3.5cm
		\\
		\ns\ds\!\!\! \hskip0.55cm 
		+ C \dbE _{t,t_0} \[ \(\int_{t_0}^{t_0+h} \int_E \gamma \big(r, X_r^{t,x,\bar u} ,\bar{u}_r,e\big)\tilde N(\mathrm d r,\mathrm de)  \)^2\]\les Ch . 
		
		\\
		%
		%
		\ea$$
		Therefore,
		$$
		\ba{ll}
		\ns\ds \frac 1h\dbE_{t,t_0}\[ W\big(t_0+h,  X^{t,x,\bar u}_{t_0+h}\big)-W\big(t_0,  X^{t,x,\bar u}_{t_0} \big) \]  \les \frac 1h\dbE_{t,t_0} \[ {{\rm  I }+\rm{II}+\rm{ III}}\]
		 \les  C  (1+ |  X^{t,x,\bar u}_{t_0}|  )  .  
		\ea
		$$

		\no All the above constants $C$ can be different and do not depend on $t_0$.
		Furthermore, by taking the expectation on the both sides of the above inequality, we get \eqref{ineq-000}-(b).
		
		Taking expectation on the both sides of \eqref{ineq-111}, and applying Fatou's Lemma (needing \eqref{ineq-000}-(a)), we have
		$$
		\ba{ll}
		\ns\ds\!\!\!  \limsup\limits_{h\to 0^+}\frac 1h\dbE\Big[W\big(t_0+h, X^{t,x,\bar u} _{t_0+h}\big)-W\big(t_0,  X^{t,x,\bar u}_{t_0}\big) \Big]\\
		\ns\ds\!\!\! = \limsup\limits_{h\to 0^+}\frac 1h\dbE\[\dbE_{t,t_0}\big[W\big(t_0+h,  X^{t,x,\bar u}_{t_0+h}\big)-W\big(t_0,  X^{t,x,\bar u}_{t_0}\big) \big]\]\\
		\ns\ds\!\!\! \les \dbE\[\limsup\limits_{h\to 0^+}\frac 1h\dbE_{t,t_0}\big[W\big(t_0+h,  X^{t,x,\bar u}_{t_0+h}\big)-W\big(t_0,   X^{t,x,\bar u}_{t_0}\big) \big]\]\\
		\ns\ds\!\!\! =  \dbE\[  \bar q(t_0) +  \bar p(t_0).  b\big(t_0,  X^{t,x,\bar u}_{t_0}, \bar u_{t_0} \big) 
		+\frac12 \tr\big( \sigma\sigma^\top  (t_0,  X^{t,x,\bar u}_{t_0} , \bar{u}_{t_0} ) \bar P(t_0) \big)
		+\cB^{\bar u} \bar \f (t_0,  X^{t,x,\bar u}_{t_0})\] . \\
		%
		
		%
		%
		\ea
		$$
		%
		%
		
		Due to the set of such points $t_0$ being of full measure in $[t,T-\d]$, by  applying Lemma \ref{le22} (needing \eqref{ineq-000}-(b)), for any  $u(\cd) \in \cU_{t,T},$ we have
		$$
		\ba{ll}
		\ns\ds\!\!\!  \dbE\big[W\big(T-\d,  X^{t,x,\bar u}_{T-\d} \big)-W(t,x) \big] \\
		\ns\ds\!\!\! =  \int_t^ {T-\delta}\dbE\[\bar q(s) +  \bar p(s).  b\big(s,  X_s^{t,x,\bar u}, \bar u_s \big) 
		+\frac12 \tr\big( \sigma\sigma^\top  (s, X_s^{t,x,\bar u} , \bar{u}_s ) \bar P(s) \big) +\cB^{\bar u} \bar \f (s,  X_s^{t,x,\bar u})\] \mathrm ds. \\
		%
		%
		\ea
		$$
		According to Lebesgue  dominated convergence theorem and  the continuity properties of $W(\cd,\cd)$ and $X^{t,x;\bar u }_\cd $, letting $\d\to0$ in the above, we get
		$$
		\ba{ll}
		\ns\ds\!\!\!   \dbE\big[W\big(T,  X_T^{t,x,\bar u} \big)-W(t,x) \big]= \dbE\big[\Phi\big(  X_T ^{t,x,\bar u} \big)-W(t,x) \big]\\
		\ns\ds\!\!\!   \les \int_t^T \dbE\[\bar q(s) +  \bar p(s)  . b\big(s, X_s^{t,x,\bar u}  ,\bar{u}_s\big)  +
		\frac 12 \tr\big(\sigma\sigma^\top (s,  X_s ^{t,x,\bar u} ,\bar{u}_s)\bar P(s)\big) +\cB^{\bar u} \bar \f (s,  X_s^{t,x,\bar u} )\] \mathrm ds\\
		%
		\ns\ds\les - \dbE \[\int_t^T f\big(
		s,  X_s^{t,x,\bar u} ,  Y_s^{t,x,\bar u},\bar p(s). \sigma(s, X^{t,x;\bar u }_s  ,\bar u_s),  \cC^{\bar u} W(s,  X_s^{t,x,\bar u} ),\bar u_s\big) \mathrm ds\]\\
		\ns\ds\les - \dbE\[\int_t^T f\big( s,  X_s^{t,x,\bar u} ,   Y_s^{t,x,\bar u} ,  Z_s^{t,x,\bar u} , \int_E  l(e) V_s^{t,x,\bar u}(e) \nu(\mathrm de),
		\bar{u}_s\big) \mathrm ds - A_T^{t,x,\bar u}\],
		\ea$$
		where we have used the conditions (ii), (iii), (iv) and $A_\cd^{t,x,\bar u} \in \cA^2_{\dbF}(t,T;\dbR)$.

		Thus, for any $(t,x)\in[0,T]\times\dbR^n$,
		$$
		\ba{ll}
		\ns\ds W(t,x)\ges  \dbE\Big[\Phi\big(X^{t,x;\bar u }_T \big)  \!+\! \int_t^T  \! \! f\big(s,  X_s^{t,x,\bar u} ,  Y_s ^{t,x,\bar u} ,  Z_s^{t,x,\bar u} , \int_E  l(e) V_s^{t,x,\bar u}(e) \nu(\mathrm de),
		\bar{u}_s\big) \mathrm ds -  A_T^{t,x,\bar u} \Big]=J\big(t,x; \bar u(\cd) \big). \\
		\ea$$
		Combined with \eqref{ee1}, we get
		$
		W(t,x)= J\big(t,x; \bar u(\cd)  \big),
		$
		which means $\bar u(\cd) $ is an optimal control of  Problem (C)$_{t,x}$.
	\end{proof}

	\ms
	 Now, we shall construct the feedback optimal control of Problem (C)$_{t,x}$ from the viscosity solution of PIDE \eqref{HJB}.
	 \begin{lemma} \label{le5.3}\sl 
	 Assume {\bf (H1)} and {\bf (H2)} hold true. 
	 Then the value function $W(\cd,\cd)$ defined by \eqref{VF} is the only function in $\Theta$, which satisfies
	  the following: for any $(t,x)\in [0,T]\times \dbR^n$, $(q,p,P)\in D^{1,2,+}_{t+,x}W(t,x)$ and $\varphi \in \mathfrak{I} (q,p,P;W)(t,x) $
	 \begin{equation}\label{VIS-inequ}
	   \max\big\{W(t,x)- h(t,x),-q-\inf_{u\in U}\mathcal H \big(t,x,  W(t,x) ,p ,P ,\cB^{u} \varphi(t, x),\cC^{u} \varphi(t, x),u\big)  \big\}\les 0.
	 \end{equation}
	 \end{lemma}

	\begin{proof}
		From Lemma \ref{vis-solution}, we know the value function $W(\cd,\cd)\in \Theta$ is the  unique viscosity solution of \eqref{HJB}. Then, based on Lemma \ref{L11}, we can obtain some test function $\f \in  \mathfrak{I} (q,p,P;W)(t,x)$. From the Definition \ref{vis-def}-(i), we have
		 $$
		 	\max\big\{W(t,x)- h(t,x),-\f_t(t,x)-\inf_{u\in U}\mathcal H \big(t,x,  W(t,x) ,\f_x(t,x) ,\f_{xx}(t,x),\cB^{u} \varphi(t, x),\cC^{u} \varphi(t, x),u\big)  \big\}\les 0.
		 $$
		And thus, \eqref{VIS-inequ} holds ture. Further, the uniqueness comes from the uniqueness of the viscosity solution of  \eqref{HJB} in $\Theta$.
	\end{proof}

	\begin{theorem}\label{vis-svt-law}\sl Assume that the conditions {\rm \textbf{(H$_1$)-(H$_5$)}} and {\rm \bf(C)}  hold true.
		Let $W(\cd,\cd)\in C_p([0,T]\times\dbR^n)$ be the viscosity solution of   PIDE     \eqref{HJB}.
		Then, for each $(t,x)\in[0,T]\times \dbR^n$,
		\begin{equation*}\label{vis-1}
			\inf_{(q,p,P,\f,u)\in D^{1,2,+}_{t+,x} W(t,x)\times\mathfrak{I} (q,p,P;W)(t,x) \times U}\[ q+\mathcal H \big(t,x,  W(t,x) ,p ,P ,\cB^{u} \varphi(t, x),\cC^{u} \varphi(t, x),u\big) \]\ges W(t,x)-h(t,x).
		\end{equation*}
		Furthermore, if $\mathbbm{u}(\cdot,\cdot)\in\sU_{t,T}$ and for all $(t,x)\in[0,T]\times \dbR^n$, $ \mathbbm{q},$ $\mathbbm{p},$ $\mathbbm{P}$ are measurable functions
		satisfying $(\mathbbm{q}(t,x) ,\mathbbm{p}(t,x),\mathbbm{P}(t,x))\in D^{1,2,+}_{t+,x}W(t,x)$, and
		\begin{equation}\label{con-vis}
			 \left\{\ba{ll}
			 \ns\ds\!\!\! {\rm(i)}\     \mathbbm p(s,X _s^{t,x,\mathbbm u}) . \si\big(s,  X _s^{t,x,\mathbbm u} ,\mathbbm{u}(s,X _s^{t,x,\mathbbm u})\big)=  \mathbbm Z _s^{t,x,\mathbbm u} ,\ \mbox{a.e.} \ s\in [t,T],\   \dbP\mbox{-a.s.};\\
			 \ns\ds\!\!\!  {\rm(ii)}\ W \big(s, X_s^{t,x,\mathbbm u} +\gamma(s, X_s^{t,x,\mathbbm u},\mathbbm  u(s,X _s^{t,x,\mathbbm u}), e)\big)-W(s, X_s^{t,x,\mathbbm  u})
			 =   V_s^{t,x,\mathbbm  u}(e),\ e\in E; \\ 
			 \ns\ds\!\!\!  {\rm(iii)} \  \dbE   \[   \int_t^T  \!\!\! \Big( \mathbbm q(s, X_s^{t,x,\mathbbm u})  \!+\!    \mathcal H \big(s,  X_s^{t,x,\mathbbm u} ,  Y_s^{t,x,\mathbbm u} ,\Xi(s, X_s^{t,x,\mathbbm u}),
			 \cB^{\mathbbm u } \varphi(s,  X_s^{t,x,\mathbbm u} ),\cC^{\mathbbm u } \varphi(s,  X_s^{t,x,\bar u} ),
			 \mathbbm u (s, X_s^{t,x,\mathbbm u})\big)\Big) \mathrm ds   \]\\
			 \ns\ds\!\!\!\hskip 8.35cm    \les   0,\q\mbox{with some } \f(\cd,\cd) \in    \mathfrak{I} (\mathbbm q,\mathbbm p, \mathbbm P; W)\big(s,  X _s ^{t,x,\mathbbm u}  \big),
			 \ea\right.
		\end{equation}
		where $ \Xi (\cd,X^{t,x;\mathbbm{u}}_\cd)=\big (\mathbbm{p}(\cd,X^{t,x;\mathbbm{u}}(\cd)),\mathbbm{P}(\cd,X^{t,x;\mathbbm{u}}(\cd)) \big)$ and
		$ X^{t,x;\mathbbm{u} } $, $\big(Y^{t,x;\mathbbm{u} } ,Z^{t,x;\mathbbm{u} } ,V ^{t,x;\mathbbm{u} }(\cd),A^{t,x;\mathbbm{u} }  \big)$ satisfy   \eqref{SDEP-u} and  \eqref{RBSDEP-u} with $ \mathbbm{u} (\cd,X^{t,x;\mathbbm{u} }_\cd) \in \cU_{t,T} $, respectively.
		Then,     $ \mathbbm{u}(\cd,\cd)$ is an optimal feedback  control law   of Problem (C)$_{t,x}$.
	\end{theorem}

	\begin{proof}  
		The first result can be directly obtained from  the uniqueness of the viscosity solution of PIDE \eqref{HJB} and   Lemma \ref{le5.3}.

	   Subsequently, for any $(t,x)\in [0,T]\times \dbR^n$, set
		$$\ba{ll}
		\bar u(s):= \mathbbm{u}(s,X^{t,x;\mathbbm{u} }_s ), \q \bar q(s):= \mathbbm{q}(s,X^{t,x;\mathbbm{u} }_s ),\q \bar p(s):= \mathbbm{p}(s,X^{t,x;\mathbbm{u} }_s ),\q  \bar P(s) := \mathbbm{P}(s,X^{t,x; \mathbbm{u}}_s ),\q s\in[t,T],
		\ea$$
		then, they satisfy (i), (ii) and (iii) in Theorem \ref{SVT-VS}. So  $ \bar u(\cd)$ is an optimal control, i.e.,  $\mathbbm{u}(\cd,\cd)$ is an optimal feedback control law.
	\end{proof}

	Finally, we have a look at the procedures of finding the optimal feedback control law.
	From Theorem \ref{vis-svt-law}, we  can get the candidate of  optimal feedback control law by minimizing 
	$$
	q+\mathcal H \big(t,x,  W(t,x) ,p ,P ,\cB^{u} \varphi(t, x),\cC^{u} \varphi(t, x),u\big) 
	$$
	over $D^{1,2,+}_{t+,x} W(t,x)\times\mathfrak{I} (q,p,P;W)(t,x) \times U$.
	Further, to ensure the candidate to be the true optimal feedback control law,  we need to make sure
	the candidate $\mathbbm{u}(\cd,\cd)$  belongs to $\sU_{t,T}$ and \eqref{con-vis} is vaild under $(\mathbbm{q},\mathbbm{p},\mathbbm{P})\in D^{1,2,+}_{t+,x}W(t,x)$ and $\f \in \mathfrak{I} (\mathbbm{q},\mathbbm{p},\mathbbm{P};W)(t,x)$.

	\end{document}